\documentclass[12pt,a4paper]{amsart}

\usepackage[latin1]{inputenc}

\usepackage{amsfonts}
\usepackage{amsmath}
\usepackage{amssymb}
\usepackage{amsthm}
\usepackage{mathrsfs}
\usepackage{color}

\usepackage{float}
\restylefloat{table}

\usepackage{graphicx}
\usepackage[all]{xy}

\usepackage{enumerate}

\usepackage{tikz-cd}

\usepackage{hyperref}

\usepackage{a4wide}

\usepackage{bbm}

\newtheoremstyle{plain}
  {6pt}   
  {6pt}   
  {\itshape}  
  {0pt}       
  {\bfseries} 
  {.}         
  {5pt plus 1pt minus 1pt} 
  {}          

\newtheoremstyle{definition}
  {6pt}   
  {6pt}   
  {\normalfont}  
  {0pt}       
  {\bfseries} 
  {.}         
  {5pt plus 1pt minus 1pt} 
  {}          

\theoremstyle{plain}
\newtheorem*{thm*}{Theorem}
\newtheorem{thm}{Theorem}[section]
\newtheorem{prop}[thm]{Proposition}
\newtheorem{cor}[thm]{Corollary}
\newtheorem{lem}[thm]{Lemma}

\theoremstyle{definition}
\newtheorem{defn}[thm]{Definition}

\newtheorem{ex}[thm]{Example}
\newtheorem{rmk}[thm]{Remark}

\numberwithin{equation}{thm}

\newcommand{\emphbf}[1]{\emph{\textbf{#1}}}
\newcommand{\bcdot}{{\mbox{\boldmath{$\cdot$}}}}

\DeclareMathAlphabet{\mathpzc}{OT1}{pzc}{m}{it}

\newcommand{\f}{\varphi}
\newcommand{\la}{\lambda}

\DeclareMathOperator{\id}{id}
\DeclareMathOperator{\Hom}{Hom}
\DeclareMathOperator{\Ext}{Ext}

\DeclareMathOperator{\modu}{mod}

\DeclareMathOperator{\End}{End}

\DeclareMathOperator{\image}{Im}

\DeclareMathOperator{\ind}{ind}

\DeclareMathOperator{\projdim}{projdim}
\DeclareMathOperator{\op}{op}

\DeclareMathOperator{\comp}{comp}
\DeclareMathOperator{\Mor}{Mor}
\DeclareMathOperator{\cone}{cone}
\DeclareMathOperator{\coind}{coind}

\DeclareMathOperator{\Br}{Br}
\DeclareMathOperator{\Coh}{Coh}

\newcommand{\wt}[1]{\widetilde{#1}}
\newcommand{\wh}[1]{\widehat{#1}}

\begin{document}

\title{Ringel duality as an instance of Koszul duality}
\author{Agnieszka Bodzenta}
\author{Julian K\"ulshammer}
\date{\today}

\address{Agnieszka Bodzenta\\
Faculty of Mathematics, Informatics and Mechanics 
University of Warsaw \\ Banacha 2 \\ Warsaw 02-097,
Poland} \email{A.Bodzenta@mimuw.edu.pl}

\address{Julian K\"ulshammer\\
Institute of Algebra and Number Theory,
University of Stuttgart \\ Pfaffenwaldring 57 \\ 70569 Stuttgart,
Germany} \email{kuelsha@mathematik.uni-stuttgart.de}

\begin{abstract}
In \cite{KKO14}, S. Koenig, S. Ovsienko and the second author showed that every quasi-hereditary algebra is Morita equivalent to the right algebra, i.e. the opposite algebra of the left dual, of a coring. Let $A$ be an associative algebra and $V$ an $A$-coring whose right algebra $R$ is quasi-hereditary. In this paper, we give a combinatorial description of an associative algebra $B$ and a $B$-coring $W$ whose right algebra is the Ringel dual of $R$. We apply our results in small examples to obtain  restrictions on the $A_\infty$-structure of the $\Ext$-algebra of standard modules over a class of quasi-hereditary algebras related to birational morphisms of smooth surfaces.
\end{abstract}

\maketitle

\section{Introduction}

Exceptional collections appear frequently in algebraic and symplectic geometry as well as in representation theory. For example, they appear in the process of identifying certain Fukaya--Seidel categories, see e.g. \cite{Sei08}. In algebraic geometry, starting with the example of the projective space by A. Beilinson \cite{Bei78, Bei84}, (strong) exceptional collections are used to realise derived categories of coherent sheaves as derived categories of finite dimensional algebras \cite{Bon89}. In representation theory, the main source of examples for exceptional collections are quasi-hereditary algebras. In this situation, the exceptional collections consist of so-called standard modules. Examples of quasi-hereditary algebras are  Schur algebras and algebras of global dimension smaller than or equal to two. Also, blocks of BGG category $\mathcal{O}$ are equivalent to categories of modules over quasi-hereditary algebras. In work with S. Koenig and S. Ovsienko \cite{KKO14}, the second author showed that quasi-hereditary algebras  can equivalently be described as the right (or left) algebras of directed corings (also called bocses for `bimodule over category with coalgebra structure'). In this description, the category of modules filtered by standard modules is equivalent to the category of modules over the directed bocs, i.e. the Kleisli category of the corresponding comonad, see Definition \ref{def:directedbocs} for the definition of a directed bocs, Theorem \ref{maintheoremKKO} for the main theorem of \cite{KKO14} and  Proposition \ref{boxrepresentations} for a quiver theoretic perspective on the category of modules over a bocs. Let $\Lambda$ be a quasi-hereditary algebra and $\mathcal{D}$ the (graded) dual of the bar resolution of the $\Ext$-algebra of the standard modules over $\Lambda$, equipped with its canonical $A_\infty$-structure. The directed bocs $\mathfrak{A}$ of $\Lambda$ is obtained from the quotient of $\mathcal{D}$ by the differential ideal generated by the negative degree part. The inclusion $\modu \mathfrak{A}\to \modu \Lambda$ yields an equivalence $\mathcal{D}^b(\modu\mathfrak{A})\simeq \mathcal{D}^b(\modu\Lambda)$. 

An important concept in the theory of exceptional collections is the (left and right) mutation introduced by A. Bondal in \cite{Bon89}. For an exceptional collection $(\Delta_\mathtt{1},\dots,\Delta_\mathtt{n})$ in a triangulated category, its left mutation at $\mathtt{i}$ is the exceptional collection $(\Delta'_\mathtt{1},\dots,\Delta'_{\mathtt{n}})$ with $\Delta'_{\mathtt{l}}=\Delta_{\mathtt{l}}$ for all $\mathtt{l}\neq \mathtt{i},\mathtt{i}+1$ and $\Delta'_{\mathtt{i}+1}=\Delta_{\mathtt{i}}$. Mutations of exceptional collections induce an action of the braid group on the set of exceptional collections in a triangulated category. Given an exceptional collection, there are two distinguished other collections, the left and the right (Koszul) dual. They are obtained by mutating the exceptional collection along the `global half-twist' $\beta_{\mathtt{n}-1}(\beta_{\mathtt{n-2}}\beta_{\mathtt{n-1}})\cdots (\beta_{\mathtt{2}}\cdots \beta_{\mathtt{n-1}})(\beta_{\mathtt{1}}\cdots \beta_{\mathtt{n-1}})\in \Br_n$ or its inverse. 
For the collection of standard modules over a quasi-hereditary algebra $\Lambda$ elements of the right Koszul dual exceptional collection are again modules -- the costandard modules over $\Lambda$ denoted by $\nabla$. There is in fact another quasi-hereditary algebra, derived equivalent to $\Lambda$, having the exceptional collection of costandard modules as standard modules. It was introduced by C. Ringel in \cite{Rin91} and is therefore called the Ringel dual of the original quasi-hereditary algebra. In the language of directed bocses, the Ringel dual of the right algebra of a directed bocs is given by its left algebra. 

Let $(A,V)$ be a directed bocs. In \cite{Ovs06}, S. Ovsienko proposed a construction of a bocs $\mathfrak{B}=(B,W)$, whose right algebra is Morita equivalent to the left algebra of $(A,V)$ and vice versa. But his paper is lacking a lot of details in the proof of the construction. This paper provides full details of the proof of the following theorem. 

\begin{thm*}
Let $\Lambda$ be a quasi-hereditary algebra. Let $\mathfrak{A} = (A,V)$ be the corresponding directed bocs. Denote by $\overline{V}$ the kernel of the counit. Let $\mathcal{D}$ be the dual of the bar resolution of the differential graded algebra $T_A(\overline{V})$.
The quotient of $\mathcal{D}$ by the DG ideal generated by the negative degree part provides a combinatorial construction for a bocs $(B,W)$ having the Ringel dual of $\Lambda$ as the right algebra. 
\end{thm*}

Note that since the directed coring $(A,V)$ provides the same information as the $A_\infty$-structure on $\Ext^*_\Lambda(\Delta,\Delta)$, this gives a combinatorial method to obtain the $A_\infty$-structure on $\Ext^*_\Lambda(\nabla,\nabla)$ from the $A_\infty$-structure on $\Ext^*(\Delta,\Delta)$. This result should be compared with S. Oppermann's combinatorial construction in \cite{Opp17} of the silting mutation of a differential graded algebra. As noted in \cite[Appendix A.2]{KKO14}, the way to associate a directed bocs to a quasi-hereditary algebra is non-unique. As every algebra is isomorphic to the Ringel dual of its Ringel dual, the same holds true for Ringel duals. Uniqueness of directed bocses associated to quasi-hereditary algebras will be discussed in \cite{KM17}.

The proof of the above theorem is divided into several steps which we find of independent interest. First, we consider explicit complexes $\Box_{\mathtt{i}}^\bcdot$ of $\mathfrak{A}$-modules and prove that they are homotopically projective. 
Since in the bocs description, standard modules correspond to simple modules over the bocs, we conclude that the complexes $\Box_{\mathtt{i}}^\bcdot$ form an exceptional collection left dual to the exceptional collection of standard modules. (Recall that since $\Lambda$ is of finite global dimension, $\mathcal{D}^b(
\modu \Lambda)$ admits a Serre functor.) Therefore, $\Box_{\mathtt{i}}^\bcdot$ is the Serre dual of the costandard module $ \nabla_{\mathtt{i}}$.
If we denote by $\Diamond_{\mathtt{i}}$ the $\mathbbm{k}$-duals of the analogously defined $\mathfrak{A}^\textrm{op}$-modules then $\mathcal{F}(\Diamond) \simeq \mathcal{F}(\nabla)$. As Ringel duality yields an equivalence $\mathcal{F}(\nabla_{\Lambda})\simeq \mathcal{F}(\Delta_{R(\Lambda)})$, in order to find the dual bocs we present the category $\mathcal{F}(\Diamond)$ as the category of right modules over a bocs. To this end we prove (see Theorem \ref{FofBox} for the precise statement):
\begin{thm*}
	The category $\mathcal{F}(\Box)$ is equivalent to the category of complexes $N^\bcdot$ of $\mathfrak{A}$-modules such that $N^j \cong \overline{V}^{\otimes_A j} \otimes_{\mathbb{L}}Y$, for an $\mathbb{L}$-module $Y$. The differential $N^j \to N^{j+1}$ is encoded in a map $c_Y \colon Y \to \overline{V}\otimes_{\mathbb{L}}Y$. 
\end{thm*}
In the next step we further simplify the description of $\mathcal{F}(\Box)$ and with Theorem \ref{modN} show
\begin{thm*}
	The category $\mathcal{F}(\Box)$ is equivalent to a category $\mathcal{N}(\mathfrak{A})$ whose objects are pairs $(Y, c_Y)$ of an $\mathbb{L}$-module and a map $c_Y \colon Y \to  \overline{V}\otimes_{\mathbb{L}}Y$ satisfying some additional condition.
\end{thm*}
Finally, we show that $\mathcal{N}(\mathfrak{A})$ is equivalent to the category $\mathcal{R}(\mathfrak{A})$ whose objects are $\mathbb{L}$-modules $Y$ together with a map $s_Y\in \Hom_{\mathbb{L}\otimes \mathbb{L}}(\mathbb{D} \overline{V}, \Hom_\mathbbm{k}(Y,Y))$ while morphisms $(Y,s_Y) \to (Z,s_Z)$ are represented by elements of $\Hom_{\mathbb{L} \otimes \mathbb{L}}(\mathbb{D}A, \Hom_{\mathbbm{k}}(Y,Z))$. This, together with the quiver theoretic description of $\modu \Lambda$ in Proposiotion \ref{boxrepresentations}, suggests that $\mathcal{R}(\mathfrak{A})$ is equivalent to the category of a bocs $(B,W)$ where $B$ is an algebra generated by $\mathbb{D}\overline{V}$ and $W$ is a bimodule generated by $\mathbb{D}A$. In other words, the bocs $(B,W)$ is the quotient of the dual of the bar resolution of the bocs $(A,V)$ by the differential ideal generated by negative degrees. Hence, from the point of view of bocses, Ringel duality is a special case of Koszul duality for DG algebras.

For a quasi-hereditary algebra $\Lambda$ the $t$-structure on $\mathcal{D}^b(\modu \Lambda)$ glued along the filtration $\langle \Delta(\mathtt{1}) \rangle \subset \langle \Delta(\mathtt{1}), \Delta(\mathtt{2}) \rangle \subset \ldots \subset \langle \Delta(\mathtt{1}), \ldots, \Delta(\mathtt{n-1}) \rangle \subset \mathbb{D}^b(\modu \Lambda)$ from the standard $t$-structures on $\langle \Delta(\mathtt{i}) \rangle \simeq \mathcal{D}^b(\modu \mathbbm{k})$ is the standard $t$-structure, while the $t$-structure glued along a similar filtration for the right dual collection $\langle \nabla(\mathtt{n}) \rangle \subset \langle \nabla(\mathtt{n}), \nabla(\mathtt{n-1}) \rangle \subset \ldots \subset \langle \nabla(\mathtt{n}), \ldots, \nabla(\mathtt{2}) \rangle \subset \mathcal{D}^b(\modu \Lambda)$ is the $t$-structure where the characteristic tilting $\Lambda$-module is the projective generator for the heart, c.f. \cite{BodBon17}. In other words, Ringel duality can be viewed as passing to the right dual exceptional collection. With our main theorem we show that it yields Koszul duality on the level of bocses.	
	
The idea that passing to the dual exceptional collection corresponds to Koszul duality goes back to A. Bondal who proved in \cite{Bon89} that the full exceptional collection of simple modules over a path algebra of a directed quiver is right dual to the full exceptional collection of projective modules. The relation between Koszul duality and dual exceptional collections was further studied by A. Beilinson, D. Ginzburg and V. Schechtman in \cite{BGS88} for so-called mixed DG algebras. 
Also in the case of quasi-hereditary algebras, a relation between Koszul duality and Ringel duality (and also Serre duality) has been observed in a particular instance, namely that of strict polynomial functors (or equivalently Schur algebras of symmetric groups $S(n,d)$ for $n\geq d$). In this special case, it is more closely related to classical Koszul duality between the exterior algebra and the polynomial ring: (derived) tensoring with the exterior power $\Lambda^d$ induces an autoequivalence of the derived category of strict polynomial functors whose square, (derived) tensoring with symmetric power gives a Serre functor on this derived category. This is work by M. Cha\l upnik \cite{Cha08} and A. Touz\'e \cite{Tou13}, see also the summary by H. Krause, \cite{Kra13}.

The strategy of our proof for the part on Ringel duality follows \cite{Ovs06}. As already remarked, in his paper, the proofs of all statements except for the well-definedness in Theorem \ref{modN} are sketches. Here we provide explicit calculations. Furthermore,  in the proofs of Theorem \ref{FofBox} and the equivalence of \ref{modN} we chose to incline to explicit calculations and the known theory of quasi-hereditary algebras instead of referring to a yet to be developed derived homological algebra for bocses.

As mentioned earlier, in algebraic geometry, exceptional collections are used to realise derived categories of coherent sheaves on projective varieties as derived categories of finite dimensional algebras. In \cite{Bon89}, A. Bondal proved that full strong exceptional collections yield such derived equivalences. Here an exceptional collection is called \emphbf{full} if it classically generates the derived category and \emphbf{strong} if $\Hom(\Delta_\mathtt{i}, \Delta_\mathtt{l} [s]) = 0$ for all $\mathtt{i}$ and $\mathtt{l}$ and every $s \neq 0$. The next step is considering full \emphbf{almost strong} exceptional collections, i.e. collections where $\Hom(\Delta_\mathtt{i},\Delta_\mathtt{l}[s])=0$ for all $\mathtt{i}$ and $\mathtt{l}$ and every $s\neq 0,1$. In representation theory, quasi-hereditary algebras with such a set of standard modules are called \emphbf{left strongly quasi-hereditary}. They appear in O. Iyama's proof of finiteness of representation dimension, see \cite{Iya03, Rin10}. 

In algebraic geometry, examples of almost strong exceptional collections are given by exceptional collections of line bundles on smooth rational surfaces. Recall that every smooth rational surface $X$ (except for $\mathbb{P}^2$ which has a strong exceptional collection by the work of A. Be\u\i linson \cite{Bei78,Bei84}) is obtained from a Hirzebruch surface $\mathbb{F}_a$ by a sequence of blow-ups, i.e. there exists a birational morphism $f\colon X \to \mathbb{F}_a$. It is well-known that the derived category of a Hirzebruch surface admits a full strong exceptional collection $(F_\mathtt{1},\ldots,F_\mathtt{4})$, see e.g. \cite{HP11}. In \cite{HP11, HP14}, L. Hille and M. Perling showed how to obtain a full almost strong exceptional in $\mathcal{D}^b(\Coh X)$ from such a full strong exceptional collection on $\mathbb{F}_a$. In \cite{Bod13}, the first author proved that this exceptional collection on $X$ can be mutated to a collection $(\Delta_\mathtt{1}, \ldots, \Delta_{\mathtt{n-4}}, f^*(F_\mathtt{1}), \ldots, f^*(F_\mathtt{4}))$. The collection $(\Delta_\mathtt{1},\ldots,\Delta_{\mathtt{n-4}})$ is almost strong and the full subcategory of $\mathcal{D}^b(\Coh X)$ generated by $(\Delta_\mathtt{1},\ldots, \Delta_{\mathtt{n-4}})$ is equivalent to the derived category of modules over a quasi-hereditary algebra $\Lambda_f$. The objects $\Delta_\mathtt{1},\ldots, \Delta_{\mathtt{n-4}}$ correspond to the standard modules over $\Lambda_f$ and the dimension of both $\Hom_X(\Delta_\mathtt{i}, \Delta_\mathtt{l})$ and $\Ext^1_X(\Delta_\mathtt{i}, \Delta_\mathtt{l})$ is at most one for any pair $(\mathtt{i},\mathtt{l})$. Finally, the category of $\Lambda_f$-modules admits a duality which preserves simple modules. As the morphism $f$ is a contraction of a curve, we call algebras satisfying the above properties \emphbf{curve-like}. 

As an application of the main result of our paper, we illustrate how to obtain restrictions on the possible $A_\infty$-structures on the Ext-algebras of standard modules over curve-like algebras. We use them to classify all possible curve-like algebras with up to 4 simple modules.
\begin{thm*}
	There is one curve-like algebra with two simple modules, there are three curve-like algebras with three simple modules and thirteen curve-like algebras with four simple modules. All algebras with two and three simple modules are Morita equivalent to an algebra $\Lambda_f$ or its Ringel dual for a birational morphism $f$ of smooth surfaces. There are four curve-like algebras with four simple modules which are not Morita equivalent to $\Lambda_f$ or its Ringel dual for any $f$. In the classification of Section \ref{ssec_four_simple} these are algebras A1, B1, B2 and G1.  
\end{thm*}

The paper is structured as follows. In Section \ref{notation} we fix our notation. In Section \ref{quasihereditaryalgebras} we recall the necessary background on quasi-hereditary algebras. In particular, we identify the costandard modules in the derived category. Section \ref{directedbocses} is devoted to recalling the main results of \cite{KKO14} describing quasi-hereditary algebras in terms of directed bocses. In Section \ref{sec_homotopically_projective} we describe the Serre duals of $\nabla_\mathtt{i}$ as objects of the derived category of the module category of the bocs. In Section \ref{sec_cat_F_Box} we describe the category of modules filtered by costandard modules in terms of a ``category of comodules'' $\mathcal{N}(\mathfrak{A})$.  In Section \ref{dualising} we use a standard isomorphism to translate $\mathcal{N}(\mathfrak{A})$ into the category $\mathcal{R}(\mathfrak{A})$ closer to quiver representations. Section \ref{ringeldualbocs} defines the bocs corresponding to the Ringel dual of a quasi-hereditary algebra given the datum of the bocs of a quasi-hereditary algebra. Finally, in Section \ref{applications} we apply our results to obtain restrictions on the $A_\infty$-algebra structures on $\Ext$-algebras of standard modules over curve-like algebras.

\textbf{Acknowledgement} We would like to thank Steffen K{\"o}nig and the anonymous referee for many useful remarks. The first named author was partially supported by the EPSRC grant EP/K021400/1.

\section{Notation}\label{notation}

We work over an algebraically closed field $\mathbbm{k}$. We consider $\mathbbm{k}$-algebras which are always assumed to be unital, associative, and finite dimensional. For an algebra $\Lambda$, we denote by $\modu \Lambda$ the category of finite dimensional left $\Lambda$-modules. We consider the duality functor $\mathbb{D} = \Hom_{\mathbbm{k}}(-, \mathbbm{k}) \colon (\modu \Lambda)^{\textrm{op}} \to \modu (\Lambda^{\textrm{op}})$ which maps left $\Lambda$-modules to right ones. If $\Lambda$ is a basic algebra with $\Lambda/\textrm{rad} \Lambda \cong \mathbbm{k}^{\oplus n}$, the set $\{\mathtt{1},\ldots, \mathtt{n}\}$ parametrizes distinct maximal left ideals in $\Lambda$, i.e. pairwise non-isomorphic simple $\Lambda$-modules. For $\mathtt{i}\in \{\mathtt{1},\ldots, \mathtt{n}\}$ we denote by $L(\mathtt{i})$ the corresponding simple module and by $P(\mathtt{i})$ its irreducible projective cover. We shall write $\mathtt{i}$, $\mathtt{l}$ and $\mathtt{m}$ for elements of the set $\{\mathtt{1}, \ldots, \mathtt{n}\}$. 
We fix pairwise orthogonal idempotents $\{e_\mathtt{1},\ldots, e_{\mathtt{n}}\}$ in $\Lambda$. This choice yields an isomorphism $P(\mathtt{i}) \cong A e_{\mathtt{i}}$ and a homomorphism of algebras $\iota \colon \mathbb{L} \to \Lambda$, for the semi-simple algebra $\mathbb{L}:= \prod_{\mathtt{m}=\mathtt{1}}^{\mathtt{n}} \mathbbm{k}$. The map $\iota$ allows us to consider any $\Lambda$-module as an $\mathbb{L}$-module. 

A quiver $Q=(Q_0,Q_1,s,t)$ consists of a set of vertices $Q_0$, a set of arrows $Q_1$ and two functions $s,t\colon Q_1\to Q_0$ giving the source and the target of an arrow, respectively. The bocses corresponding to left strongly quasi-hereditary algebras can be described using differential biquivers. In this case, the arrows of $Q$ are graded. We write $Q_1^j$ for the component of degree $j$ in $Q_1$.

For category $\mathcal{A}$ which is either abelian or triangulated and a set of objects $\Theta\subset \mathcal{A}$ we denote by $\mathcal{F}(\Theta)$ the category of \emphbf{$\Theta$-filtered objects}, e.g. $\mathcal{A} = \modu \Lambda$ for a quasi-hereditary algebra $\Lambda$ and $\Theta= \Delta$ or $\Theta = \nabla$ . If $\mathcal{A}$ is triangulated, $\mathcal{F}(\Theta)$ is the smallest \emphbf{extension closed} subcategory of $\mathcal{A}$ containing $\Theta$, i.e. the smallest subcategory $\mathcal{A}'$ such that $\Theta \subset \mathcal{A}'$ and, for any distinguished triangle $X \to Y \to Z \to X[1]$ with $X$ and $Z$ in $\mathcal{A}'$ the object $Y$ also belongs to $\mathcal{A}'$. In the case when $\mathcal{A}$ is abelian, the subcategory $\mathcal{F}(\Theta)\subset /mathcal{D}^b(\mathcal{A})$ is an exact category and can be described as the full subcategory of $\mathcal{A}$ whose objects admit a finite filtration with graded factors isomorphic to objects of $\Theta$. If the objects in $\Theta$ are indecomposable, $\mathcal{F}(\Theta)$ is idempotent complete.

In general, for an idempotent complete exact category $\mathcal{E}$ we denote by $\mathcal{C}^b(\mathcal{E})$ the category of bounded complexes in $\mathcal{E}$, by $\mathcal{K}^b(\mathcal{E})$ its bounded homotopy category, and by $\mathcal{D}^b(\mathcal{E})$ its bounded derived category, which exists by work of Thomason and Trobaugh \cite{TT90}, see also \cite{Nee90}. The shift functor in $\mathcal{C}^b(\mathcal{E})$, $\mathcal{K}^b(\mathcal{E})$ as well as in $\mathcal{D}^b(\mathcal{E})$ is denoted by $[1]$.

For a bocs $\mathfrak{A} = (A, V)$ as defined in Definition \ref{def:directedbocs} we consider the following $A$-bilinear maps:
\begin{itemize}
	\item $m_A\colon A\otimes_\mathbb{L} A\to A$, the multiplication map
	\item $m_l\colon A\otimes_\mathbb{L} \overline{V}\to \overline{V}$, the defining map for the left module structure,
	\item $m_r\colon \overline{V}\otimes_\mathbb{L} A\to \overline{V}$, the definining map for the right module structure,
	\item $m_L\colon A\otimes_{\mathbb{L}} (\overline{V}\otimes_A\overline{V})\to \overline{V}\otimes_A\overline{V}$, the left module structure map,
	\item $m_R\colon (\overline{V}\otimes_A \overline{V})\otimes_{\mathbb{L}} A\to \overline{V}\otimes_A \overline{V}$, the right module structure map
	\item $m_{\overline{V}}\colon \overline{V}\otimes_\mathbb{L} \overline{V}\to \overline{V}\otimes_A \overline{V}$, the natural projection. 
\end{itemize}

\section{Quasi-hereditary algebras}\label{quasihereditaryalgebras}

Quasi-hereditary algebras were introduced by L. Scott in \cite{Sco87}, see also \cite{CPS88} for the more general notion of a highest weight category which allows infinitely many simple objects. Their distinguished feature is the existence of certain modules $\Delta_{\mathtt{i}}$ which are quotients of indecomposable projectives and project onto the simple modules. Dually, there also exist modules $\nabla_{\mathtt{i}}$ which have simple socle and embedd into the indecomposable injectives. In \cite{Rin91}, C. Ringel proved that the category of modules filtered by such modules has Auslander--Reiten sequences. He also introduced an algebra, now called the Ringel dual, such that the category of modules filtered by standard modules for the Ringel dual is equivalent to the category of modules filtered by costandard modules. Important examples of quasi-hereditary algebras include blocks of BGG category $\mathcal{O}$ associated to a complex semisimple Lie algebra (\cite{BGG76}), Schur algebras of symmetric groups (\cite{Don81, Par89}), and algebras of global dimension at most two (\cite{DR89}). For further information on quasi-hereditary algebras, see the excellent survey articles by Dlab--Ringel \cite{DR92} and Klucznik--Koenig \cite{KK99}.

\begin{defn}[{\cite[Theorem 1]{DR92}}]
Let $\Lambda$ be an algebra with isomorphism classes of simple modules indexed by $\{\mathtt{1},\dots,\mathtt{n}\}$. For $\mathtt{i}\in \{\mathtt{1},\dots,\mathtt{n}\}$ the module 
\[\Delta(\mathtt{i}):=P(\mathtt{i})/\left(\sum_{\substack{f\colon P(\mathtt{l})\to P(\mathtt{i})\\\mathtt{l}> \mathtt{i}}}\image f\right)\]
is called the \emphbf{standard module} associated to $\mathtt{i}$.  The algebra $\Lambda$ is called \emphbf{quasi-hereditary} if $\End_{\Lambda}(\Delta(\mathtt{i}))\cong \mathbbm{k}$ for all $\mathtt{i}\in \{\mathtt{1},\dots,\mathtt{n}\}$ and $\Lambda\in \mathcal{F}(\Delta)$.
\end{defn}

\begin{rmk}
Dually, using the injective modules $I(\mathtt{i})$ with socle $L(\mathtt{i})$, one can define the \emphbf{costandard modules} $\nabla(\mathtt{i})$. For an algebra $\Lambda$ to be quasi-hereditary is then equivalent to $\End_\Lambda(\nabla(\mathtt{i}))\cong \mathbbm{k}$ for all $\mathtt{i}\in \{\mathtt{1},\dots,\mathtt{n}\}$ and $\Lambda\in \mathcal{F}(\nabla)$. 
\end{rmk}

Dlab--Ringel's standardisation theorem states that every set of objects in an abelian category which behaves like the set of standard modules for a quasi-hereditary algebra (i.e. forms an exceptional collection in the abelian category), actually comes from a quasi-hereditary algebra:

\begin{thm}[Dlab--Ringel standardisation theorem, {\cite[Theorem 2]{DR92}}]
	\label{dlabringel}
Let $\mathcal{C}$ be an abelian category. Let $\{\Theta(\mathtt{i})\,|\,\mathtt{i}=\mathtt{1},\dots,\mathtt{n}\}$ be a standardisable set of objects, i.e. a set of objects in $\mathcal{C}$ satisfying
\begin{enumerate}
\item[(F)] $\dim_{\mathbbm{k}} \Hom_\mathcal{C}(\Theta(\mathtt{i}),\Theta(\mathtt{l}))<\infty$, $\dim_{\mathbbm{k}} \Ext^1_\mathcal{C}(\Theta(\mathtt{i}),\Theta(\mathtt{l}))<\infty$,
\item[(D)] $\End_\mathcal{C}(\Theta(\mathtt{i}))\cong \mathbbm{k}$, $\Hom_\mathcal{C}(\Theta(\mathtt{i}),\Theta(\mathtt{l}))\neq 0\Rightarrow \mathtt{i}\leq \mathtt{l}$ and $\Ext^1_\mathcal{C}(\Theta(\mathtt{i}),\Theta(\mathtt{l}))\neq 0\Rightarrow \mathtt{i}< \mathtt{l}$. 
\end{enumerate}
Then, there exists a quasi-hereditary algebra $\Gamma$, unique up to Morita equivalence, such that the  categories $\mathcal{F}(\Theta)$ and $\mathcal{F}(\Delta_\Gamma)$ are equivalent.
\end{thm}

The costandard modules form a standardisable set in $\modu \Lambda$ with respect to the opposite ordering of $\{\mathtt{1},\dots,\mathtt{n}\}$.

As already mentioned in Section \ref{notation} the categories $\mathcal{F}(\Delta)$ and $\mathcal{F}(\nabla)$ are idempotent split exact categories.
Recall that an object $T$ in an exact category $\mathcal{E}$ is an \emphbf{$\Ext$-projective generator for} $\mathcal{E}$ if, for any $X \in \mathcal{E}$, there exists an admissible epimorphism $T^{\oplus j} \to X$ in $\mathcal{E}$, for some $j$, while the group $\Ext^1_{\mathcal{E}}(T,X)$ vanishes.

\begin{defn}
Let $\Lambda$ be a quasi-hereditary algebra. The (unique up to multiplicities of direct summands) $\Ext$-projective generator $T$ of $\mathcal{F}(\nabla)$ is called the \emphbf{characteristic tilting module} of $\Lambda$. The opposite of the endomorphism algebra of the characteristic tilting module is called the \emphbf{Ringel dual} of $\Lambda$. It is Morita equivalent to the algebra $\Gamma$ obtained by applying Dlab--Ringel standardisation theorem to the standardisable set of costandard modules.
\end{defn}

\begin{rmk}
As proven by Ringel in \cite[Theorem 5]{Rin91}, the module $T$ is indeed a \emphbf{tilting} module in the sense of Miyashita \cite{Miy86}, i.e. 
\begin{enumerate}[{(T}1{)}]
\item $\projdim T<\infty$,
\item $\Ext^j_\Lambda(T,T)=0$ for all $j\neq 0$,
\item $T$ has $\mathtt{n}$ indecomposable direct summands up to isomorphism.
\end{enumerate}
This implies in particular that a quasi-hereditary algebra $\Lambda$ and its Ringel dual $\Gamma$ are derived equivalent.
\end{rmk}

\begin{lem}[{\cite[Lemma 7.1]{MS16}}]
	Let $\Lambda$ be a quasi-hereditary algebra. Then 
	\[\mathcal{D}^b(\mathcal{F}(\Delta))\simeq \mathcal{D}^b(\modu \Lambda)\simeq \mathcal{D}^b(\mathcal{F}(\nabla)).\]  
\end{lem}

The costandard modules can be identified in $\mathcal{D}^b(\modu \Lambda)$ by a certain orthogonality property with respect to the standard modules. In Section \ref{sec_cat_F_Box}, we use this to identify the objects corresponding to the costandard modules in yet another description of the derived category of $\Lambda$.

\begin{lem}\label{costandards}
Let $\Lambda$ be a quasi-hereditary algebra, $\mathtt{i}\in \{\mathtt{1},\dots, \mathtt{n}\}$, and $M$ an object of $\mathcal{D}^b(\modu \Lambda)$ such that for all $\mathtt{i}\in \{\mathtt{1},\dots,\mathtt{n}\}$, $\Hom_\Lambda(\Delta(\mathtt{l}),M[s])\cong \begin{cases}\mathbbm{k}&\text{if $s=0, \mathtt{l}=\mathtt{i}$,}\\0&\text{else.}\end{cases}$ Then $M\cong \nabla(\mathtt{i})$.
\end{lem}

\begin{proof}
First of all note that $\Ext^s_\Lambda(\Delta(\mathtt{l}),M)=0$ for $s\neq 0$ implies that $\Ext^s_\Lambda(\Lambda,M)=0$ since $\Lambda$ is filtered by $\Delta$'s. This implies that $M$ is indeed a module. 

The condition $\Ext^1_\Lambda(\Delta(\mathtt{l}),M)= 0$ for all $\mathtt{l}$ is equivalent to $M\in \mathcal{F}(\nabla)$ by \cite[Proposition 2.1]{KK99}. In this case $\dim\Hom_{\Lambda}(\Delta(\mathtt{l}),M)=[M:\nabla(\mathtt{l})]$ counts the multiplicity of $\nabla(\mathtt{s})$ in $M$ in any given $\nabla$-filtration of $M$, see e.g. \cite[Lemma 2.4]{DR92} for the dual statement. It follows that $M\cong \nabla(\mathtt{i})$. 
\end{proof}

Note that the above lemma shows that the full exceptional collection $\langle \nabla(\mathtt{i}) \rangle$ of costandard modules is right dual to the full exceptional collection $\langle \Delta(\mathtt{i}) \rangle$ of standard modules, cf. \cite[Lemma 2.5]{BS10}.

\section{Directed bocses}\label{directedbocses}

In this section we recall from \cite{KKO14} the alternative definition of quasi-hereditary algebras via bocses. Bocses were introduced by Roiter in \cite{Roi79, Roi80}. This concept is closely related to the concept of a differential biquiver, which was studied already in \cite{KR75} by A. Kleiner and M. Roiter, see Theorem \ref{bocsbiquiver} which was generalised by T. Brzezi\'nski in \cite{Brz13}. The most striking application of the theory of bocses is certainly Yu. Drozd's tame and wild dichotomy theorem \cite{D80} (see also \cite{CB88}). Bocses are sometimes also called corings. For further reading, we refer to the survey article \cite{K16}, for general theory of bocses or corings to \cite{BW03}, \cite{BSZ09}, and \cite{Bur05}.

\begin{defn}\label{def:directedbocs}
\begin{enumerate}[(i)]
\item A \emphbf{prebocs} is a tuple $\mathfrak{A}=(A,V,\mu)$ where $A$ is an algebra, $V$ is an $A$-bimodule and $\mu\colon V\to V\otimes_A V$ is a  coassociative morphism of $A$-$A$-bimodules called the \emphbf{comultiplication}. 
\item A \emphbf{bocs} $\mathfrak{A}=(A,V,\mu,\varepsilon)$ is a prebocs $(A,V,\mu)$ together with an $A$-$A$-bilinear map $\varepsilon\colon V\to A$ satisfying the usual \emphbf{counit} axiom. In this case $V$ is also called an \emphbf{$A$-coring}.
\item A bocs is \emphbf{normal} if there is a \emphbf{grouplike} element $\omega\in V$ i.e. an element such that $\mu(\omega)=\omega\otimes \omega$ and $\varepsilon(\omega)=1$. 
\item\label{directedbocs:v} A bocs is \emphbf{directed} if the counit is surjective, the Gabriel quiver of $A$ is directed, and $\overline{V}:=\ker \varepsilon\cong \bigoplus Ae_\mathtt{l}\otimes_{\mathbbm{k}} e_\mathtt{i}A $ where the sum runs over certain $\mathtt{i},\mathtt{l}$ all satisfying $\mathtt{i}< \mathtt{l}$.
\item For a bocs $\mathfrak{A}=(A,V)$ the algebra $R:=R_\mathfrak{A}:=\Hom_A(V,A)$ with the multiplication $g\circ f$ of $f,g\in \Hom_{A}(V,A)$ given by the composition of the maps
\[
\begin{tikzcd}
V\arrow{r}{\mu} &V\otimes_A V\arrow{r}{1\otimes g} &V\otimes_A A\arrow{r}{\sim} &V\arrow{r}{f} &A
\end{tikzcd}
\]
is the \emphbf{right algebra} of $(A,V)$. 
\item Dually, the \emphbf{left algebra} of $(A,V)$ is the algebra $L:=L_\mathfrak{A}:=\Hom_{A^{\op}}(V,A)$ with the multiplication $g\circ f$ of two morphisms $f,g\in \Hom_{A^{\op}}(V,A)$ given by the composition of the maps
\[\begin{tikzcd}
V\arrow{r}{\mu} &V\otimes_A V\arrow{r}{f\otimes 1} &A\otimes_A V\arrow{r}{\sim} &V\arrow{r}{g}&A.
\end{tikzcd}
\] 
\end{enumerate}
\end{defn}

\begin{rmk}
\begin{enumerate}[(i)]
\item Over an algebraically closed field, any projective indecomposable $A$-bimodule is of the form $Ae_\mathtt{l} \otimes_{\mathbbm{k}} e_\mathtt{i} A$, for some pair $(\mathtt{i},\mathtt{l})$. Hence, condition \eqref{directedbocs:v} states that the kernel of $\varepsilon$ is projective and, for any direct summand $Ae_\mathtt{l} \otimes_{\mathbbm{k}} e_\mathtt{i} A$ of $\ker \varepsilon$, we have $\mathtt{i}<\mathtt{l}$.
\item In \cite{BW03} and other literature on corings, the opposite algebra of the right algebra is called the left dual of the coring, the opposite algebra of the left algebra is called the right dual of the coring. Our notation and terminology is adopted from \cite{BB91}.
\end{enumerate}
\end{rmk}

By Morita equivalence, one can always assume that the underlying algebra $A$ is basic. Hence, up to a choice of orthogonal idempotents $e_{\mathtt{i}}$, it can be regarded as a category
with objects $\mathtt{1},\dots,\mathtt{n}$.
In this context, multiplying a grouplike $\omega$ from the left and from the right with $e_\mathtt{i}$ we obtain elements $\omega_\mathtt{i}:=e_\mathtt{i}\omega e_\mathtt{i}$ with $\varepsilon(\omega_\mathtt{i})=e_\mathtt{i}$. 

For a bocs $\mathfrak{A}$, one can construct its category of representations. It is a concrete description of the more abstractly defined Kleisli category of the comonad $V\otimes_A -$.

\begin{defn}
Let $\mathfrak{A}=(A,V,\mu,\varepsilon)$ be a bocs. The category $\modu \mathfrak{A}$ of representations of $\mathfrak{A}$ is defined via:
\begin{description}
\item[objects] are $A$-modules, 
\item[morphisms] for $M,N\in \modu \mathfrak{A}$, $\Hom_{\mathfrak{A}}(M,N)=\Hom_{A\otimes A^{\op}}(V,\Hom_{\mathbbm{k}}(M,N))$,
\item[composition] for $f\in \Hom_{\mathfrak{A}}(L,M)$ and $g\in \Hom_{\mathfrak{A}}(M,N)$ their composition $g\circ f$ is given by the following composition of $A$-bilinear maps:
\[
\begin{tikzcd}
V\arrow{r}{\mu} &V\otimes_A V\arrow{r}{g\otimes f} &\Hom_{\mathbbm{k}}(M,N)\otimes_A \Hom_{\mathbbm{k}}(L,M)\arrow{r}{\comp} &\Hom_{\mathbbm{k}}(L,N),
\end{tikzcd}
\]
\item[unit] the morphism $\mathbbm{1}_M\in \Hom_{\mathfrak{A}}(M,M)$ is given by the composition of
\[
\begin{tikzcd}
V\arrow{r}{\varepsilon} &A\arrow{r}{\lambda} &\Hom_{\mathbbm{k}}(M,M)
\end{tikzcd}\]
where $\lambda$ is the morphism sending an element $a\in A$ to left multiplication with $a$. 
\end{description}
\end{defn}

Note that each bocs $\mathfrak{A}=(A,V,\mu,\varepsilon)$ has an \emphbf{opposite bocs} $\mathfrak{A}^{\op}=(A^{\op},V^{\op},\mu,\varepsilon)$ where $A^{\op}$ is the opposite algebra of $A$, $V^{\op}=V$, but regarded as  an $A^{\op}$-$A^{\op}$-bimodule instead of an $A$-$A$-bimodule and the comultiplication and counit remain unchanged. 

As in the case of algebras, the categories of $\mathfrak{A}$-modules and $\mathfrak{A}^{\op}$-modules are dual to each other:

\begin{lem}\label{lem_duality}
Let $\mathfrak{A}$ be a bocs, $\mathfrak{A}^{\op}$ its opposite bocs. Then, the $\mathbbm{k}$-duality $\mathbb{D}=\Hom_{\mathbbm{k}}(-,\mathbbm{k})$ induces a duality $\modu \mathfrak{A}\to \modu\mathfrak{A}^{\op}$. 
\end{lem}

\begin{proof}
Define $\mathbb{D}M:=\Hom_{\mathbbm{k}}(M,\mathbbm{k})$ for an $\mathfrak{A}$-module $M$. This defines $\mathbb{D}$ on objects. To define it on morphisms recall the bimodule action of $A$ on $\Hom_\mathbbm{k}({}_A M,{}_A N)$ as well as $\Hom_{\mathbbm{k}}({}_{A^{\op}}\mathbb{D}N,{}_{A^{\op}}\mathbb{D}M)$. For $f\in \Hom_{\mathbbm{k}}(M,N)$, $x\in M$, $a,b\in A$, the bimodule action is defined by $(afb)(x):=a\cdot f(bx)$. Dually, for $\varphi\in \Hom_{\mathbbm{k}}(\mathbb{D}N,\mathbb{D}M)$, $\chi\in \mathbb{D}N$, the bimodule action is defined by $(a\varphi b)(\chi)=\varphi(\chi a)\cdot b$. We claim that $\Hom_{\mathbbm{k}}(M,N)\cong \Hom_{\mathbbm{k}}(\mathbb{D}N,\mathbb{D}M)$ as $A$-$A$-bimodules. Noting that $(\mathbb{D}f)(\chi)=\chi\circ f$, this follows from
\begin{align*}
(a(\mathbb{D}f)b)(\chi)(x)&=((\mathbb{D}f)(\chi a)\cdot b)(x)=(\mathbb{D}f)(\chi a)(bx)
=(\chi a)(f(bx))\\
&=\chi(a f(bx))
=(\chi \circ (afb))(x)
=\mathbb{D}(afb)(\chi)(x).
\end{align*}
It follows that $\Hom_{A\otimes A^{op}}(V,\Hom_\mathbbm{k}(M,N))\cong \Hom_{A^{\op}\otimes A}(V,\Hom_{\mathbbm{k}}(\mathbb{D}N,\mathbb{D} M))$ proving that $\mathbb{D}$ defines a duality on $\modu \mathfrak{A}$. 
\end{proof}

We recall the main result of \cite{KKO14} stating that the category of $\Delta$-filtered modules for a quasi-hereditary algebra $\Lambda$ can be obtained as the category of modules for a bocs $\mathfrak{A}$ which can be constructed explicitly from the $A_\infty$-algebra structure on $\Ext^*_\Lambda(\Delta,\Delta)$.

\begin{thm}\label{maintheoremKKO}
Let $\Lambda$ be a quasi-hereditary algebra. Then, the following statements hold.
\begin{enumerate}[(i)]
\item There exists a Morita equivalent algebra $R\sim_{\Mor} \Lambda$ such that $R$ is the right algebra of a directed normal bocs $(A,V)$. The algebra $A$ can be chosen to be basic. 
\item Conversely, the right algebra of every directed normal bocs is quasi-hereditary. In this case, there are equivalences of categories $\modu \mathfrak{A}\simeq \mathcal{F}(\Delta_R)\simeq \ind_A^R$ where the latter is the category of all induced modules from $A$ to $R$, i.e. all $R$-modules of the form $R\otimes_A M$ for some $A$-module $M$. In particular, the simple $A$-module $L(\mathtt{i})$ in $\modu \mathfrak{A}$ is mapped to the standard module $\Delta(\mathtt{i})=R\otimes_A L(\mathtt{i})$ for $R$. 
\item Dually, the left algebra of every directed normal bocs is quasi-hereditary with equivalences of categories $\modu \mathfrak{A}\simeq \mathcal{F}(\nabla_L)\simeq \coind_A^L$ where the latter is the category of all coinduced modules from $A$ to $L$, i.e. $L$-modules of the form $\Hom_{A}(L,M)$ for some $A$-module $M$.
\item\label{maintheoremKKO:iv} Let $\mathfrak{A}=(A,V)$ be a directed normal bocs with right algebra $R$. Then, 
\[\Ext^j_A(M,N)\cong \Ext^j_R(R\otimes_A M,R\otimes_A N)\] for all $j\geq 2$. 
\item The left algebra $L$ of a directed normal bocs $\mathfrak{A}$ is Morita equivalent to the Ringel dual of its right algebra $R$.
\item Moreover, for $\mathfrak{A}^{\op}$, $L_{\mathfrak{A}^{\op}}$ is isomorphic to $R_\mathfrak{A}^{\op}$. 
\end{enumerate}
\end{thm}

\begin{rmk}\label{exactstructureonbocsrep}
The equivalence $\modu \mathfrak{A}\simeq \ind_A^R$ induces  the structure of an exact category on $\modu \mathfrak{A}$ through restriction of the natural exact structure on $\modu R$. This exact structure can alternatively be defined by setting the exact sequences to be those which are equivalent (in $\modu \mathfrak{A}$) to images of exact sequences under  the embedding $\modu A\to \modu \mathfrak{A}$, for details see \cite{KKO14, KM17}.
\end{rmk}

\begin{rmk} 
The duality $\mathbb{D}\colon \modu \mathfrak{A} \to \modu \mathfrak{A}^{\op}$ of Lemma \ref{lem_duality} is compatible with the duality $\mathbb{D}(-) = \Hom_{\mathbbm{k}}(-,\mathbbm{k})\colon \modu R_{\mathfrak{A}} \to \modu L_{\mathfrak{A}^{\op}}$ in the sense that the following diagram commutes up to natural isomorphism:
\[
\begin{tikzcd}
\modu \mathfrak{A}\arrow{r}{\mathbb{D}}\arrow{d}[swap]{R\otimes_A -}&\modu \mathfrak{A}^{\op}\arrow{d}{\Hom_{A^{\op}}(L_{\mathfrak{A}^{\op}},-)}\\ \modu R_{\mathfrak{A}} \arrow{r}{\mathbb{D}} &\modu L_{\mathfrak{A}^{\op}}
\end{tikzcd}
\]
Indeed, for a left $A$-module $M$, the dual of the induced $R_{\mathfrak{A}}$-module is $\Hom_{\mathbbm{k}}(R_{\mathfrak{A}}\otimes_A M, \mathbbm{k})$ with right $R_{\mathfrak{A}}$-module structure induced by the left $R_{\mathfrak{A}}$ -module structure on $R_{\mathfrak{A}}\otimes_A M$. On the other hand by the tensor-hom adjunction, the left $L_{\mathfrak{A}^{\op}}$-module coinduced from $\mathbb{D}(M)$, $\Hom_{A^{\op}}(L_{\mathfrak{A}^{\op}}, \Hom_{\mathbbm{k}}(M, \mathbbm{k}))$, is isomorphic to $\Hom_{\mathbbm{k}}(M \otimes_{A^{\op}} L_{\mathfrak{A}^{\op}}, \mathbbm{k})$ with left $L_{\mathfrak{A}^{\op}}$-module structure induced by the right $L_{\mathfrak{A}^{\op}}$-module structure of $M \otimes_{A^{\op}} L_{\mathfrak{A}^{\op}}$. Since $R_{\mathfrak{A}}^{\op} \cong L_{\mathfrak{A}^{\op}}$, we have an isomorphism $(M\otimes_{A^{\op}} L_{\mathfrak{A}^{\op}})_{L_{\mathfrak{A}^{\op}}} \cong {}_{R_{\mathfrak{A}}}(R_{\mathfrak{A}}\otimes_A M)$ which induces a natural isomorphism $\mathbb{D}(\ind_A^{R_{\mathfrak{A}}}(M)) \cong \coind_A^{L_{\mathfrak{A}^{\op}}}(\mathbb{D}(M))$ of left $L_{\mathfrak{A}^{\op}}$ modules. 
\end{rmk}

Our goal in this article is to construct, directly from $(A,V)$, a bocs $(B,W)$ whose right algebra is Morita equivalent to the Ringel dual of the right algebra of $(A,V)$, i.e. the left algebra of $(A,V)$. To describe explicitly how the bocs $(B,W)$ is obtained we need the following lemma from \cite[Lemmas 7.5--7.7]{KKO14}.

\begin{lem}\label{bocsconstruction}
Let $B$ be a category with set of objects $\{\mathtt{1},\ldots,\mathtt{n}\}$ and let $U_1$ be a $B$-bimodule. Assume that the tensor category $U:=\bigoplus_{j=0}^\infty U_1^{\otimes j}$ is endowed with the tensor grading, i.e. $\deg B=0$ and $\deg U_1=1$. Suppose that $U$ is equipped with a differential $d$. Denote by $(d(B))$ the $B$-bimodule generated by $d(B)$ and let $W:=U_1/(d(B))$ with $\pi\colon U_1\to W$ the canonical projection. 
\begin{enumerate}[(i)]
\item There is a prebocs $(B,W,\mu)$ such that $\mu\pi=(\pi\otimes_B\pi)d_1$.
\item Assume, in addition, that $U_1=U_{\Omega}\oplus \overline{U}$ is decomposed as $B$-bimodule, where $U_{\Omega}$ is a projective bimodule $U_{\Omega}=\bigoplus_{\mathtt{i}\in \mathbb{L}} B\omega_{\mathtt{i}} B$ with $\omega_{\mathtt{i}}$ a generator of $U_{\Omega}$, i.e. the image of an element $e_{\mathtt{i}}\otimes e_\mathtt{i}$ under a fixed direct summand embedding $Be_{\mathtt{i}}\otimes_{\mathbbm{k}} e_{\mathtt{i}}B\hookrightarrow U_\Omega$. Suppose that 
\begin{enumerate}[{(d}1{)}]
\item $d(\omega_{\mathtt{i}})=\omega_{\mathtt{i}}\otimes \omega_{\mathtt{i}}$, 
\item for all $b\in B(\mathtt{i},\mathtt{l})$ we have $d(b)=\omega_\mathtt{l}b-b\omega_\mathtt{i}+\partial b$ for some $\partial b\in \overline{U}$,
\item for all $u\in \overline{U}(\mathtt{i},\mathtt{l})$ we have $d(u)=\omega_\mathtt{l}u+u\omega_\mathtt{i}+\partial u$ for some $\partial u\in \overline{U}\otimes \overline{U}$.
\end{enumerate}
Then, the prebocs is a bocs with counit $\varepsilon\colon W\to B$ such that $\tilde{\varepsilon}=\varepsilon \pi$ where $\tilde{\varepsilon}\colon U_1\to B$ is given by $\tilde{\varepsilon}(\omega_\mathtt{i})=\mathbbm{1}_\mathtt{i}$ and $\tilde{\varepsilon}(\overline{U})=0$. 
\item If $\overline{U}$ is a projective bimodule then $\overline{W}:=\ker \varepsilon$ is a projective bimodule.
\end{enumerate}
\end{lem}
Above, we denote by $B(\mathtt{i},\mathtt{l})$, respectively $\overline{U}(\mathtt{i}, \mathtt{l})$, morphisms from $\mathtt{i}$ to $\mathtt{l}$ in $B$, respectively in $\overline{U}$, i.e. $B(\mathtt{i},\mathtt{l})\cong e_\mathtt{l}Be_\mathtt{i}$ and $\overline{U}(\mathtt{i},\mathtt{l})\cong e_\mathtt{l} \overline{U}e_\mathtt{i}$. 

Let $\Lambda$ be a quasi-hereditary algebra. Let $E:=\Ext^*_\Lambda(\Delta,\Delta)$ be the $\Ext$-algebra  of the direct sum of the standard modules. As the cohomology of the dg algebra $\Hom_\Lambda(P,P)$, where $P$ is a projective resolution of $\Delta$, by Kadeishvili's theorem \cite[Theorem 1]{Kad82}, $E$ has the structure of an  $A_\infty$-algebra. Let $C = T(E[1])$ be the differential graded coalgebra equal to the bar construction of $E$. The $\mathbbm{k}$-dual of $C$, $\mathcal{D} = \mathbb{D}C$ is a differential graded algebra. Then $U = \mathcal{D}/(\mathcal{D}_{\leq -1}, d(\mathcal{D}_{-1})) $ is differential graded algebra satisfying the conditions of the above lemma. The resulting bocs, denoted by $(A,V)$, has a right algebra $R$ Morita equivalent to $\Lambda$. Note that since $E$ as well as its bar construction are finite dimensional, we could equally well have first taken the dual of $\mathbb{D}E$ and then apply the cobar construction, i.e. considered the differential graded algebra $T((\mathbb{D}E)[-1])$, which is also isomorphic to $T(\mathbb{D}(E[1]))$. This is what was considered in \cite{KKO14}. For the equivalence, see e.g. \cite[Lemma 9]{EL17}, and also \cite[Section 19]{FHT01} for the case of DG algebras.

\section{The homotopically projective objects $\Box_\mathtt{i}$}\label{sec_homotopically_projective}

	Let $\mathfrak{A}$ be a directed normal bocs, $R$ its right algebra, and $T \colon \modu \mathfrak{A} \to \modu R$ the functor given on objects by $R\otimes_A (-)$. To define $T$ on morphisms, note that we have an isomorphism $\gamma\colon \Hom_\mathfrak{A}(M,N)\cong \Hom_A(M,R\otimes_A N)$. Then for $f\in \Hom_\mathfrak{A}(M,N)$, $T(f)$ is the composition $T(f)\colon R \otimes_A M \xrightarrow{\id_R\otimes \gamma(f)} R \otimes_A R \otimes_A N \xrightarrow{\Hom_A(\mu,A)\otimes_A N } R \otimes_AN $.
	The functor $T$ is fully faithful and yields an equivalence of $\modu \mathfrak{A}$ with the full subcategory $\mathcal{F}(\Delta)$ of $\modu R$. The derived functor $T \colon \mathcal{D}^b(\modu \mathfrak{A}) \to \mathcal{D}^b(\modu R)$ is an equivalence and it maps simple $\mathfrak{A}$-modules $L(\mathtt{i})$ to standard $R$-modules. With Lemma \ref{propertiesofBox} below we show that, for every $\mathtt{i} \in \{\mathtt{1},\ldots, \mathtt{n}\}$, the category $\mathcal{D}^b(\modu \mathfrak{A})$ contains a homotopically projective object $\Box_{\mathtt{i}}$ representing a certain cohomology functor, see Lemma \ref{propertiesofBox}. In particular, $\dim \Hom_{\mathfrak{A}}(\Box_\mathtt{i}, L(\mathtt{l})) =\delta_{\mathtt{i} \mathtt{l}}$ and $\Ext^j(\Box_\mathtt{i}, L(\mathtt{l}))$ vanishes for $j\neq 0$. Under the equivalence $T$, we get an analogous property of objects $T(\Box_{\mathtt{i}})$ and $\Delta(\mathtt{l})\cong T(L(\mathtt{l}))$. It follows that $\langle T(\Box_\mathtt{i})\rangle $ is a full exceptional collection left dual to $\langle \Delta(j) \rangle$. Since the collection $\langle \nabla(\mathtt{i})\rangle$ is right dual to $\langle \Delta(\mathtt{l})\rangle$, see Lemma \ref{costandards}, we conclude that $T(\Box_\mathtt{i})$ is the image under the inverse of the Serre functor of $\nabla(\mathtt{i})$:
	$$
	T(\Box_{\mathtt{i}}) = \textrm{RHom}_{R^{\textrm{op}}} (\mathbb{D}(\nabla(\mathtt{i})), R),
	$$
	where $\mathbb{D}(\nabla(\mathtt{i})) = \Hom_{\mathbbm{k}} (\nabla(\mathtt{i}), \mathbbm{k})$. 
		
	Since the duality $\mathbb{D} = \Hom_{\mathbbm{k}}(-, \mathbbm{k}) \colon \modu \mathfrak{A}^{\textrm{op}} \to \modu \mathfrak{A}$ preserves simple modules, we have $\dim \Hom_{\mathfrak{A}}(L(\mathtt{l}), \mathbb{D}(\Box_{\mathtt{i}}^{\mathfrak{A}^{\textrm{op}}})) = \delta_{\mathtt{i}\mathtt{l}}$, i.e. the images of $\mathbb{D}(\Box_{\mathtt{i}}^{\mathfrak{A}^{\textrm{op}}})$ in $\mathcal{D}^b(\modu R)$ are the costandard modules. In Theorem \ref{modN} and Proposition \ref{dualmodN} we shall give equivalent descriptions of the extension closed subcategory $\mathcal{F}(\Box_\mathtt{i})$ of $\mathcal{D}^b(\modu \mathfrak{A})$ generated by $\Box_{\mathtt{i}}$, while in the main Theorem \ref{thm_Rigel_dual_bocs} we shall use these descriptions of $\mathcal{F}(\Box^{\mathfrak{A}^{\textrm{op}}})$ to describe $\mathcal{F}(\nabla)$ as a category of modules over a bocs.


Let $\mathfrak{A} = (A,V,\mu, \varepsilon)$ be a directed normal  bocs. Then $V = A \oplus \overline{V}$ as left and as right modules (but in general not as bimodules) and the restriction of the comultiplication $\mu \colon V \to V \otimes_A V$ to $\overline{V}$ is given by 
\begin{align*}
&\overline{\mu} \colon \overline{V} \to (A \otimes_A \overline{V}) \oplus (\overline{V} \otimes_A A) \oplus (\overline{V} \otimes_A \overline{V}),& &\overline{\mu} = (\omega \otimes_A \textrm{Id}, \textrm{Id} \otimes_A \omega, \partial_1)^T.&
\end{align*}

It is well known that to a normal bocs $(A,V)$ one can associate a tensor algebra equipped with a differential, see e.g. \cite[Lemma 3.2]{BSZ09}. (The first proof of this fact in the case where $A$ is the path algebra of a quiver is due to Roiter, see \cite{Roi79, Roi80}, cf. Theorem \ref{bocsbiquiver}.) Let $\mathcal{U}:=A[\overline{V}]:=\bigoplus_{j=0}^\infty \overline{V}^{\otimes j}$ be the tensor algebra considered as a graded algebra via the tensor grading, i.e. $\deg A=0$ and $\deg \overline{V}=1$. Defining $\partial_0 \colon A \to \overline{V}$ and $\partial_1 \colon \overline{V} \to \overline{V} \otimes_A \overline{V}$ via
\begin{align*} 
&\partial_0 a:=a\omega-\omega a,& &\partial_1 v=\mu(v)-\omega\otimes v-v\otimes \omega,&
\end{align*} 
and extending by the graded Leibniz rule we obtain a differential $\partial \colon \mathcal{U} \to \mathcal{U}$. Note that if $\mathfrak{A}=(A,V)$ is directed, then $\overline{V}^{\otimes_A \mathtt{n}}=0$ and hence $\mathcal{U}$ is finite dimensional. The following proposition is well known, cf. \cite[Proposition 3.5]{BSZ09}, \cite[Lemma 9.1]{KKO14}.

\begin{prop}\label{boxrepresentations}
Let $\mathfrak{A}=(A,V)$ be a directed normal  bocs. Let $Q_1^1$ be a set of generators for the projective $A$-$A$-bimodule $\overline{V}$ (i.e. elements corresponding to $e_\mathtt{l}\otimes_\mathbbm{k} e_\mathtt{i}$ in a direct summand of $\overline{V}$ of the form $Ae_\mathtt{l}\otimes_{\mathbbm{k}} e_\mathtt{i}A$; in this case we write $v\colon \mathtt{i}\to \mathtt{l}\in Q_1^1$). Write $\omega_\mathtt{i}$ for a generator of $Ae_\mathtt{i}\otimes_{\mathbbm{k}} e_\mathtt{i}A$. Let $\partial_0$ and $\partial_1$ be as defined above. The category $\modu \mathfrak{A}$ can be described equivalently as the category of $A$-modules with morphisms 
\[f\in \Hom_{A\otimes A^{\op}}\left(\bigoplus_{\mathtt{i}\in \{\mathtt{1},\dots,\mathtt{n}\}}A\omega_\mathtt{i} A\oplus\bigoplus_{v\in Q_1^1} Av A, \Hom_{\mathbbm{k}}(M,N)\right)\]
satisfying for all $a\in A(\mathtt{i},\mathtt{l})$ the relation
\[f(\omega_\mathtt{l}a-a\omega_\mathtt{i}+\partial_0 a)=0.\] 
In this language, the composition of two morphisms $f,g$ is given as
\begin{align*}
(gf)(\omega_{\mathtt{i}})&:=g(\omega_\mathtt{i})f(\omega_\mathtt{i})\\
(gf)(v)&:=g(\omega_\mathtt{l})f(v)+g(v)f(\omega_\mathtt{i})+\sum_{(v)} g(v_{(1)})f(v_{(2)})
\end{align*}
for all $\mathtt{i}\in \{\mathtt{1},\dots,\mathtt{n}\}$ and $v\colon \mathtt{i}\to \mathtt{l}$ runs through the elements of $Q_1^1$. Here we use Sweedler notation and write $\partial_1(v)=\sum_{(v)} v_{(1)}\otimes v_{(2)}$ with $v_{(1)}, v_{(2)}\in \overline{V}$. 
\end{prop}

In the remainder we describe the morphisms of bocs representations in this language which should be compared to the familiar presentation of morphisms of quivers with relations. Traditionally for a basic algebra $A=\mathbbm{k}Q/I$, a morphism of representations $\gamma$ is given by $\mathtt{n}$ linear maps $\gamma_\mathtt{i}\colon M_\mathtt{i}\to N_\mathtt{i}$ such that $a \gamma_\mathtt{i}(x)=\gamma_{\mathtt{l}}(ax)$ for all $a\in A(\mathtt{i},\mathtt{l})$ and all $x\in M_\mathtt{i}$. Setting $g(\omega_\mathtt{i})=\gamma_\mathtt{i}$ and $g(v)=0$ defines a morphism of bocs representations $g\colon M\to N$. It is easy to check that this defines an essentially surjective and faithful functor $\Phi\colon \modu A\to \modu \mathfrak{A}$, which is in general not full. In the description of the previous  proposition, a morphism $f$ is in the image of $\Phi$ if and only if $f(v)=0$ for all $v\in Q^1_1$. We call such a morphism $A$-linear. From now on, we will not distinguish between a morphism in $\modu A$ and its image in $\modu \mathfrak{A}$.  
Additionally to those $A$-linear morphisms there are in general some extra maps. In the case of a regular bocs (see Section \ref{applications} for a definition) the map with $f(\omega_\mathtt{i})=0$ for all $\mathtt{i}\in \{\mathtt{1},\dots,\mathtt{n}\}$ and $f(v)=\mathbbm{1}_\mathbbm{k}$ for some $v\colon \mathtt{i}\to \mathtt{l}$ defines a homorphism of bocs representations between the simple $A$-modules $L(\mathtt{i})$ and $L(\mathtt{l})$, i.e. a homomorphism between the corresponding standard modules over the associated right algebra, see Lemma \ref{miemietz-lemma} for the precise statement.

The language of differential biquivers, introduced by M. Kleiner and A. Roiter in \cite{KR75} is useful when working with normal bocses $\mathfrak{A} = (A,V)$ with projective kernel where the algebra $A$ is hereditary.  Such bocses are also called \emphbf{free} and they correspond to almost strong exceptional collections and left strongly quasi-hereditary algebras (Proposition \ref{prop_almost_strong}).  
A \emphbf{biquiver} is a quiver $(Q_0,Q_1)$ where the arrows are either of degree $0$ or $1$. The arrows of degree $0$ are called \emphbf{solid}. The arrows of  degree $1$ are called \emphbf{dashed}. A \emphbf{differential biquiver} is a biquiver $Q$ together with a linear map $\partial\colon \mathbbm{k}[Q]\to \mathbbm{k}[Q]$ of degree $1$ which squares to $0$, satisfies $\partial(e)=0$ for the trivial paths and satisfies the graded Leibniz rule. The following theorem is due to A. Roiter, \cite{Roi79, Roi80}

\begin{thm}\label{bocsbiquiver}
	There is a one-to-one correspondence between free normal bocses with projective kernel and differential biquivers given by:
	\begin{itemize}
		\item Given a differential biquiver $(Q,\partial)$, the corresponding bocs $(A,V)$ is given by $A:=kQ^0$, the path algebra of the degree $0$ part, and $V=A\omega\oplus \underbrace{A\otimes_{\mathbb{L}} \mathbbm{k}Q^1\otimes_{\mathbb{L}} A}_{=:\overline{V}}$ as left modules (where $A\omega\cong A$) with right module structure given by the embedding $\begin{pmatrix}1\\\partial\end{pmatrix}\colon A\to A\omega\oplus \overline{V}$. The comultiplication is then given by \[\mu(a\omega+v)=a\omega\otimes \omega +v\otimes \omega+\omega\otimes v + \partial_1(v)\] and the counit by $\varepsilon(a\omega+v)=a$. 
		\item In the other direction, let $(A,V)$ be a free normal bocs with projective kernel. Let $\omega$ be a group-like element and let $Q$ be the biquiver with degree $0$ part such that $kQ^0\cong A$. Let $Q^1$ be a free generating system of $\overline{V}:=\ker\varepsilon$, i.e. $\overline{V}\cong A\otimes_{\mathbb{L}} \mathbbm{k}Q^1\otimes_{\mathbb{L}} A$. Define $\partial(a)=a\omega-\omega a$ for a solid arrow $a$ and $\partial(v)=\mu(v)-\omega\otimes v-v\otimes \omega$ for a dashed arrow (cf. Proposition \ref{boxrepresentations}). 
	\end{itemize}
\end{thm}

\begin{rmk}\label{rmk_bocs_of_strongly}
	Specialising the general construction from \cite{KKO14} to the case of left strongly quasi-hereditary algebras, $A=\mathbbm{k}Q^0$ is given by $\mathbb{L}[s^{-1}\mathbb{D}\Ext^1(\Delta,\Delta)]$ and similarly $\mathbbm{k}Q^1$ is given by $\mathbb{L}[s^{-1}\mathbb{D}\Hom(\Delta,\Delta)]$. The differential is obtained as the dual of the higher multiplications on the dual bar construction of the $A_\infty$-algebra $\Ext^*_{\Lambda}(\Delta,\Delta)$, see \cite{KKO14, K16} for more details.
\end{rmk}

Using the description of Proposition \ref{boxrepresentations}, we define the object $\Box_{\mathtt{i}}\in \mathcal{D}^b(\modu \mathfrak{A})$ by 
\[(\Box_{\mathtt{i}})^j:=\begin{cases}A\otimes_{\mathbb{L}} L(\mathtt{i})&\text{for $j=0$,}\\\overline{V}^{\otimes_{A} j}\otimes_{\mathbb{L}} L(\mathtt{i})&\text{for $j\geq 1$,}\\0&\text{else,}\end{cases}\]
with differential given by 
\begin{align*}
&d_\Box(\omega_{\mathtt{l}})(x\otimes \lambda)=-\partial(x)\otimes \lambda,& &\textrm{for all } \mathtt{l}\in \{\mathtt{1},\dots,\mathtt{n}\},\, x \in \overline{V},&\\
&d_\Box(v)(x\otimes \lambda)=v\otimes x\otimes \lambda,& &\textrm{for all } v\in Q_1^1.&
\end{align*}
We write $\mathbbm{1}_\mathtt{i}$ for the element $1\otimes e_\mathtt{i}\in  \Box_\mathtt{i}^0=A\otimes_\mathbb{L} L(\mathtt{i})=P(\mathtt{i})$.

\begin{lem}
$\Box_{\mathtt{i}}\in \mathcal{D}^b(\modu \mathfrak{A})$. 
\end{lem} 

\begin{proof}
That $\Box_{\mathtt{i}}$ is a bounded complex follows from the fact that the bocs is directed.
We have to check that $d_\Box$ defines a morphism in $\modu \mathfrak{A}$ which squares to $0$. For the first claim note that for $a\in A(\mathtt{l},\mathtt{m})$, $x\in \overline{V}^{\otimes_{A} j}(\mathtt{i},\mathtt{l})$ and $\lambda\in L(\mathtt{i})$,
\begin{align*}
d_\Box(\omega_\mathtt{m}a-a\omega_\mathtt{l}+\partial(a))(x\otimes \lambda)&=-\partial(ax)\otimes \lambda+a\partial(x)\otimes \lambda+\partial(a)x\otimes \lambda\\
&=(-\partial(ax)+a\partial(x)+\partial(a)x)\otimes \lambda=0
\end{align*}
because $\partial$ satisfies the graded Leibniz rule. It also squares to zero:
\begin{align*}
(d^{j+1}_\Box d^j_\Box)(\omega_\mathtt{l})(x\otimes \lambda)&=d_\Box^{j+1}(\omega_\mathtt{l})\left(d_\Box^j(\omega_\mathtt{l})(x\otimes \lambda)\right)
=d_\Box^{j+1}(\omega_\mathtt{l})(-\partial(x)\otimes \lambda)\\
&=\partial^2(x)\otimes \lambda
=0
\end{align*}
as $\partial$ also squares to zero. In the next set of equations we use Sweedler notation and write $\partial(v)=\sum_{(v)}v_{(1)}\otimes v_{(2)}$ for $v\in Q^1_1(\mathtt{l},\mathtt{m})$. Furthermore, as the part with ``$\otimes \lambda$'' does not change throughout, we surpress it from the notation. Then, 
\begin{align*}
(d^{j+1}_\Box d^j_\Box)(v)(x)&=d_\Box^{j+1}(\omega_\mathtt{m})(d_\Box^j(v)(x))+d_\Box^{j+1}(v)(d_\Box^j(\omega_\mathtt{l})(x))+\sum_{(v)}d_\Box^{j+1}(v_{(1)})(d_\Box^j(v_{(2)})(x))\\
&=d_\Box^{j+1}(\omega_\mathtt{m})(v\otimes x)+d_\Box^{j+1}(v)(-\partial(x))+\sum_{(v)} d_\Box(v_{(1)})(v_{(2)}\otimes x)\\
&=-\partial(v\otimes x)-v\otimes \partial(x)+\partial(v)\otimes x=0
\end{align*}
as $\partial$ satisfies the graded Leibniz rule.
\end{proof}

\begin{ex}\label{example}
We illustrate the notions discussed in Sections \ref{sec_cat_F_Box} and \ref{dualising} in the following running example which belongs to the class of curve-like algebras studied in Section \ref{applications}. It corresponds to the Auslander algebra of the algebra $\mathbbm{k}[x]/(x^3)$, see 2C in Section \ref{ssec_3_simples}. 

Consider the differential biquiver
\begin{equation}\label{eqtn_quiver_with_3_vertices}
\xymatrix{\mathtt{1} \ar@<0.5ex>[r]^a \ar@{.>}@<-0.5ex>[r]_{\varphi} \ar@/^1pc/@<1ex>[rr]^c \ar@{.>}@/_1pc/@<-1ex>[rr]_{\chi} & \mathtt{2} \ar@<0.5ex>[r]^b \ar@{.>}@<-0.5ex>[r]_{\psi}& \mathtt{3} }
\end{equation}
with 
\begin{align*}
&\partial^1(\chi) = \psi \varphi,& &\partial^0(c)= \psi a + b \varphi.&
\end{align*}
Using the construction in Theorem \ref{bocsbiquiver}, it yields a bocs $\mathfrak{A} = (A, V)$  where the algebra $A$ is the path algebra of the quiver on the solid arrows of (\ref{eqtn_quiver_with_3_vertices}), which is of extended Dynkin type $\tilde{\mathbb{A}}_2$. The bimodule $\overline{V}$ is the direct sum of the projective bimodules $Ae_\mathtt{2}\otimes e_\mathtt{1}A$, $Ae_\mathtt{3}\otimes e_\mathtt{2}A$ and $Ae_\mathtt{3}\otimes e_\mathtt{1}A$. The respective generators are indicated by dashed arrows in the quiver. Using this description,
\begin{align*}
\overline{V} = \textrm{Span}_{\mathbbm{k}}(\varphi,\, \psi,\, \chi,\, \psi a,\, b\varphi).
\end{align*}
The bimodule $V$ is then given by $V=A\omega\oplus \overline{V}$ as left modules, where the right action is deformed by $\partial^0$, i.e. $\omega\cdot c=c\omega-\psi a-b\varphi$. 

The comultiplication $\mu$ is given on generators by
\begin{align*}
\mu(\varphi)&=\omega\otimes \varphi+\varphi\otimes \omega,\\
\mu(\psi)&=\omega\otimes \psi+\psi\otimes \omega,\\
\mu(\chi)  &=\omega\otimes \chi+\chi\otimes \omega+ \psi \otimes \varphi.
\end{align*}

The complexes $\Box_\mathtt{1}$, $\Box_\mathtt{2}$, $\Box_\mathtt{3}$ are determined by the simple $\mathbb{L}$-modules $L(\mathtt{1})$, $L(\mathtt{2})$ and $L(\mathtt{3})$. As complexes of $\mathfrak{A}$-modules, $\Box_\mathtt{1}$, $\Box_\mathtt{2}$ and $\Box_\mathtt{3}$ are given by left $A$-modules $(\Box_\mathtt{i})^j$ together with $\mathbbm{k}$-linear maps $(\Box_\mathtt{i})^j \to (\Box_\mathtt{i})^{j+1}$, for $\omega_\mathtt{1}$, $\omega_\mathtt{2}$, $\omega_\mathtt{3}$ and any dashed arrow in \ref{eqtn_quiver_with_3_vertices}, see Proposition \ref{boxrepresentations}.

The complex $\Box_\mathtt{1}$ consists of three modules. $(\Box_\mathtt{1})^0\cong P_\mathtt{1}$ is the projective $A$-module with basis $e_\mathtt{1},\, a,\, ba,\, c$. The module $(\Box_\mathtt{1})^1$ has $\mathbbm{k}$-basis $\varphi,\, \chi,\, b\varphi, \, \psi a$ and $(\Box_\mathtt{1})^2$ is one-dimensional with basis $\psi \otimes \varphi$.

For $\mathtt{i}=\mathtt{1},\mathtt{2},\mathtt{3}$, the differential $d_{\Box_\mathtt{1}}(\omega_\mathtt{i})$ is non-zero on the following elements:
\begin{align*}
&d_{\Box_\mathtt{1}}(\omega_\mathtt{i})(c) = -\psi a - b\phi,& &d_{\Box_\mathtt{1}}(\omega_\mathtt{i})(\chi)=- \psi \otimes \varphi.&
\end{align*}
For the dashed arrows in the quiver \ref{eqtn_quiver_with_3_vertices}, the maps $d_{\Box_\mathtt{1}}$ are non-zero on the following elements of the basis:
\begin{align*}
&d_{\Box_\mathtt{1}}(\varphi)(e_\mathtt{1})= \varphi,& &d_{\Box_\mathtt{1}}(\psi)(a)= \psi a,& &d_{\Box_\mathtt{1}}(\psi)(\varphi)= \psi \otimes \varphi,& &d_{\Box_\mathtt{1}}(\chi)(e_\mathtt{1})= \chi.&
\end{align*} 

Writing the complex as a complex of representations of the corresponding differential biquiver, i.e. using Proposition \ref{boxrepresentations}, we obtain the following

\begin{center}
	\includegraphics{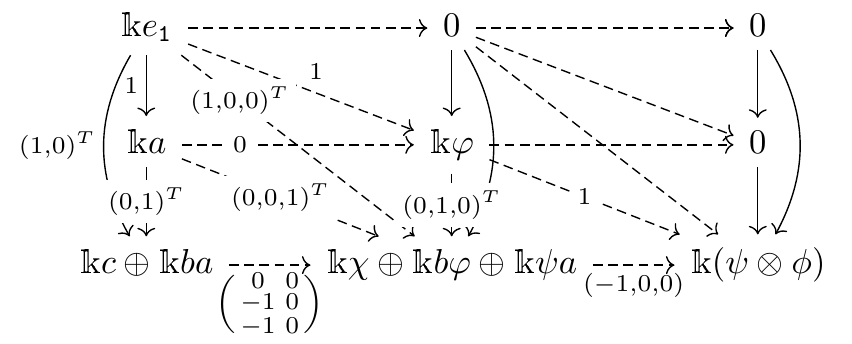}
\end{center}

The complex $\Box_\mathtt{2}$ consists of two modules. $(\Box_\mathtt{2})^0$ is the projective $A$-module with basis $e_\mathtt{2},\, b$ while $(\Box_\mathtt{2})^1$ is one dimensional with basis $\psi$. The differentials $d_{\Box_\mathtt{2}}(\omega_\mathtt{i})$, $d_{\Box_\mathtt{2}}(\varphi)$ and $d_{\Box_\mathtt{2}}(\chi)$ vanish while
\begin{align*}
&d_{\Box_\mathtt{2}}(\psi)(e_\mathtt{2})= \psi.
\end{align*}

Again using the language of biquivers we obtain:

\begin{center}
\includegraphics{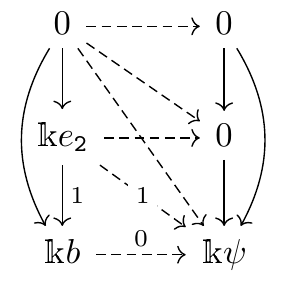}
\end{center}

Finally, the complex $\Box_\mathtt{3}$ consists of the one-dimensional projective $A$-module with basis $e_\mathtt{3}$ in degree zero. 

\end{ex}

The next lemma is a slightly strengthened form of the universal property of projective modules in $\modu \mathfrak{A}$. 

\begin{lem}\label{existenceofs}
Let $M\in \modu \mathfrak{A}$. Let $x\in M_\mathtt{i}$. Then, there exists a unique morphism $s\colon P(\mathtt{i})\to M$ satisfying $s(\omega_\mathtt{i})(\mathbbm{1}_\mathtt{i})=x$ and $s(v)=0$ for all $v\in Q_1^1$. 
\end{lem}

\begin{proof}
The statement is true regarding $P(\mathtt{i})$ and $M$ as objects in $\modu A$. Applying the natural functor $\modu A\to \modu \mathfrak{A}$ gives a morphism in $\modu \mathfrak{A}$ with the claimed properties.
\end{proof}

Let $M\in \mathcal{D}^b(\modu \mathfrak{A})$. Then, as one can easily see from the description of the morphisms in $\modu \mathfrak{A}$, there is an exact functor $\mathcal{D}^b(\modu \mathfrak{A})\to \mathcal{D}^b(\modu \mathbb{L})$ given by sending $(M,d_M)$ to $(M,d_M(\omega))$. As $\mu(\omega)=\omega\otimes \omega$, it follows that $(d_M\circ d_M)(\omega)=d_M(\omega)\circ d_M(\omega)$, hence the functor sends complexes to complexes.
Composing with the standard cohomology functor $H^j$ yields a cohomological functor $H^j\colon \mathcal{D}^b(\modu \mathfrak{A})\to \modu \mathbb{L}$. Composing further with the projection to the $\mathtt{i}$-th component yields cohomological functors $H_\mathtt{i}^j\colon \mathcal{D}^b(\modu \mathfrak{A})\to \modu \mathbbm{k}$. The next lemma shows that these functors are represented by $\Box_\mathtt{i}[-j]$. 

For an idempotent split exact category $\mathcal{E}$ a complex $N$ is called \emphbf{homotopically projective} if $\Hom_{\mathcal{K}^b(\mathcal{E})}(N,M)=0$ for each acyclic complex $M\in K^b(\mathcal{E})$. In this case, $\Hom_{\mathcal{K}^b(\mathcal{E})}(N,-)\cong \Hom_{\mathcal{D}^b(\mathcal{E})}(N,-)$. Dually, a complex $N$ is called \emphbf{homotopically injective} if $\Hom_{\mathcal{K}^b(\mathcal{E})}(M,N)=0$ for each acyclic complex $M$. In this case, $\Hom_{\mathcal{K}^b(\mathcal{E})}(-,N)\cong \Hom_{\mathcal{D}^b(\mathcal{E})}(-,N)$. 

\begin{lem}\label{propertiesofBox}
Let $M\in \mathcal{C}^b(\modu \mathfrak{A})$.
\begin{enumerate}[(i)]
\item\label{Boxhom:i} $\Hom_{\mathcal{C}^b(\modu \mathfrak{A})}(\Box_\mathtt{i},M)\cong Z^0(d_M(\omega_\mathtt{i}))\oplus \bigoplus_{j\in \mathbb{Z}}\bigoplus_{v\in Q^1_1}\Hom_{\mathbbm{k}}(\Box_\mathtt{i}^j(s(v)),M^j(t(v)))$.
\item\label{Boxhom:ii} Two maps are homotopic if and only if they coincide when projected to $H^0(d_M(\omega_\mathtt{i}))$. In particular, $H^j_\mathtt{i}$ is represented by $\Box_\mathtt{i}[-j]$. 
\item\label{Boxhom:iii} The homotopy in \eqref{Boxhom:ii} can be chosen to be a morphism of $A$-modules. 
\item The object $\Box_\mathtt{i}$ is homotopically projective. 
\end{enumerate}
\end{lem}

\begin{proof}
\begin{enumerate}[(i)]
\item We show that the map 
\begin{align*}\Hom_{\mathcal{C}^b(\modu \mathfrak{A})}(\Box_\mathtt{i},M)&\to Z^0(d_M(\omega_\mathtt{i}))\oplus \bigoplus_{j\in \mathbb{Z}}\bigoplus_{v\in Q^1_1}\Hom_{\mathbbm{k}}(\Box_{\mathtt{i}}^j(s(v)),M^j(t(v)))\\
 (f^j)_{j\in \mathbb{Z}}&\mapsto (f^0(\omega_\mathtt{i})(\mathbbm{1}_\mathtt{i}), f^j(v))_{j \in \mathbb{Z}, v\in Q_1^1}
\end{align*}
is an isomorphism.

Let $f=(f^j)_{j\in \mathbb{Z}}\in \Hom_{\mathcal{C}^b(\modu \mathfrak{A})}(\Box_\mathtt{i},M)$. First we show  $f^0(\omega_\mathtt{i})(\mathbbm{1}_\mathtt{i})\in Z^0(d_M(\omega_\mathtt{i}))$:
\begin{align*}
d_M^0(\omega_\mathtt{i})(f^0(\omega_\mathtt{i})(\mathbbm{1}_\mathtt{i}))&=(d_M^0f^0)(\omega_\mathtt{i})(\mathbbm{1}_\mathtt{i})=(f^1d_\Box^0)(\omega_\mathtt{i})(\mathbbm{1}_\mathtt{i})=f^1(\omega_\mathtt{i})(d_\Box^0(\omega_\mathtt{i})(\mathbbm{1}_\mathtt{i}))\\
&=f^1(\omega_\mathtt{i})(-\partial(1)\otimes e_\mathtt{i})=f^1(\omega_\mathtt{i})(0)=0,
\end{align*}
as $\mathbbm{1}_\mathtt{i} = 1 \otimes e_\mathtt{i}$. Next we show that $(f^0(\omega_\mathtt{i})(\mathbbm{1}_\mathtt{i}),f^j(v))$ uniquely defines $f$. For achieving this, for all $\mathtt{m}\in \{\mathtt{1},\dots,\mathtt{n}\}$, $f^j(\omega_\mathtt{m})$ has to be specified from the given data. For $j<0$, it is clear that $f^j(\omega_\mathtt{m})=0$ since $\Box_\mathtt{i}^j=0$. For $j\geq 1$, $f^j(\omega_\mathtt{m})$ can be defined recursively from $f^{j-1}(\omega_\mathtt{l})$: for this we compute $(f^jd_\Box^{j-1})(v)(x\otimes \lambda)$ for $v\in Q_1^1(\mathtt{l},\mathtt{m})$, $x\in \overline{V}^{\otimes_{\mathbb{L}}^{j-1}}(\mathtt{i},\mathtt{l})$, and $\lambda\in L(\mathtt{i})$ in two ways. We again use Sweedler notation and write $\partial(v)=\sum_{(v)} v_{(1)}\otimes v_{(2)}$. On the one hand,
\begin{multline*}\tag{$\star$}
(f^jd_\Box^{j-1})(v)(x\otimes \lambda)=f^j(\omega_\mathtt{m})(d_\Box^{j-1}(v)(x\otimes \lambda))+f^j(v)(d_\Box^{j-1}(\omega_\mathtt{l})(x\otimes \lambda))\\+\sum_{(v)} f^j(v_{(1)})d_\Box^{j-1}(v_{(2)})(x\otimes \lambda),
\end{multline*}
on the other hand, since $f^jd_\Box^{j-1}=d^{j-1}_Mf^{j-1}$, this is equal to
\begin{align*}\tag{$\star \star$}
(d_M^{j-1}f^{j-1})(v)(x\otimes \lambda)=d_M^{j-1}(\omega_\mathtt{m})(f^{j-1}(v)(x\otimes \lambda))+d_M^{j-1}(v)(f^{j-1}(\omega_\mathtt{l})(x\otimes \lambda))\\+\sum_{(v)}d_M^{j-1}(v_{(1)})(f^{j-1}(v_{(2)})(x\otimes \lambda)).
\end{align*}
Comparing the two expressions, one sees that $f^j(\omega_{\mathtt{m}})(d_\Box^{j-1}(v)(x\otimes \lambda))=f^j(\omega_\mathtt{m})(v\otimes x\otimes \lambda)$ is determined by $f^{j-1}(\omega_\mathtt{l})(x\otimes \lambda)$ and summands containing $f(v)$ where $v\in Q_1^1$. 
For $j=0$ note that for $a\in A(\mathtt{i},\mathtt{m})$,
\begin{align*}
f^0(\omega_\mathtt{m})(a)&=(f^0(\omega_\mathtt{m})a)(\mathbbm{1}_\mathtt{i})=f^0(\omega_\mathtt{m}a)(\mathbbm{1}_\mathtt{i})=f^0(a\omega_\mathtt{l})(\mathbbm{1}_\mathtt{i})-f^0(\partial(a))(\mathbbm{1}_\mathtt{i})\\&=af^0(\omega_\mathtt{i})(\mathbbm{1}_\mathtt{i})-f^0(\partial(a))(\mathbbm{1}_\mathtt{i}).
\end{align*}
Hence, $f$ is completely determined by the given data. To check that each such data defines a morphism of complexes, we have to prove that each of the $f^j$ is indeed a morphism of bocs representations and that $(d_M^{j-1}f^{j-1})(\omega_\mathtt{m})=(f^jd_\Box^{j-1})(\omega_\mathtt{m})$ for all $\mathtt{m}\in \{\mathtt{1},\dots,\mathtt{n}\}$ (for the $v\in Q_1^1$ this statement is already true by the construction of $f^j(\omega_\mathtt{m})$). 

For checking that each $f^j$ is a morphism of bocs representations, taking into account $(\star)$ and $\partial(av)=\partial(a)\otimes v+a\partial(v)$, $f^j(\omega_\mathtt{m} a)(v\otimes x)=f^j(\omega_\mathtt{m})(av\otimes x) = f^j(\omega_\mathtt{m})(d_\Box^{j-1}(av)(x))$ is equal to (again surpressing ``$\otimes \lambda$'' from the notation)
\[
(f^jd_\Box^{j-1})(av)(x)-f^j(av)d_\Box^{j-1}(\omega_\mathtt{l})(x)-f^j(\partial a)d_\Box^{j-1}(v)(x)-\sum_{(v)}f^j(av_{(1)})d_\Box^{j-1}(v_{(2)})(x).
\]
By subtracting the term $f^j(a\omega_\mathtt{m})(v\otimes x)=af^j(\omega_\mathtt{m})(v\otimes x) = af^j(\omega_\mathtt{m})(d_\Box^{j-1}(v))(x)$ we get $-f^j(\partial a)d_\Box^{j-1}(v)(x)$ because all the maps involved have been defined to be $A$-linear. Thus, $f^j$ is a morphism of bocs representations, see Proposition \ref{boxrepresentations}.

To check that $(d_M^{j-1}f^{j-1})(\omega_\mathtt{m})=(f^jd_\Box^{j-1})(\omega_\mathtt{m})$ we apply it to some $v\otimes x$ with $v\in Q_1^1(\mathtt{l},\mathtt{m})$ and $x\in \overline{V}^{\otimes_A j-2}(\mathtt{i},\mathtt{l})$ and use induction (the case of $j\leq 0$ being vacuously true). We write $\partial(v)=\sum_{(v)} v_{(1)}\otimes v_{(2)}$. Using the description of composition of Proposition \ref{boxrepresentations}, we get:
\begin{align*}\tag{$\star \star \star$}
(f^jd_\Box^{j-1})(\omega_\mathtt{m})(v\otimes x)&=f^{j}(\omega_\mathtt{m})d_\Box^{j-1}(\omega_\mathtt{m})(v\otimes x)\\
&=\left(-\sum_{(v)}f^j(\omega_\mathtt{m})(v_{(1)}\otimes v_{(2)}\otimes x)\right)+f^j(\omega_\mathtt{m})(v\otimes \partial(x)).
\end{align*}
We compute the two summands separately. Writing $(\partial \otimes \mathbbm{1}_{\overline{V}})\partial(v)=(\mathbbm{1}_{\overline{V}}\otimes \partial)\partial(v)=\sum_{(v)} v_{(1)}\otimes v_{(2)}\otimes v_{(3)}$, and using $(\star)$, $(\star \star)$ we get that the first summand 
\[- \sum_{(v)}f^j(\omega_\mathtt{m})(v_{(1)} \otimes v_{(2)} \otimes x) = -\sum_{(v)} f^j(\omega_\mathtt{m})(d_{\Box}^{j-1}(v_{(1)})  (v_{(2)} \otimes x))\]
 is equal to 
\footnotesize\begin{multline*}
\bigg(\underbrace{\sum_{(v)}-d_M^{j-1}(\omega_\mathtt{m})f^{j-1}(v_{(1)})(v_{(2)}\otimes x)}_{(\diamond)}-\underbrace{d_M^{j-1}(v_{(1)})f^{j-1}(\omega)(v_{(2)}\otimes x)-d_M^{j-1}(v_{(1)})f^{j-1}(v_{(2)})(v_{(3)}\otimes x)}_{(*)}\\+\underbrace{f^j(v_{(1)})d_\Box^{j-1}(\omega)(v_{(2)}\otimes x)+f^j(v_{(1)})d_\Box(v_{(2)})(v_{(3)}\otimes x)}_{(\dagger)}\bigg).
\end{multline*}\normalsize
Again using $(\star)$ and $(\star \star)$ the second summand $f^j(\omega_\mathtt{m})(v \otimes \partial\,x) = f^{j}(\omega_\mathtt{m})(d_{\Box}^{j-1}(v)(\partial\,x))$ of $(\star \star \star)$ is equal to
\begin{multline*}
\underbrace{d_M^{j-1}(\omega_\mathtt{m})f^{j-1}(v)(\partial x)}_{(\diamond)}+\underbrace{d^{j-1}_M(v)f^{j-1}(\omega_\mathtt{l})(\partial x)}_{(\bullet)}-\underbrace{f^j(v)d_\Box^{j-1}(\omega_\mathtt{l})(\partial x)}_{(\circ)}\\+\bigg(\underbrace{\sum_{(v)}d_M^{j-1}(v_{(1)})f^{j-1}(v_{(2)})(\partial x)}_{(*)}-\underbrace{f^j(v_{(1)})d_\Box^{j-1}(v_{(2)})(\partial x)}_{(\dagger)}\bigg).
\end{multline*}
The term marked $(\circ)$ vanishes as $\partial$ is a differential. The three terms marked with $(\dagger)$ cancel out as $\partial$ is a derivation:
\begin{align*}
&d_\Box^{j-1}(\omega)(v_{(2)} \otimes x) + \sum_{(v_{(2)})}d_\Box^{j-1}(v_{(2)})(v_{(3)}\otimes x) - d_\Box^{j-1}(v_{(2)})(\partial x) &\\
&=\sum_{(v_{(2)})}-v_{(2)}\otimes v_{(3)} \otimes x + v_{(2)} \otimes \partial x + \sum_{(v_{(2)})}v_{(2)}\otimes v_{(3)} \otimes x -v_{(2)} \otimes \partial x = 0.&
\end{align*}
Note that all three terms marked $(*)$ start with $d_M^{j-1}(v_{(1)})$. Using $(\star)$ and $(\star \star)$ to define $f^{j-1}(\omega_\mathtt{m})(v_{(2)} \otimes x)$, we get that
\begin{align*}
	&f^{j-1}(v_{(2)})(\partial x) - f^{j-1}(\omega)(v_{(2)} \otimes x) - \sum_{(v_{(2)})}f^{j-1}(v_{(2)})(v_{(3)} \otimes x)&  \\
	& =-d_M^{j-2}(\omega)f^{j-2}(v_{(2)})(x) - d_M^{j-2}(v_{(2)})f^{j-2}(\omega)(x) - \sum_{(v_{(2)})} d_M^{j-2}(v_{(2)})f^{j-2}(v_{(3)})(x).&
\end{align*}
Thus,  the three terms marked $(*)$ combine to give
\begin{multline*}
\sum_{(v)}\underbrace{-d_M^{j-1}(v_{(1)})d_M^{j-2}(\omega)f^{j-2}(v_{(2)})(x)}_{(\diamond)}-\underbrace{d_M^{j-1}(v_{(1)})d_M^{j-2}(v_{(2)})f^{j-2}(\omega_\mathtt{l})(x)}_{(\bullet)}\\\underbrace{-d_M^{j-1}(v_{(1)})d_M^{j-2}(v_{(2)})f^{j-2}(v_{(3)})(x)}_{(\diamond)}.
\end{multline*}
By inductive hypothesis, $d_M^{j-2} \circ f^{j-2} = f^{j-1} \circ d_{\Box}^{j-2}$. Moreover, since $d_M$ squares to zero, the composition defined in Proposition \ref{boxrepresentations} implies that
\begin{align*}
\sum_{(v)}d_M^{j-1}(v_{(1)}) d_M^{j-2}(v_{(2)}) + d_M^{j-1}(\omega_\mathtt{m})d_M^{j-2}(v) + d_M^{j-1}(v)d_M^{j-2}(\omega_\mathtt{l}) =0.
\end{align*}
Then, using $(\star \star)$, the term marked $(\bullet)$ adds up with the one marked $(\bullet)$ in the second summand to give:
\begin{align*}
&d_M^{j-1}(v)f^{j-1}(\omega_\mathtt{l})(\partial x) - \sum_{(v)} d_M^{j-1}(v_{(1)})d_M^{j-2}(v_{(2)})f^{j-2}(\omega_\mathtt{l})(x)&\\
&=d_M^{j-1}(v)f^{j-1}(\omega_\mathtt{l})(\partial x) + d_M^{j-1}(\omega_\mathtt{m})d_M^{j-2}(v) f^{j-2}(\omega_\mathtt{l})(x) + d_M^{j-1}(v)d_M^{j-2}(\omega_\mathtt{l})f^{j-2}(\omega_\mathtt{l})(x)&\\
&=(d_M^{j-1}f^{j-1})(v)(\partial x) - d_M^{j-1}(\omega_\mathtt{m})f^{j-1}(v)(\partial x) - \sum_{(v)} d_M^{j-1}(v_{(1)})f^{j-1}(v_{(2)})(\partial x) &\\
&+ d_M^{j-1}(\omega_\mathtt{m})d_M^{j-2}(v)f^{j-2}(\omega_\mathtt{l})(x) + d_M^{j-1}(v)d_M^{j-2}(\omega_\mathtt{l})f^{j-2}(\omega_\mathtt{l})(x)&\\
&=d_M^{j-1}(\omega_\mathtt{m})f^{j-1}(v)(\partial x) + d_M^{j-1}(v)f^{j-1}(\omega_\mathtt{l})(\partial x) + \sum_{(v)} d_M^{j-1}(v_{(1)})(f^{j-1}(v_{(2)})(\partial x)) &\\
&- d_M^{j-1}(\omega_\mathtt{m})f^{j-1}(v)(\partial x) - \sum_{(v)} d_M^{j-1}(v_{(1)})f^{j-1}(v_{(2)})(\partial x) &\\
&+ d_M^{j-1}(\omega_\mathtt{m})d_M^{j-2}(v)f^{j-2}(\omega_\mathtt{l})(x) + d_M^{j-1}(v)d_M^{j-2}(\omega_\mathtt{l})f^{j-2}(\omega_\mathtt{l})(x)&\\
&=d_M^{j-1}(\omega_\mathtt{m})d_M^{j-2}(v)f^{j-2}(\omega_\mathtt{l})(x) + d_M^{j-1}(v) (d_M^{j-2}f^{j-2})(\omega_\mathtt{l})(x) + d_M^{j-1}(v)f^{j-1}(\omega_\mathtt{l})(\partial x)&\\
&=d_M^{j-1}(\omega_\mathtt{m})d_M^{j-2}(v)f^{j-2}(\omega_\mathtt{l})(x) + d_M^{j-1}(v)(d_M^{j-2}f^{j-2})(\omega_\mathtt{l})(x) + d_M^{j-1}(v)f^{j-1}(\omega_\mathtt{l})(\partial x)&\\
&=d_M^{j-1}(\omega_\mathtt{m})d_M^{j-2}(v)f^{j-2}(\omega_\mathtt{l})(x), 
\end{align*}
since
\begin{align*}
(d_M^{j-2}f^{j-2})(\omega_\mathtt{l})(x) = (f^{j-1}d_\Box^{j-1})(\omega_\mathtt{l})(x) = -f^{j-1}(\partial x).
\end{align*}
Summing up $d_M^{j-1}(\omega_\mathtt{m})d_M^{j-2}(v)f^{j-2}(\omega_\mathtt{l})(x)$ with the terms marked by $(\diamond)$, gives:
\begin{align*}
&d_M^{j-1}(\omega_\mathtt{m})d_M^{j-2}(v)f^{j-2}(\omega_\mathtt{l})(x) + d_M^{j-1}(\omega_\mathtt{m})f^{j-1}(v)(\partial x) -\sum_{(v)}( d_M^{j-1}(\omega_\mathtt{m}) f^{j-1}(v_{(1)})(v_{(2)}\otimes x)&\\
&+d_M^{j-1}(v_{(1)})d_M^{j-2}(\omega_\mathtt{l})f^{j-2}(v_{(2)})(x) + d_M^{j-1}(v_{(1)})d_M^{j-2}(v_{(2)})f^{j-2}(v_{(3)})(x))&\\
&=d_M^{j-1}(\omega_\mathtt{m})d_M^{j-2}(v)f^{j-2}(\omega_\mathtt{l})(x) + d_M^{j-1}(\omega_\mathtt{m})f^{j-1}(v)(\partial x)-\sum_{(v)}(d_M^{j-1}(\omega_\mathtt{m})f^{j-1}(v_{(1)})(v_{(2)}\otimes x) &\\
&- d_M^{j-1}(\omega_\mathtt{m})d_M^{j-2}(v_{(1)})f^{j-2}(v_{(2)})(x))
\end{align*}
where again we use the fact that $(d_M^{j-1}d_M^{j-2})(v_{(1)}) = 0$ and the composition defined by Proposition \ref{boxrepresentations}. To simplify the notation, we subtract the initial $d_M^{j-1}(\omega_\mathtt{m})$ from the above before transforming it further:
\begin{align*}
&d_M^{j-2}(v)f^{j-2}(\omega_\mathtt{l})(x) + f^{j-1}(v) (\partial x) &\\
&- \sum_{(v)}f^{j-1}(v_{(1)})(v_{(2)}\otimes x) + \sum_{(v)} d_M^{j-2}(v_{(1)})f^{j-2}(v_{(2)})(x)&\\
&=d_M^{j-2}(v)f^{j-2}(\omega_\mathtt{l})(x) - f^{j-1}(v) d_\Box^{j-1}(\omega_\mathtt{l})(x)+f^{j-1}(\omega_\mathtt{m})d_\Box^{j-2}(v)(x) &\\
&+ f^{j-1}(v)d_\Box^{j-1}(\omega_\mathtt{l})(x) -d_M^{j-2}(\omega_\mathtt{m})f^{j-2}(v)(x) - d_M^{j-2}(v)(f^{j-2}(\omega)(x))&\\
&=f^{j-1}(\omega_\mathtt{m})(d_\Box^{j-2}(v)(x)) -d_M^{j-2}(\omega_\mathtt{m})f^{j-2}(v)(x),
\end{align*}
where the first equality follows from $d_\Box(\omega)(x) = -\partial x$ and $(\star)$, $(\star \star)$.

It shows that 
\begin{align*}
(\star \star \star) = d_M^{j-1}(\omega_\mathtt{m})f^{j-1}(\omega_\mathtt{m})(v \otimes x) - d_M^{j-1}(\omega_\mathtt{m})d_M^{j-2}(\omega_\mathtt{m})f^{j-2}(v)(x)= d_M^{j-1}f^{j-1}(\omega_\mathtt{m})(v\otimes x),
\end{align*}
where the last equation follows from the fact that $d_M$ squares to zero.
\item[(ii)/(iii)] In view of (i), it suffices to prove that $f$ is homotopic to zero if and only if $f^0(\omega_\mathtt{i})(\mathbbm{1}_\mathtt{i})=0$ in $H^0(d_M(\omega_\mathtt{i}))$. First suppose that $f$ is homotopic to zero, that is there exist maps $s^j\colon \Box_\mathtt{i}^j\to M^{j-1}$ such that $f^j=d_M^{j-1} s^j+s^{j+1}d_\Box^j$ for all $j\in \mathbb{Z}$. In particular, 
\begin{align*}
f^0(\omega_\mathtt{i})(\mathbbm{1}_\mathtt{i})&=(d_M^{-1}s^0)(\omega_\mathtt{i})(\mathbbm{1}_\mathtt{i})+(s^1d_\Box^0)(\omega_\mathtt{i})(\mathbbm{1}_\mathtt{i})\\
&=d_M^{-1}(\omega_\mathtt{i})s^0(\omega_\mathtt{i})(\mathbbm{1}_\mathtt{i})+s^1(\omega_\mathtt{i})d_\Box^0(\omega_\mathtt{i})(\mathbbm{1}_\mathtt{i})\\
&=d_M^{-1}(\omega_\mathtt{i})s^0(\omega_\mathtt{i})(\mathbbm{1}_\mathtt{i})+s^1(\omega_\mathtt{i})(0)\\
&=d_M^{-1}(\omega_\mathtt{i})s^0(\omega_\mathtt{i})(\mathbbm{1}_\mathtt{i}).
\end{align*}
Thus, $f^0(\omega_\mathtt{i})(\mathbbm{1}_\mathtt{i})=0$ in $H^0(d_M(\omega_\mathtt{i}))$. Conversely, we first consider $f$ such that $f^0(\omega_\mathtt{i})(\mathbbm{1}_\mathtt{i})=0$ in $Z^0(d_M(\omega_\mathtt{i}))$ and show that in this case there exists a homotopy $s$ between $f$ and $0$. Define the homotopy $s$ as follows: $s(v)=0$ for all $v\in Q^1_1$ and $0=s^{0}\colon A\otimes_{\mathbb{L}}L(i)\to M^{-1}$ and inductively 
\[s^{j+1}(\omega_\mathtt{m})(x\otimes y\otimes \lambda):=f^j(x)(y\otimes \lambda)-d_M^{j-1}(x)s^j(\omega_\mathtt{l})(y\otimes \lambda),\]
where $x\in \overline{V}(\mathtt{l},\mathtt{m})$ and $y\in \overline{V}^{\otimes n}(\mathtt{i},\mathtt{l})$.
First we need to check that $s^{j+1}$ defines a morphism of bocs representations. Again, we use Proposition \ref{boxrepresentations}. We omit the ``$\otimes \lambda$'' since it does not effect the calculation:
\footnotesize
\begin{align*}
s^{j+1}(\omega a-a\omega+\partial(a))(x\otimes y)&=s^{j+1}(\omega)(ax\otimes y)-as^{j+1}(\omega)(x\otimes y)\\
&=f^j(ax)(y)-d_M^{j-1}(ax)s^j(\omega)(y)-a\left(f^j(x)(y)-d^{j-1}_M(x)s^j(\omega)(y)\right)=0,
\end{align*}
\normalsize
since $f$ and $d_M$ are morphisms of $A$-bimodules.

To check the identity $f^j=s^{j+1}d_\Box^j+d_M^{j-1}s^j$ we need to apply it to $\omega_\mathtt{m}$ as well as to $v$ for $v\in \overline{V}$. For $\omega_\mathtt{m}$ we apply it to $x\otimes y\otimes \lambda$ where $x\in \overline{V}(\mathtt{l},\mathtt{m})$ and $y\in \overline{V}^{\otimes n}(\mathtt{i},\mathtt{l})$. Using Sweedler notation, we write $\partial x=\sum_{(x)} x_{(1)}\otimes x_{(2)}$. We compute each of the summands of
\[(s^{j+1}d_\Box^j+d_M^{j-1}s^j)(\omega_\mathtt{m})(x\otimes y\otimes \lambda)=s^{j+1}(\omega_\mathtt{m})d_\Box^{j}(\omega_{\mathtt{m}})(x\otimes y\otimes \lambda)+d_M^{j-1}(\omega_\mathtt{m})s^j(\omega_\mathtt{m})(x\otimes y\otimes \lambda)\]
separately. To save some space and make the equations more readable we omit the ``$\otimes \lambda$'' since it does not change throughout the whole calculation.
The first summand is equal to
\begin{align*}
&s^{j+1}(\omega_\mathtt{m})d_\Box^j(\omega_\mathtt{m})(x\otimes y)\\
&=\left(-\sum_{(x)} s^{j+1}(\omega_\mathtt{m})(x_{(1)}\otimes x_{(2)}\otimes y)\right)+s^{j+1}(\omega_\mathtt{m})(x\otimes \partial(y))\\
&=\left(-\sum_{(x)}f^j(x_{(1)})(x_{(2)}\otimes y)+d_M^{j-1}(x_{(1)})s^j(\omega)(x_{(2)}\otimes y)\right)\\
&+s^{j+1}(\omega_\mathtt{m})(x\otimes \partial(y))\\
&=\left(-\sum_{(x)}f^j(x_{(1)})(x_{(2)}\otimes y)+d^{j-1}_M(x_{(1)})\left(f^{j-1}(x_{(2)})(y)-d_M^{j-2}(x_{(2)})s^{j-1}(\omega_\mathtt{l})(y)\right)\right)\\
&+s^{j+1}(\omega_\mathtt{m})(x\otimes \partial(y))\\
&=\big(-\underbrace{\sum_{(x)} f^j(x_{(1)})(x_{(2)}\otimes y)}_{(**)}+\underbrace{d_M^{j-1}(x_{(1)})f^{j-1}(x_{(2)})(y)}_{(*)}-\underbrace{d_M^{j-1}(x_{(1)})d_M^{j-2}(x_{(2)})s^{j-1}(\omega_\mathtt{l})(y)}_{(\dagger)}\big)\\
&+\underbrace{s^{j+1}(\omega_\mathtt{m})(x\otimes \partial y)}_{(\diamond)}.
\end{align*}
The second summand equals
\begin{align*}
d_M^{j-1}(\omega_\mathtt{m})s^j(\omega_\mathtt{m})(x\otimes y)&=d_M^{j-1}(\omega_\mathtt{m})\left(f^{j-1}(x)(y)-d^{j-2}_M(x)s^{j-1}(\omega_\mathtt{l})(y)\right)\\
&=\underbrace{d_M^{j-1}(\omega_\mathtt{m})f^{j-1}(x)(y)}_{(*)}-\underbrace{d_M^{j-1}(\omega_\mathtt{m})d_M^{j-2}(x)s^{j-1}(\omega_\mathtt{l})(y)}_{(\dagger)}.
\end{align*}
As $d_M$ squares to zero, the two terms marked with $(\dagger)$ sum up to 
\begin{align*}
(\dagger)&=-(d_M^{j-1}d_M^{j-2})(x)s^{j-1}(\omega_\mathtt{l})(y)+d_M^{j-1}(x)d_M^{j-2}(\omega_\mathtt{l})s^{j-1}(\omega_\mathtt{l})(y)\\
&=d_M^{j-1}(x)d_M^{j-2}(\omega_\mathtt{l})s^{j-1}(\omega_\mathtt{l})(y).
\end{align*}
Since $f$ is a morphism of complexes, the two terms marked with $(*)$ combine to give
\begin{align*}
(*)&=(d_M^{j-1}f^{j-1})(x)(y)-d_M^{j-1}(x)f^{j-1}(\omega_\mathtt{l})(y)\\
&=\underbrace{(f^{j}d_\Box^{j-1})(x)(y)}_{(**)}-\underbrace{d_M^{j-1}(x)f^{j-1}(\omega_\mathtt{l})(y)}_{(\ddagger)}.
\end{align*}
By induction, the result of $(\dagger)$ combines with $(\ddagger)$ to give
\[(\bullet)=-d_M^{j-1}(x)s^j(\omega_\mathtt{l})d^{j-1}_\Box(\omega_\mathtt{l})(y)=d^{j-1}_M(x)s^j(\omega_\mathtt{l})(\partial y).\]
The two terms marked with $(**)$ combine to give
\begin{align*}
(**)&=f^j(\omega_\mathtt{m})d_\Box^{j-1}(x)(y)+f^j(x)d_\Box^{j-1}(\omega_\mathtt{l})(y)\\
&=f^j(\omega_\mathtt{m})(x\otimes y)+\underbrace{f^j(x)(-\partial(y))}_{(\diamond)}.
\end{align*}
The terms marked with $(\diamond)$ and $(\bullet)$ cancel off by definition of $s^{j+1}$ yielding $f^j(\omega_\mathtt{m})(x\otimes y)$ as the final result.

Since $s^j(v)=0$, applying $s^{j+1}d_\Box^j+d_M^{j-1}s^j$ to $v\in Q^1_1(\mathtt{l},\mathtt{m})$ gives:
\begin{align*}
(s^{j+1}d_\Box^j+d_M^{j-1}s^j)(v)(y\otimes \lambda)&=s^{j+1}(\omega_\mathtt{m})d_\Box^{j}(v)(y\otimes \lambda)+d_M^{j-1}(v)s^j(\omega_\mathtt{l})(y\otimes \lambda)\\
&=f^j(v)(y\otimes \lambda)-d_M^{j-1}(v)s^j(\omega_\mathtt{l})(y\otimes \lambda)+d_M^{j-1}(v)s^j(\omega_\mathtt{l})(y\otimes \lambda)\\
&=f^j(v)(y\otimes \lambda).
\end{align*}
In the general case, let $f^0(\omega_\mathtt{i})(\mathbbm{1}_\mathtt{i})=d^{-1}_M(\omega_\mathtt{i})(x)$. Then, by Lemma \ref{existenceofs}, there exists a morphism $s^{0}$ with $s^{0}(\omega_\mathtt{i})(\mathbbm{1}_\mathtt{i})=x$. Then, the morphism of complexes $g$ with $g^0=f^0-d_M^{-1}s^{0}$ and $g^j=f^j$, otherwise, is a morphism of complexes satisfying $g^0(\omega_\mathtt{i})(\mathbbm{1}_\mathtt{i})=0$. Thus, it is null-homotopic by what we have shown so far. Thus, also $f$ is null-homotopic where the previous null-homotopy is adjusted by $s^{0}$. Observing that $s^{0}$ as well as the $s^j$ constructed before can be chosen to satisfy $s^j(v)=0$ for all $v\in Q^1_1$, statement (iii) follows.
\item[(iv)] Let $M$ be acyclic, i.e. let there exist exact sequences $0\to D^j\stackrel{f^j}{\to} M^{j+1}\stackrel{g^j}{\to} D^{j+1}\to 0$ in $\modu \mathfrak{A}$ for each $j\in \mathbb{Z}$ such that the differential on $M$ is given by the composition $d_M^j \colon M^j \to D^j \to M^{j+1}$. By Remark \ref{exactstructureonbocsrep} without loss of generality one can assume that $f^j$ and $g^j$ are $A$-linear, i.e. $(f^j(\omega), g^j(\omega))$ form an exact sequence of $A$-modules. An exact sequence of $A$-modules yields an exact sequence of vector spaces for every $\mathtt{i}$. In particular, $(M,d_M(\omega_\mathtt{i}))$ is acyclic. Since $\Hom_{\mathcal{K}^b(\modu \mathfrak{A})}(\Box_\mathtt{i},M)\cong H^0(d_M(\omega_\mathtt{i}))$, it follows that $\Box_{\mathtt{i}}$ is homotopically projective.\qedhere
\end{enumerate}
\end{proof}

Applying the $\mathbbm{k}$-duality $\mathbb{D}$ on $\modu \mathfrak{A}$, which obviously extends to $\mathcal{D}^b(\modu \mathfrak{A})$, we obtain a dual statement. 

\begin{cor}
Let $\Diamond_\mathtt{i}:=\mathbb{D}(\Box^{\mathfrak{A}^{\op}}_\mathtt{i})$. Then $\Hom_{\mathcal{K}^b(\modu \mathfrak{A})}(M,\Diamond_\mathtt{i})=\mathbb{D}H^0(d_M(\omega_\mathtt{i}))$. In particular, $\Diamond_\mathtt{i}$ is a homotopically injective object in $\mathcal{D}^b(\modu \mathfrak{A})$. 
\end{cor}

We are now ready to identify the costandard modules in the bocs language. 

\begin{prop}\label{diamondcostandard}
Let $\mathfrak{A}$ be a directed normal bocs. Let $R$ be its right algebra. Let $\mathcal{F}(\Diamond)$ be the extension closure of the $\Diamond_\mathtt{i}$ in $\mathcal{D}^b(\modu \mathfrak{A})$. Then $\mathcal{F}(\Diamond)\simeq \mathcal{F}(\nabla_R)$.
\end{prop}

\begin{proof}
Recall that in the description of $\mathcal{D}^b(\modu R)$ as $\mathcal{D}^b(\modu \mathfrak{A})$ the standard modules for $R$  correspond to the simple modules $L(\mathtt{i})$ in $\modu \mathfrak{A}$ for $\mathtt{i}\in \{\mathtt{1},\dots,\mathtt{n}\}$. 
By the characterisation of the costandard modules given in Lemma \ref{costandards}, we thus have to check that  
\[\Hom_{\mathcal{D}^b(\modu \mathfrak{A})}(L(\mathtt{l}),\Diamond_\mathtt{i}[s])\cong\begin{cases} \mathbbm{k}&\text{if $s=0, \mathtt{l}=\mathtt{i}$,}\\0&\text{else.}\end{cases}\]
But this statement is true since the latter is just $\mathbb{D}H^0(d_{L(\mathtt{l})[-s]}(\omega_\mathtt{i}))$. 
\end{proof}

\begin{ex}
We continue our running example. The dual bocs is the bocs associated to the quiver opposite to (\ref{eqtn_quiver_with_3_vertices}). Hence one gets analogous complexes $\Box_\mathtt{i}^{\mathfrak{A}^\textrm{op}}$. In fact, in this example $\mathfrak{A}^\textrm{op}$ is isomorphic to $\mathfrak{A}$, hence $\Diamond_{\mathtt{i}}$ is isomorphic to the $\mathbbm{k}$-dual of $\Box_{3-\mathtt{i}}$ of the running example \ref{example}. In particular, $\Diamond_\mathtt{1}$ is concentrated in degree 0, $\Diamond_\mathtt{2}$ has two components, in degree $-1$ and 0, and $\Diamond_\mathtt{3}$ lies in degrees $-2, -1, 0$ with $\dim(\Diamond_\mathtt{3})^{-2}=1$, $\dim(\Diamond_\mathtt{3})^{-1}=4$, and $\dim(\Diamond_\mathtt{3})^{0}=4$.
\end{ex}

From now on we work in the category $\mathcal{F}(\Box)$, which seems easier to handle. The duality $\mathbb{D}$ ensures that the results transfer to $\mathcal{F}(\Diamond)$. 

\section{The category $\mathcal{F}(\Box)$}\label{sec_cat_F_Box}

Recall that for a morphism of complexes $f\colon M\to N$ its mapping cone, $\cone(f)$, is defined by $(\cone(f))^j=N^j\oplus M^{j-1}$ with differential given by $d_{\cone(f)}^j=\begin{pmatrix}d_N^j&(-1)^j f^j\\0&d_M^{j-1}\end{pmatrix}$.

\begin{thm}\label{FofBox}
Let $\mathfrak{A}=(A,V)$ be a directed normal bocs. 
\begin{enumerate}[(i)]
\item\label{filtration:i} The category $\mathcal{F}(\Box)\subset \mathcal{D}^b(\modu \mathfrak{A})$ can equivalently be described as the full subcategory whose   objects $M\in \mathcal{D}^b(\modu \mathfrak{A})$, regarded as complexes, have a ``filtration'' 
\[0=M_0\subset M_1\subset \dots\subset M_{r-1}\subset M_r=M\]
such that $M_q\cong \cone(g_q)$ for fixed $g_q\colon \Box_{\mathtt{i}_q}[-1]\to M_{q-1}$ for $q=1,\dots, r$ and $\mathtt{i}_q\in \{\mathtt{1},\dots,\mathtt{n}\}$. In this case, the number $r$ is an invariant of $M$. Furthermore, the morphism $g_q$ can be chosen to be a morphism of complexes of $A$-modules. 
\item The category $\mathcal{F}(\Box)\subset \mathcal{D}^b(\modu \mathfrak{A})$ can equivalently be described as the full subcategory with objects isomorphic to objects $N$ with 
\[N^j=\begin{cases}A\otimes_\mathbb{L} Y&\text{if $j=0$,}\\\overline{V}^{\otimes_A j} \otimes_{\mathbb{L}} Y&\text{if $j>0$,}\\0&\text{else,}\end{cases}\]
 for an $\mathbb{L}$-module $Y$ with differential given by $d_N(\omega_\mathtt{m})(x\otimes y)=-\partial(x)\otimes y+(-1)^{j}x\otimes c_Y(y)$ for $x\in \overline{V}^{\otimes_A j}$, $y\in Y$ for some $\mathbb{L}$-linear map $c_Y\colon Y\to \overline{V}\otimes_\mathbb{L} Y$ and $d_N(v)(x\otimes y)=v\otimes x\otimes y$. Furthermore, the pair $(Y,c_Y)$ can be chosen in such a way that $Y$ has a filtration $0=Y_0\subset Y_1\subset \dots \subset Y_{r-1}\subset Y_r=Y$ with $c_Y(Y_{q})\subseteq \overline{V}\otimes_{\mathbb{L}} Y_{q-1}$ for $q=1,\dots, r$.
\end{enumerate}
\end{thm}

\begin{proof}
\begin{enumerate}[(i)]
\item Defining the category of objects having such filtration, it is shown in \cite[Lemma 4.2]{MS16} that this category is closed under extensions and hence it coincides with $\mathcal{F}(\Box)$. That the $g_q$ can be chosen to belong to the class of $A$-module homomorphisms follows from the fact that by Lemma \ref{propertiesofBox} each homotopy class of morphisms $\Box_{\mathtt{i}_q}\to M_{q-1}[1]$ contains an $A$-module homomorphism and that the cones of homotopic maps are isomorphic. That the number $r$ does not depend on the choice of filtration follows from the fact that the category $\mathcal{F}(\Box)^{\op}$ is equivalent to $\mathcal{F}(\nabla)$ by Proposition \ref{diamondcostandard} and hence the filtration multiplicities are invariant, see e.g. \cite[Corollary 6.4, Proposition 5.11]{MS16}. 
\item We define $N$ inductively from $M_q$ by induction on $q$. Suppose that $N_{q-1}$ corresponding to $M_{q-1}$ was already defined as 
\[N_{q-1}^j=\begin{cases}A\otimes_{\mathbb{L}} Y_{q-1}&\text{for $j=0$,}\\\overline{V}^{\otimes_A j}\otimes_\mathbb{L} Y_{q-1}&\text{if $j>0$,}\\0&\text{else.}\end{cases}\]
By assumption, there exists $g_q \colon \Box_{i_q}[-1] \to M_{q-1}$. By Lemma \ref{propertiesofBox} \eqref{Boxhom:iii}, we can assume that $g_q$ is $A$-linear. Put $N_q:=\cone(g_q)$. Then, by definition of the cone,
\[(N_q)^{j}=\begin{cases}(A\otimes Y_{q-1})\oplus (A\otimes L(\mathtt{i}_q))&\text{for $j=0$,}\\(\overline{V}^{\otimes_A j}\otimes_{\mathbb{L}} Y_{q-1})\oplus (\overline{V}^{\otimes j}\otimes L(\mathtt{i}_q))&\text{for $j>0$,}\\0&\text{else.}\end{cases}\]
Set $Y_q:=Y_{q-1}\oplus L(\mathtt{i}_q)$. The differential on $N_q$ is given by $d_{N_q}^j=\begin{pmatrix}d_{N_{q-1}}^j&(-1)^jg_q^j\\0&d_{\Box_{\mathtt{i}_q}}^{j}\end{pmatrix}$. Observe that since $g_q^j$ is chosen to be an $A$-linear map, i.e. $g^j_q(v)=0$, one has 
\[d^j_{N_q}(v)=\begin{pmatrix}d^j_{N_{q-1}}(v)&0\\0&d^{j}_{\Box_{\mathtt{i}_q}}(v)\end{pmatrix}.\]
Thus, by induction $d^j_{N_q}(v)(x\otimes y)=v\otimes x\otimes y$ as this holds for $\Box_{\mathtt{i}}$ and $d^j_{N_q}(v)$ has block diagonal shape. 
Furthermore, by the recursive definition of the $A$-linear representative of $g^j_q$ in the proof of Lemma \ref{propertiesofBox} \eqref{Boxhom:i} (taking into account that the $g_q$ are $A$-module homomorphisms), $g^j_q(\omega_\mathtt{m})(x\otimes \lambda)=x\otimes c_{N,q}(\lambda)$ for some $c_{N,q}\colon L(\mathtt{i}_q)\to \overline{V}\otimes Y_{q-1}$. Indeed,  let $c_{N,q}(\lambda e_{\mathtt{i}_q})=\lambda g^0_q(\omega_{\mathtt{i}_q})(\mathbbm{1}_{\mathtt{i}_q})$. Then, for $j=0$, $g^0_q(\omega_\mathtt{m})(x\otimes e_{\mathtt{i}_q})=xg^0_q(\omega_{\mathtt{i}_q})(\mathbbm{1}_{\mathtt{i}_q})=x c_{N,q}(e_{\mathtt{i}_q})$ and for $j\geq 1$, $g^j_q(\omega_\mathtt{m})(v\otimes x\otimes \lambda)=d_{N_{q-1}}^{j-1}(v)g_q^{j-1}(\omega_\mathtt{l})(x\otimes \lambda)=d_{N_{q-1}}^{j-1}(v)( x\otimes c_{N,q}(\lambda))=v\otimes x\otimes c_{N,q}(\lambda)$ by induction and the form of $d^j_{N_q}$ (we use $(\star)$, $(\star \star)$ and the fact that $g^j_{q-1}(v)=0$). Defining $c_{N_q}\colon Y_q\to \overline{V}\otimes_{\mathbb{L}} Y_{q-1}$ by $c_{N_q}|_{Y_{q-1}}:=c_{N_{q-1}}$ and $c_{N_q}|_{L(\mathtt{i}_q)}:=c_{N,q}$ yields the desired map. Indeed, it remains to check the first formula
\begin{align*}
d^j_{N_q}(\omega_\mathtt{m})(x\otimes (y+\lambda))
&=d^j_{N_{q-1}}(\omega_\mathtt{m})(x\otimes y)+(-1)^jg^j_{q}(\omega_\mathtt{m})(x\otimes \lambda)+d^j_{\Box_{\mathtt{i}_q}}(\omega_\mathtt{m})(x\otimes \lambda)\\
&=-\partial(x)\otimes y+(-1)^jx\otimes c_{N_{q-1}}(y)+(-1)^jg^j_q(\omega_\mathtt{m})(x\otimes \lambda)-\partial(x)\otimes \lambda\\
&=-\partial(x)\otimes y+(-1)^jx\otimes c_{N_{q-1}}(y)+(-1)^jx\otimes c_{N,q}(\lambda)-\partial(x)\otimes \lambda\\
&=-\partial(x)\otimes (y+\lambda)+(-1)^jx\otimes c_{N_q}(y+\lambda).&\qedhere
\end{align*}
\end{enumerate}
\end{proof}

We need the following easy exercise in homological algebra.

\begin{lem}\label{homologalg}
Let $\mathcal{A}$ be an additive category. Let $\mathcal{K}^b(\mathcal{A})$ be the homotopy category of $\mathcal{A}$. Let $f\colon X\to Y$ be a morphism in $\mathcal{K}^b(\mathcal{A})$. Let $N\in \mathcal{K}^b(\mathcal{A})$. 

\begin{enumerate}[(i)]
\item Every morphism $\alpha\colon \cone(f)\to N$ can be written in the form $(g,h)$ where $g\colon Y\to N$ is a morphism of complexes and $h$ is a homotopy between $gf$ and $0$. Conversely, every such pair defines a morphism $\cone(f) \to N$.
\item Suppose that $g,\tilde{g}\colon Y\to N$ are homotopic with homotopy $\tilde{h}$. Then, $(g,h)$ and $(\tilde{g},h-\tilde{h}f)$ are homotopic via the homotopy $(\tilde{h},0)$.
\item Suppose that $(g,h),(g,\hat{h})\colon \cone(f)\to N$ are two morphisms of complexes. Then $h-\hat{h}\colon X[1]\to N$ is a morphism of complexes.
\item Let $g$, $h$ and $\hat{h}$ be as in (ii). Suppose that $h-\hat{h}$ is homotopic to $\eta$ via a homotopy $s$. Then, $(g,h)$ is homotopic to $(g,\hat{h}+\eta)$ via the homotopy $(0,s)$. 
\end{enumerate}
\end{lem}

\begin{proof}
\begin{enumerate}[(i)]
\item Here we use the description of $C:=\cone(f)$ as 
\[C^j=Y^{j}\oplus X^{j+1}\]
with differential given by $d_C^j=\begin{pmatrix}d_Y^j&f^{j+1}\\0&-d_X^{j+1}\end{pmatrix}$. Let $\alpha\colon C\to N$ be an arbitrary morphism of complexes. Since $C^j=Y^j\oplus X^{j+1}$, $\alpha^j=(g^j,h^{j+1})$ for some $g^j\colon X^j\to N^j$ and $h^{j+1}\colon Y^{j+1}\to N^j$. Since $\alpha$ is a morphism of complexes $d_N^j\alpha^j=\alpha^{j+1} d_C^j$ which is equivalent to 
\[(d_N^j g^j, d_N^j h^{j+1})=(g^{j+1}, h^{j+2})\begin{pmatrix}d_Y^{j}&f^{j+1}\\0&-d_X^{j+1}\end{pmatrix}=(g^{j+1}d_Y^j, g^{j+1}f^{j+1}-h^{j+2}d_X^{j+1}).\]
Equality in the first component means that $g=(g^j)_{j\in \mathbb{Z}}\colon Y\to N$ is a morphism of complexes. Equality in the second component means that $h=(h^j)_{j\in \mathbb{Z}}$ is a homotopy between $gf$ and $0$.
\item We compute 
\begin{align*}
d_N^{j-1}(\tilde{h}^j,0)+(\tilde{h}^{j+1},0)d_C^j&=(d_N^{j-1}\tilde{h}^j,0)+(\tilde{h}^{j+1},0)\begin{pmatrix}d_Y^j&f^{j+1}\\0&-d_X^{j+1}\end{pmatrix}\\
&=(d_N^{j-1}\tilde{h}^j+\tilde{h}^{j+1}d_Y^j, \tilde{h}^{j+1}f^{j+1})\\
&=(g^j-\tilde{g}^j,h^{j+1}-(h^{j+1}-\tilde{h}^{j+1}f^{j+1})).
\end{align*}
The claim follows.
\item Recall that the differential on $X[1]$ is given by $-d_X$. Thus,
\begin{align*}
d_N^j(h^{j+1}-\hat{h}^{j+1})-(h^{j+2}-\hat{h}^{j+2})d_{X[1]}^j
&=d_N^j(h^{j+1}-\hat{h}^{j+1})+(h^{j+2}-\hat{h}^{j+2})d_X^{j+1}\\
&=g^{j+1}f^{j+1}-g^{j+1}f^{j+1}=0.
\end{align*}
\item That $h-\hat{h}$ is homotopic to $\eta$ via a homotopy $s$ is equivalent to $d_N^{j-1}s^j+s^{j+1}d_{X[1]}^j=h^{j+1}-\hat{h}^{j+1}-\eta^j$. This follows from 
\begin{align*}
d_N^{j-1}(0,s^j)+(0,s^{j+1})\begin{pmatrix}d_Y^j&f^{j+1}\\0&-d_X^{j+1}\end{pmatrix}&=(0,d_N^{j-1}s^j-s^{j+1}d_X^{j+1})\\
&=(0,d_N^{j-1}s^j+s^{j+1}d_{X[1]}^j)\\
&=(0,h-\hat{h}-\eta).\qedhere
\end{align*}
\end{enumerate}
\end{proof}

 In view of Theorem \ref{FofBox} every object of the category $\mathcal{F}(\Box)$ gives an $\mathbb{L}$-module $Y$ together with a $\mathbb{L}$-linear map $c_Y \colon Y \to \overline{V}\otimes_{\mathbb{L}} Y$. With Theorem \ref{modN} and Proposition \ref{dualmodN} below we provide an alternative description of the category $\mathcal{F}(\Box)$. Namely, we give an equivalent condition for a map $c_Y$ to define an object of $\mathcal{F}(\Box)$ and we translate it into a condition on the 'dual' map $s_Y \colon \mathbb{D}\overline{V} \to \textrm{Hom}_\mathbbm{k}(Y,Y)$. This allows us to assign to any object of $\mathcal{F}(\Box)$ a module over an appropriate quotient of the tensor algebra $\mathbb{L}[\mathbb{D}\overline{V}]$. Together with an accurate description of morphisms in $\mathcal{F}(\Box)$ this will allow us to present $\mathcal{F}(\Diamond) \simeq \mathcal{F}(\Box)^{\textrm{op}}$ as a module over a bocs $\mathfrak{B} = (B, W)$, Theorem \ref{thm_Rigel_dual_bocs} (the algebra $B$ will be the mentioned quotient of $\mathbb{L}[\mathbb{D}\overline{V}]$). 

\begin{thm}\label{modN}
Let $\mathfrak{A}=(A,V,\mu,\varepsilon)$ be a directed normal bocs. 
\begin{enumerate}[(i)]
\item We define category $\mathcal{N}(\mathfrak{A})$ via:
\begin{description}
\item[objects] pairs $(Y,c_Y)$ where $Y$ is an $\mathbb{L}$-module and $c_Y\colon Y\to \overline{V}\otimes_{\mathbb{L}} Y$ is an $\mathbb{L}$-linear map satisfying 
\begin{equation*}\tag{$\dagger$}
(\partial_1\otimes \mathbbm{1}_Y)c_Y+(m_{\overline{V}}\otimes \mathbbm{1}_Y)(\mathbbm{1}_{\overline{V}}\otimes c_Y)c_Y=0
\end{equation*} 
such that there is a filtration $(Y_q)_{q=1,\dots,r}$ with $Y_q/Y_{q-1}\cong L(\mathtt{i}_q)$ and $c_Y|_{Y_q}\colon Y_q\to \overline{V}\otimes_{\mathbb{L}} Y_{q-1}$. 
\item[morphisms] A morphism $(Y,c_Y)\to (Z,c_Z)$ is given by a map $c_f\colon Y\to A\otimes_\mathbb{L} Z$ satisfying 
\begin{equation*}\tag{$\dagger \dagger$}
-(m_l\otimes \mathbbm{1}_Z)(\mathbbm{1}_{\overline{V}}\otimes c_Z)c_f+(\partial_0\otimes \mathbbm{1}_Z)c_f+(m_r\otimes \mathbbm{1}_Z)(\mathbbm{1}_{\overline{V}}\otimes c_f)c_Y=0.
\end{equation*}
\item[composition] Given $c_g\colon (Y,c_Y)\to (Z,c_Z)$ and $c_f\colon (X,c_X)\to (Y,c_Y)$, the map corresponding to their composition is obtained by $c_{gf}:=(m_A\otimes \mathbbm{1}_Z)(\mathbbm{1}_A\otimes c_g)c_f$. 
\item[unit] $c_{\mathbbm{1}}\colon Y\to A\otimes_{\mathbb{L}} Y$ is given by $y\mapsto 1\otimes y$. 
\end{description}
\item The categories $\mathcal{F}(\Box)$ and $\mathcal{N}(\mathfrak{A})$ are equivalent.
\end{enumerate}
\end{thm}

\begin{proof}
We only prove (ii), (i) follows by transport of structure. We define a functor $\Xi\colon \mathcal{N}(\mathfrak{A})\to \mathcal{F}(\Box)$. The pair $(Y,c_Y)$ is sent to the complex $\Xi(Y)$ with 
\[(\Xi(Y))^j:=\begin{cases}A\otimes_\mathbb{L} Y&\text{for $j=0$}\\\overline{V}^{\otimes_A j}\otimes_\mathbb{L} Y&\text{for $j>0$}\\0&\text{else}\end{cases}\] 
and differential given by $d_{\Xi(Y)}^j(\omega_\mathtt{m})(x\otimes y)=-\partial(x)\otimes y+(-1)^j x\otimes c_Y(y)$ and $d_{\Xi(Y)}^j(v)(x\otimes y)=v\otimes x\otimes y$ for $v\in Q_1^1(\mathtt{m},\mathtt{l})$, $x\in \overline{V}^{\otimes_A j}(\mathtt{i},\mathtt{m})$, and $y\in Y$. Furthermore, for a morphism $c_f\colon Y\to A\otimes_{\mathbb{L}} Z$ define $\Xi(c_f)$ to be the $A$-linear map of complexes with $\Xi(c_f)^j(x\otimes y):=(\mathbbm{1}_{\overline{V}^{\otimes_A(j-1)}}\otimes m_r\otimes \mathbbm{1}_Z)(x\otimes_A c_f(y))$.

We first prove that this functor is well-defined. For this, we first show that $d_{\Xi(Y)}^j$ is a morphism in $\modu \mathfrak{A}$. Indeed, for $a\in A(\mathtt{l},\mathtt{m})$,
\begin{multline*}
d_{\Xi(Y)}^j(\omega_\mathtt{m} a-a\omega_{\mathtt{l}}+\partial a)(x\otimes y)\\=-\partial(ax)\otimes y+(-1)^j ax\otimes c_Y(y)+a\partial(x)\otimes y-(-1)^jax\otimes c_Y(y)+\partial(a)\otimes x\otimes y=0
\end{multline*}
by the Leibniz rule for $\partial$. To prove that $d_{\Xi(Y)}$ defines a differential, write $c_Y(y)=\sum_{(y)} v_{(1)}\otimes y_{(2)}$ with $v_{(1)}\in \overline{V}$ and $y_{(2)}\in Y$. Then,
\[
(d_{\Xi(Y)}^{j+1}d_{\Xi(Y)}^j)(\omega_\mathtt{m})(x\otimes y)
=d_{\Xi(Y)}^{j+1}(\omega_\mathtt{m})(-\partial(x)\otimes y+(-1)^jx\otimes c_Y(y)).\]
The first summand is equal to
\begin{align*}
d_{\Xi(Y)}^{j+1}(\omega_\mathtt{m})(-\partial(x)\otimes y)&=\partial^2(x)\otimes y-(-1)^{j+1}\partial(x)\otimes c_Y(y)\\
&=(-1)^{j+2}\partial(x)\otimes c_Y(y).
\end{align*}
The second summand equals
\begin{align*}
&d_{\Xi(Y)}^{j+1}(\omega_\mathtt{m})((-1)^jx\otimes c_Y(y))\\
&=\sum_{(y)}(-1)^j(-\partial(x\otimes v_{(1)})\otimes y_{(2)}+(-1)^{2j+1}x\otimes v_{(1)}\otimes c_Y(y_{(2)}))\\
&=\sum_{(y)}(-1)^{j+1}\partial(x)\otimes v_{(1)}\otimes y_{(2)}+(-1)^{2j+1}x\otimes \partial(v_{(1)})\otimes y_{(2)}+(-1)^{2j+1}x\otimes v_{(1)}\otimes c_Y(y_{(2)})\\
&=\sum_{(y)}(-1)^{j+1}\partial(x)\otimes v_{(1)}\otimes y_{(2)}=(-1)^{j+1}\partial(x)\otimes c_Y(y).
\end{align*}
The third equality follows from $(\dagger)$. 
Thus, the two summands cancel each other. 

Now we check that
\begin{itemize}
	\item $d_N$ is a differential if and only if condition $(\dagger)$ is satisfied,
	\item $f$ is a morphism of complexes if and only if $c_f$ satisfies $(\dagger \dagger)$.
\end{itemize}

Writing $\partial(v)=\sum_{(v)} v_{(1)}\otimes v_{(2)}$ with $v_{(1)}, v_{(2)}\in \overline{V}$, we get 
\small
\begin{multline*}
(d_{\Xi(Y)}^{j+1}d_{\Xi(Y)}^j)(v)(x\otimes y)\\=d_{\Xi(Y)}^{j+1}(\omega_\mathtt{m})(v\otimes x\otimes y)+d_{\Xi(Y)}^{j+1}(v)(-\partial(x)\otimes y+(-1)^j x\otimes c_Y(y))+\sum_{(v)}d_{\Xi(Y)}^{j+1}(v_{(1)})d_{\Xi(Y)}^j(v_{(2)})(x\otimes y).
\end{multline*}\normalsize
The first summand equals
\[d_{\Xi(Y)}^{j+1}(\omega_\mathtt{m})(v\otimes x\otimes y)=-\partial(v\otimes x)\otimes y+(-1)^{j+1}(v\otimes x\otimes c_Y(y)),\]
the second
\[d_{\Xi(Y)}^{j+1}(v)(-\partial(x)\otimes y+(-1)^{j}x\otimes c_Y(y))=-v\otimes \partial(x)\otimes y+(-1)^jv\otimes x\otimes c_Y(y),\]
the third
\[d_{\Xi(Y)}^{j+1}(v_{(1)})d_{\Xi(M)}^j(v_{(2)})(x\otimes y)=\partial(v)\otimes x\otimes y.\]
The respective first parts of each of these three summands cancel each other as $\partial$ satisfies the graded Leibniz rule. The remaining parts of the first and the second summand also cancel each other.

As $-\otimes_{\mathbb{L}}-$ is an exact functor, the filtration $(Y_q)_{q=1,\dots,r}$ induces a filtration of $\Xi(Y)$ in $\mathcal{C}^b(\modu \mathfrak{A})$. The subquotient $\Xi(Y_{q})/\Xi(Y_{q-1})$ has $\overline{V}^{\otimes_A j} \otimes_{\mathbb{L}} Y_q/Y_{q-1}$ in degree $j\geq 0$. Since $c_Y(Y_{q})\subset \overline{V}\otimes Y_{q-1}$, the differential on the subquotient is given by $d(\omega)(x\otimes y)= -\partial(x)\otimes y$, $d(v)(x\otimes y) =x \otimes x \otimes y$, i.e. $Y_q/Y_{q-1}\cong \Box_{\mathtt{i}_q}$.  This shows that the functor is well-defined on objects. 

For checking well-definedness on morphisms write $c_Y(y)=\sum_{(y)} v_{(1)} \otimes y_{(2)}$ as well as $c_f(y)=\sum_{(y)} a_{(1)}\otimes z_{(2)}$ with $v_{(1)}\in \overline{V}$, $a_{(1)}\in A$, $y_{(2)}\in Y$, and $z_{(2)}\in Z$. Then, for a morphism $c_f\colon Y\to A\otimes Z$,
\footnotesize\begin{align*}
(\Xi(c_f)^{j+1}d_{\Xi(Y)}^j)(\omega_\mathtt{m})(x\otimes y)&=\Xi(c_f)^{j+1}(\omega_\mathtt{m})(-\partial(x)\otimes y+(-1)^jx\otimes c_Y(y))\\
&=\underbrace{-\partial(x)a_{(1)}\otimes z_{(2)}}_{(\diamond)}+\underbrace{\sum_{(y)}(-1)^j(\mathbbm{1}_{\overline{V}^{\otimes_A(j-1)}}\otimes m_r\otimes \mathbbm{1}_Z)x\otimes v_{(1)}\otimes c_f(y_{(2)}).}_{(*)}
\end{align*}\normalsize
On the other hand,
\small
\begin{align*}
(d_{\Xi(Z)}^j\Xi(c_f)^{j})(\omega_{\mathtt{m}})(x\otimes y)&=\sum_{(y)}d_{\Xi(Z)}(\omega_\mathtt{m})(xa_{(1)}\otimes z_{(2)})\\
&=\sum_{(y)}-\partial(xa_{(1)})\otimes z_{(2)}+(-1)^jxa_{(1)}\otimes c_Y(z_{(2)})\\
&=\underbrace{-\partial(x)\circ c_f(y)}_{(\diamond)}+\underbrace{\sum_{(y)}(-1)^{j+1}x\otimes \partial(a_{(1)})\otimes z_{(2)}+(-1)^jx a_{(1)}\otimes c_Y(z_{(2)}).}_{(*)}
\end{align*}
The two parts labelled $(\diamond)$ are equal, the parts labelled $(*)$ are equal by the defining property of $c_f$ to be a morphism in $\mathcal{N}(\mathfrak{A})$. 
\normalsize
Furthermore note that,
\begin{align*}
(\Xi(c_f)^{j+1}d_{\Xi(Y)}^j)(v)(x\otimes y)&=\Xi(c_f)(\omega_\mathtt{m})(v\otimes x\otimes y)\\
&=v\otimes x\otimes c_f(y)
\end{align*}
and
\begin{align*}
(d_{\Xi(Z)}^j\Xi(c_f))(v)(x\otimes y)&=d_{\Xi(Z)}(v)(x\otimes c_f(y))\\
&=v\otimes x\otimes c_f(y).
\end{align*}
To check that $\Xi$ is a functor let $c_f\colon (X,c_X)\to (Y,c_Y)$ and $c_g\colon (Y,c_Y)\to (Z,c_Z)$. Then,
\begin{align*}
\Xi(c_g)^j\circ \Xi(c_f)^j&=(\mathbbm{1}_{\overline{V}^{\otimes_A (j-1)}}\otimes m_r\otimes \mathbbm{1}_Z)(\mathbbm{1}_{\overline{V}^{\otimes_A j}}\otimes c_g)(\mathbbm{1}_{\overline{V}^{\otimes_A (j-1)}}\otimes m_r\otimes \mathbbm{1}_Y)(\mathbbm{1}_{\overline{V}^{\otimes_A j}} \otimes c_f)\\
&=(\mathbbm{1}_{\overline{V}^{\otimes_A (j-1)}}\otimes m_r\otimes \mathbbm{1}_Z)(\mathbbm{1}_{\overline{V}^{\otimes_A(j-1)}}\otimes m_r\otimes \mathbbm{1}_Z)(\mathbbm{1}_{\overline{V}^{\otimes_A j}\otimes_{\mathbbm{L}} A}\otimes c_g)(1_{\overline{V}^{\otimes_A j}}\otimes c_f)\\
&=(\mathbbm{1}_{\overline{V}^{\otimes_A (j-1)}}\otimes m_r\otimes \mathbbm{1}_Z)(1_{\overline{V}^{\otimes_A j}}\otimes m_A\otimes \mathbbm{1}_Z)(\mathbbm{1}_{\overline{V}^{\otimes_A j}\otimes_{\mathbbm{L}} A}\otimes c_g)(1_{\overline{V}^{\otimes_A j}}\otimes c_f)\\
&=\Xi(c_{gf})^j
\end{align*}
and for $x\in \overline{V}^{\otimes_A j}$, $y\in Y$,
\begin{align*}
\Xi(c_{\mathbbm{1}})(x\otimes y)&=(\mathbbm{1}_{\overline{V}^{\otimes_A (j-1)}}\otimes m_r\otimes \mathbbm{1}_Z)(x\otimes c_{\mathbbm{1}}(y))&=(\mathbbm{1}_{\overline{V}^{\otimes_A (j-1)}}\otimes m_r\otimes \mathbbm{1}_Z)(x\otimes 1\otimes y)\\
&=x\otimes y.
\end{align*}
To check that $\Xi$ is an equivalence, note that by Theorem \ref{FofBox}, for each $M\in \mathcal{D}^b(\modu \mathfrak{A})$ there is an isomorphic $N=\Xi(Y)$ with a map $c_Y\colon Y\to \overline{V}\otimes_{\mathbb{L}} Y$. Then, the previous calculations show that $c_Y$ satisfies the defining property of $(Y,c_Y)$ to be an object of $\mathcal{N}(\mathfrak{A})$ as $d_N$ is a differential. 

That $\Xi$ is faithful follows from the fact that $c_f(y)=\Xi(c_f)^0(1\otimes y)$. Thus, $\Xi(c_f)=0$ implies that $c_f=0$. We are left with proving that $\Xi$ is full. For this let $f=(f^j)_{j\in \mathbb{Z}}$ be an arbitrary morphism $\Xi(Y)\to \Xi(Z)$. We prove by induction on the length $r$ of the filtration that $f$ is represented by an $A$-linear map. For $r=1$ this is the content of Lemma \ref{propertiesofBox}. For the induction step, note that $\Xi(Y_q)$ is constructed as the cone of a morphism $g_q\colon \Box_{\mathtt{i}_q}[-1]\to  \Xi(Y_{q-1})$. Thus, by the Lemma \ref{homologalg}, each morphism $\alpha\colon \Xi(Y_q)\to \Xi(Z)$ can be represented by a pair $(g,h)$ where $g\colon \Xi(Y_{q-1})\to \Xi(Z)$ is a morphism of complexes and $h\colon \Box_{\mathtt{i}_q}[-1] \to \Xi(Z)$ is a homotopy between $gg_q$ and $0$. By induction, $g$ is homotopic to an $A$-linear $\tilde{g}$ via a homotopy $\tilde{h}$. Thus, $(g,h)$ is homotopic to $(\tilde{g},h-\tilde{h}g_q)$. Let $\hat{h}$ be an $A$-linear homotopy between $\tilde{g}g_q$ and $0$ which can
  be chosen by Lemma \ref{propertiesofBox}. Then, by the foregoing lemma, $(\tilde{g},\hat{h})$ is a morphism $\Xi(Y_q)\to \Xi(Z)$ as well. Again invoking the foregoing lemma, $h-\tilde{h}g_q-\hat{h}$ is a morphism of complexes. By Lemma \ref{propertiesofBox}, there exists an $A$-linear map $\eta$ homotopic to $h-\tilde{h}g_q-\hat{h}$. Furthermore the previous lemma also implies that $(\tilde{g},h-\tilde{h}g_q)$ is homotopic to $(\tilde{g},\hat{h}+\eta)$, which is $A$-linear.

We have to show that the $f^j$ can be chosen (within their homotopy class) to satisfy $f^j(x\otimes y)=x\otimes_A c_f(y)$. We proceed by induction. Let $c_f(y):=f^0(1\otimes y)$. Then, by $A$-linearity of $f$, $f^0(a\otimes y)=ac_f(y)$. Recalling the recursive definition of $f^j$ in the proof of Lemma \ref{propertiesofBox}, it follows that $f^j(v\otimes x\otimes y)=v\otimes x\otimes c_f(y)$ as $f^j$ is $A$-linear and $d_{\Xi(Z)}(v)(x\otimes z)=v\otimes x\otimes z$. Defining $c_f$ in this way, the foregoing calculations show that in order for $f$ to be a morphism of complexes, the defining property for $c_f$ has to be satisfied. 
\end{proof}

\begin{ex}\label{example2}
In the running example, consider an $\mathbb{L}$-module $Y$ with $Y_1\cong L(\mathtt{2})$, $Y_2/Y_1 \cong L(\mathtt{3})$, $Y_3/Y_2 \cong L(\mathtt{1})$ and $Y/Y_3 \cong L(\mathtt{3})$. We choose a basis $v_1,v_2,v_3,w_3$ of $Y$, with $v_i$ supported at the vertex $\mathtt{i}$ and $Y_3$ spanned by $v_1, v_2, v_3$. Then $\overline{V} \otimes_{\mathbb{L}} Y$  is a left $\mathbb{L}$-module supported at the vertices $\mathtt{2}$ and $\mathtt{3}$. The space $e_\mathtt{2} \overline{V}\otimes_{\mathbb{L}} Y$ is spanned by $\varphi \otimes v_1$, while $e_\mathtt{3} \overline{V} \otimes_{\mathbb{L}} Y = \textrm{Span }\{\psi \otimes v_3, \, \psi \otimes w_3,\, \chi \otimes v_1,\, b\varphi \otimes v_1,\, \psi a \otimes v_1\}$.

The map $c_Y \colon Y \to \overline{V}\otimes_{\mathbb{L}} Y$ defined by $c_Y(v_1)=0$, $c_Y(v_2)=0$, $c_Y(v_3)=\psi \otimes v_2$, $c_Y(w_3)=\psi \otimes v_2$ maps $Y_q$ to $\overline{V} \otimes Y_{q-1}$ and satisfies condition $(\dagger)$. Therefore, it defines a complex $N^\bcdot$ in $\mathcal{F}(\Box)$. In fact, one can check that it is the complex $\mathbb{S}(\nabla(\mathtt{1})\oplus \nabla(\mathtt{2})\oplus I(\mathtt{2}))$ where $\mathbb{S}$ denotes the Serre functor of $\mathcal{D}^b(\modu R)$ and $I(\mathtt{2})$ is the injective envelope of $\nabla(\mathtt{2})$. 

On the other hand, map $\wt{c}_Y$ equal to $c_Y$ on $v_1$, $v_2$ and $v_3$ and such that $\wt{c}_Y(w_3) = \chi \otimes v_1$ does not satisfy $(\dagger)$:
$$
(\partial_1 \otimes \mathbbm{1}_Y)\widetilde{c}_Y(w_3) + (m_{\overline{V}} \otimes \mathbbm{1}_Y)(\mathbbm{1}_{\overline{V}} \otimes \widetilde{c}_Y)\widetilde{c}_Y(w_3) = \psi \varphi \otimes v_1 + 0 \neq 0. 
$$ 
Thus the composition of maps $N^0 \xrightarrow{d^0} N^1 \xrightarrow{d^1} N^2$ defined by $\wt{c}_Y$ is non-zero. Indeed:
\begin{align*}
d^2(\omega_\mathtt{i})(e_\mathtt{3} \otimes w_3) &= d(\omega_\mathtt{i})(-\partial(e_\mathtt{3})\otimes w_3 - e_\mathtt{3} \otimes \wt{c}_Y(w_3)) = -d(\omega_\mathtt{i})(e_\mathtt{3} \otimes_A \chi\otimes v_1) \\
&=  \partial(\chi) \otimes v_1 - \chi \otimes \wt{c}_Y(v_1) = \psi \otimes \varphi \otimes v_1 \neq 0.
\end{align*}
A morphism between two objects of $\mathcal{F}(\Box)$ defined by $\mathbb{L}$-modules $Y$ and $Z$ is given by a map $c_f \colon Y \to A \otimes_{\mathbb{L}}Z$ satisfying $(\dagger \dagger)$. 
\end{ex}

\section{Dualising $\mathcal{N}(\mathfrak{A})$}\label{dualising}

In this section, we dualise the definition of $\mathcal{N}(\mathfrak{A})$ by replacing the map $c_V\colon Y\to \overline{V}\otimes_{\mathbb{L}} Y$ with a map $s_Y\colon \mathbb{D}\overline{V}\to \Hom_{\mathbbm{k}}(Y,Y)$. For two finite dimensional $\mathbb{L}$-modules $X,Y$ let $p_{X,Y}\colon \mathbb{D}(X\otimes_{\mathbbm{L}} Y)\to \mathbb{D}Y\otimes_\mathbb{L} \mathbb{D}X$ be the canonical isomorphism. Furthermore note that $\mathbb{L}\cong \prod_{\mathtt{i}=\mathtt{1}}^{\mathtt{n}}\mathbbm{k}$ as algebras, and hence every $\mathbb{L}$-module $U$ is projective. In particular, we can fix dual bases $x^s$ of $U$ and $\xi^s$ of $\mathbb{D}U$, i.e. bases such that $\xi^s(x^t)=\delta_{st}$.

\begin{lem}\label{psiisomorphisms}
Let $Y,Z\in \modu \mathbb{L}$. Let $U$ be an $\mathbb{L}$-module with basis $x^s$ and its dual $\xi^s\in\mathbb{D}U$. 
\begin{enumerate}[(i)] 
\item\label{psiisomorphisms:i} There is an isomorphism
\[\Hom_{\mathbb{L}\otimes \mathbb{L}}(\mathbb{D}U,\Hom_{\mathbbm{k}}(Y,Z))\cong \Hom_{\mathbb{L}}(Y,U\otimes_\mathbb{L} Z).\]
This isomorphism is given from left to right by $\Phi^U(f)(y):= \sum_{s} x^s\otimes f(\xi^s)(y)$ for $f\in \Hom_{\mathbb{L}\otimes \mathbb{L}}(\mathbb{D}U,\Hom_{\mathbbm{k}}(Y,Z))$ and $y\in Y$ and in the other direction by $\Psi^U(g)(\varphi)(y):=m(\varphi\otimes \mathbbm{1}_Z)(g(y))$, where $g\in \Hom_{\mathbb{L}}(Y,U\otimes_{\mathbb{L}} Z), \varphi\in \mathbb{D}U, y\in Y$, and $m$ is the canonical identification $\mathbbm{k}\otimes_{\mathbbm{k}} Z\to Z$. 
\item\label{psiisomorphisms:ii} The above isomorphism is functorial in the sense that for each morphism $h\colon U\to U'$ of $\mathbb{L}$-modules $\Psi^{U'}((h\otimes \mathbbm{1}_Y)f)=\Psi^U(f)\circ \mathbb{D}h$. 
\item\label{psiisomorphisms:iii} Furthermore, the isomorphism is compatible with the standard adjunction in the following sense: Let $\alpha\colon \Hom_{\mathbb{L}\otimes \mathbb{L}}(\mathbb{D}U',\Hom_{\mathbb{L}\otimes \mathbb{L}}(\mathbb{D}U,\Hom_{\mathbbm{k}}(Y,Z)))\to \Hom_{\mathbb{L}}(\mathbb{D}U\otimes_{\mathbb{L}}\mathbb{D}U',\Hom_{\mathbbm{k}}(Y,Z))$ be the standard bimodule tensor-hom adjunction. Then, 
\[\Hom(p_{U',U},\Hom_{\mathbbm{k}}(Y,Z))\alpha\Hom_{\mathbb{L}\otimes \mathbb{L}}(\mathbb{D}U',\Psi^U)\Psi^{U'}=\Psi^{U'\otimes U}.\]
\item\label{psiisomorphisms:iv} Finally, the isomorphism is compatible with composition in the following sense:
\[\alpha\Hom_{\mathbb{L}\otimes \mathbb{L}}(\mathbb{D}U',\Psi^{U})\Psi^{U'}((1\otimes g) f)=m_{\mathbb{L}}(\Psi^U \otimes \Psi^{U'})(g\otimes f).\]
\end{enumerate}
\end{lem}

\begin{proof}
\begin{enumerate}[(i)]
\item To check that these define inverse equivalences note that
\begin{align*}
\Psi^U\Phi^U(f)(\varphi)(y)&=m(\varphi\otimes \mathbbm{1}_Z)(\Phi^U(f)(y))&=\sum_s m(\varphi\otimes \mathbbm{1}_Z)(x^s\otimes f(\xi^s)(y))\\
&=\sum_s m(\varphi(x^s)\otimes f(\xi^s)(y))
&=\sum_s m(1\otimes f(\varphi(x^s)\xi^s)(y)\\
&=f(\varphi)(y)
\end{align*}
and
\begin{align*}
\Phi^U\Psi^U(g)(y)&=\sum_sx^s\otimes \Psi^U(g)(\xi^s)(y)
&=\sum_sx^s\otimes m(\xi^s\otimes \mathbbm{1})g(y)\\
&=\sum_s\sum_{(y)} x^s\xi^s(u_{(1)})\otimes z_{(2)}
&=\sum_{(y)} u_{(1)}\otimes z_{(2)}\\
&=g(y)
\end{align*}
where, using Sweedler notation, $g(y)=:\sum_{(y)}u_{(1)}\otimes z_{(2)}$.
\item Using that $\mathbb{D}h(\varphi)=\varphi\circ h$ for $\varphi\in \mathbb{D}U'$, we obtain for $y\in Y$ and $f\colon Y\to U\otimes Z$:
\begin{align*}
(\Psi^U(f)\circ \mathbb{D}h)(\varphi)(y)&=\Psi^U(f)(\varphi\circ h)(y)\\
&=m((\varphi\circ h\otimes \mathbbm{1}_Z)f(y)\\
&=m(\varphi\otimes \mathbbm{1}_Z)(h\otimes \mathbbm{1}_Z)f(y)\\
&=\Psi^{U'}((h\otimes \mathbbm{1}_Z)f)(\varphi)(y).
\end{align*}
\item To check equality we apply $\Psi^{U'\otimes U}$ to some $g\in \Hom_{\mathbb{L}}(Y,U'\otimes_{\mathbb{L}} U\otimes_{\mathbb{L}} Z)$ the result to $p_{U',U}^{-1}(\varphi\otimes \varphi')$ where $\varphi\in \mathbb{D}U$ and $\varphi'\in \mathbb{D}U'$, and finally the result to $y\in Y$:
\small
\begin{align*}
\Psi^{U'\otimes U}(g)(p_{U',U}^{-1}(\varphi\otimes \varphi'))(y)&=m(p_{U',U}^{-1}(\varphi\otimes \varphi')\otimes \mathbbm{1}_Z)g(y)\\
&=m(\varphi\otimes \mathbbm{1}_Z)(m(\varphi'\otimes \mathbbm{1}_Z)g(y))\\
&=m(\varphi\otimes \mathbbm{1}_Z)(\Psi^{U'}(g)(\varphi')(y))\\
&=\Psi^U(\Psi^{U'}(g)(\varphi'))(\varphi)(y)\\
&=\Hom(\mathbb{D}U',\Psi^U)\Psi^{U'}(g)(\varphi')(\varphi)(y)\\
&=\alpha\Hom(\mathbb{D}U',\Psi^U)\Psi^{U'}(g)(\varphi\otimes \varphi')(y)\\
&=\Hom(p_{U',U},\Hom(Y,Z))\alpha\Hom(\mathbb{D}U',\Psi^U)\Psi^{U'}(g)(p_{U',U}^{-1}(\varphi\otimes \varphi'))(y).
\end{align*}\normalsize
\item Let $\varphi\in \mathbb{D}U, \varphi'\in \mathbb{D}U'$, $x\in X$. Then,
\begin{align*}
\alpha\Hom(\mathbb{D}U',\Psi^U)\Psi^{U'}(1\otimes g)f(\varphi\otimes \varphi')(x)&=\Hom(\mathbb{D}U',\Psi^U)\Psi^{U'}((1\otimes g)f)(\varphi')(\varphi)(x)\\
&=\Psi^U(\Psi^{U'}((1\otimes g)f)(\varphi')(\varphi)(x)\\
&=\Psi^U(m(\varphi'\otimes 1)((1\otimes g)f))(\varphi)(x)\\
&=m(\varphi\otimes \mathbbm{1})m(\varphi'\otimes \mathbbm{1})(1\otimes g)f(x)\\
&=m(\varphi\otimes \mathbbm{1})(g)m(\varphi'\otimes \mathbbm{1})(f)(x)\\
&=m_\mathbb{L}(\Psi^U(g)\otimes \Psi^{U'}(f)(\varphi\otimes \varphi')(x).\qedhere
\end{align*}
\end{enumerate}
\end{proof}

\begin{prop}\label{dualmodN}
Let $\mathfrak{A}=(A,V,\mu,\varepsilon)$ be a directed bocs. 
\begin{enumerate}[(i)]
\item The following defines a category $\mathcal{R}(\mathfrak{A})$:
\begin{description}
\item[objects] pairs $(Y,s_Y)$ such that $Y$ is an $\mathbb{L}$-module and \[s_Y\in \Hom_{\mathbb{L}\otimes \mathbb{L}}(\mathbb{D}\overline{V},\Hom_{\mathbbm{k}}(Y,Y))\] satisfies 
\begin{equation}\tag{$\dagger^*$}
s_Y\mathbb{D}\partial + m(s_Y\otimes s_Y)p_{\overline{V},\overline{V}} \mathbb{D}m_{\overline{V}}=0,
\end{equation}
where $m$ denotes the composition of morphisms $m\colon \Hom_{\mathbbm{k}}(Y,Y) \otimes_{\mathbbm{k}} \Hom_{\mathbbm{k}}(Y,Y) \to \Hom_{\mathbbm{k}}(Y,Y)$.
\item[morphisms] for two objects $(Y,s_Y), (Z,s_Z)$ a morphism is given by an element $s_f\in \Hom_{\mathbb{L}\otimes \mathbb{L}}(\mathbb{D}A,\Hom_{\mathbbm{k}}(Y,Z))$ satisfying 
\begin{equation}\tag{$\dagger \dagger ^*$}
m((s_f\otimes s_Y)p_{\overline{V},A}\mathbb{D}m_r-(s_Z\otimes s_f)p_{A,\overline{V}} \mathbb{D}m_l)+s_f\mathbb{D}\partial=0.
\end{equation} 
\item[composition] for two morphisms $s_f\colon (X,s_X)\to (Y,s_Y)$ and $s_g\colon (Y,s_Y)\to (Z,s_Z)$ their composition is given by $s_{gf}:=m(s_g\otimes s_f)p_{A,A}\mathbb{D}m_A$.
\item[unit] Let $s_{\mathbbm{1}}\colon (Y,s_Y)\to (Y,s_Y)$ be the map defined by $s_{\mathbbm{1}}(v)(y)=v(1)y$. 
\end{description}
\item The categories $\mathcal{N}(\mathfrak{A})$ and $\mathcal{R}(\mathfrak{A})$ are equivalent.
\end{enumerate}
\end{prop}

\begin{proof}
We only prove (ii), (i) follows by transport of structure. According to the previous lemma, there is an isomorphism 
\[\Psi^{\overline{V}}\colon \Hom_{\mathbb{L}}(Y,\overline{V}\otimes_{\mathbb{L}}Y)\to \Hom_{\mathbb{L}\otimes \mathbb{L}}(\mathbb{D}\overline{V},\Hom_{\mathbbm{k}}(Y,Y)).\]
Define the functor $\Psi\colon \mathcal{N}(\mathfrak{A})\to \mathcal{R}(\mathfrak{A})$ on objects by $F((Y,c_Y))=(Y,\Psi^{\overline{V}}(c_Y))$. Again invoking the foregoing lemma, for each two $\mathbb{L}$-modules $Y,Z$ there is an isomorphism
\[\Psi^A\colon \Hom_{\mathbb{L}}(Y,A\otimes_{\mathbb{L}} Z)\to \Hom_{\mathbb{L}\otimes \mathbb{L}}(\mathbb{D}A, \Hom_{\mathbbm{k}}(Y,Z)).\] 
Define $\Psi$ on morphisms by $F(c_f)=\Psi^A(c_f)$ for a morphism $c_f\colon (Y,c_Y)\to (Z,c_Z)$. We have to prove that this defines a functor.
Recall that the condition for $(Y,c_Y)$ to be an element of $\mathcal{N}(\mathfrak{A})$ is
\[(\partial_1\otimes \mathbbm{1}_Y)c_Y+(m_{\overline{V}}\otimes \mathbbm{1}_Y)(\mathbbm{1}_{\overline{V}}\otimes c_Y)c_Y=0.\]
We examine the two summands separatly. Applying $\Psi^{\overline{V}\otimes_A\overline{V}}$ to the first summand yields
\[\Psi^{\overline{V}\otimes_A \overline{V}}((\partial\otimes \mathbbm{1}_Y)c_Y=s_Y\circ \mathbb{D}\partial_1\]
according to Lemma \ref{psiisomorphisms} \eqref{psiisomorphisms:ii}. For the second summand, we obtain:
\footnotesize
\begin{align*}
\Psi^{\overline{V}\otimes_A \overline{V}}((m_{\overline{V}}\otimes \mathbbm{1}_Y)(\mathbbm{1}_{\overline{V}}\otimes c_Y)c_Y)&=\Psi^{\overline{V}\otimes_{\mathbb{L}} \overline{V}}((\mathbbm{1}_{\overline{V}}\otimes c_Y)c_Y)\circ \mathbb{D}m_{\overline{V}}\\
&=(\Hom(p_{\overline{V},\overline{V}},\Hom(Y,Y))\alpha\Hom(\mathbb{D}\overline{V},\Psi^{\overline{V}})\Psi^{\overline{V}}(\mathbbm{1}_{\overline{V}}\otimes c_Y)c_Y) \mathbb{D}m_{\overline{V}}\\
&=(\alpha\Hom(\mathbb{D}\overline{V},\psi^{\overline{V}})\Psi^{\overline{V}}((\mathbbm{1}_{\overline{V}}\otimes c_Y)c_Y)p_{\overline{V},\overline{V}}\mathbb{D}m_{\overline{V}}\\
&=m_{\mathbb{L}}(\Psi^{\overline{V}}(c_Y)\otimes \Psi^{\overline{V}}(c_Y))p_{\overline{V},\overline{V}}\mathbb{D}m_{\overline{V}}\\
&=m_{\mathbb{L}}(s_Y\otimes s_Y)p_{\overline{V},\overline{V}}\mathbb{D}m_{\overline{V}}.
\end{align*}\normalsize
Here we apply Lemma \ref{psiisomorphisms} \eqref{psiisomorphisms:i}, \eqref{psiisomorphisms:ii}, and \eqref{psiisomorphisms:iv} to obtain the first, second, and third equality, respectively. 
Altogether, the condition translates to the stated equality in $\mathcal{R}(\mathfrak{A})$.

For the morphisms, the condition for $c_f\colon (Y,c_Y)\to (Z,c_Z)$ to be a morphism is
\[-(m_l\otimes \mathbbm{1}_Z)(\mathbbm{1}_A\otimes c_Z)c_f+(\partial_0\otimes \mathbbm{1}_Z)c_f+(m_r\otimes \mathbbm{1}_Z)(\mathbbm{1}_{\overline{V}}\otimes c_f)c_Y=0.\]
We apply $\Psi^{\overline{V}}$ and consider the three summands separately. For the second summand, applying Lemma \ref{psiisomorphisms} \eqref{psiisomorphisms:ii}, we obtain $s_f\mathbb{D}\partial_0$. For the first summand, we obtain
\begin{align*}
\Psi^{\overline{V}}(-(m_l\otimes \mathbbm{1}_Z)(\mathbbm{1}_A\otimes c_Z)c_f)&=-\Psi^{A\otimes \overline{V}}((\mathbbm{1}_A\otimes c_Z)c_f)\mathbb{D}m_l\\
&= -m_{\mathbb{L}}(s_Z\otimes s_f)p_{A,\overline{V}}\mathbb{D}m_l,
\end{align*}
where Lemma \ref{psiisomorphisms} \eqref{psiisomorphisms:ii}, \eqref{psiisomorphisms:iii}, and \eqref{psiisomorphisms:iv} were applied similarly to the argument for objects. Similarly the third summand translates to $m_{\mathbb{L}}(s_f\otimes s_Y)p_{\overline{V},A}\mathbb{D}m_r$. Altogether, the claimed formula for morphisms results.

For the composition, recall that $c_{gf}$ is given by 
\[c_{gf}=(m_A\otimes \mathbbm{1}_Z)(\mathbbm{1}_A\otimes c_g)c_f.\]
Applying Lemma \ref{psiisomorphisms} similarly to before, one obtains the claimed formula for $s_{gf}$.

For the units, note that 

\[\Psi(c_{\mathbbm{1}})(v)(y)=m(v\otimes \mathbbm{1}_Y)(c_{\mathbbm{1}}(y))=m(v\otimes \mathbbm{1}_Y)(1\otimes y)=v(1)y=s_{\mathbbm{1}}(v)(y).\]

From $\Psi^U$ being an isomorphism for all $U$, we obtain that $\Psi$ is an equivalence.
\end{proof}

\begin{ex}
	Recall that in the running example \ref{example2} we have considered an $\mathbb{L}$-module $Y$ and the map $c_Y \colon Y \to \overline{V} \otimes_{\mathbbm{k}} Y$. Under the equivalence $\mathcal{N}(\mathfrak{A}) \simeq \mathcal{R}(\mathfrak{A})$, the object $(Y, c_Y) \in \mathcal{N}(\mathfrak{A})$ corresponds to $(Y, s_Y)\in \mathcal{R}(\mathfrak{A})$, where $s_Y \colon \mathbb{D}\overline{V} \to \Hom_{\mathbbm{k}}(Y,Y)$ is the map with $s_Y(\varphi)\equiv 0$, $s_Y(\chi)\equiv 0$, $s_Y(\psi a)\equiv 0$, $s_Y(b \varphi)\equiv 0$ and $s_Y(\psi)(v_1)=0$, $s_Y(\psi)(v_2)=0$, $s_Y(\psi)(v_3)=v_2$ and $s_Y(\psi)(w_3)=v_2$.
%
\end{ex}





\section{Construction of the Ringel dual bocs}\label{ringeldualbocs}

The goal of this section is to construct a bocs $\mathfrak{B}$ from the data of a bocs $\mathfrak{A}$ such that the category $\mathcal{R}(\mathfrak{A}^{\op})$, which we have shown to be equivalent to $\mathcal{F}(\nabla_R)$ in the previous sections, becomes equivalent to the category of modules for $\mathfrak{B}$.

Let $\mathfrak{A} = (A,V)$ be a directed normal bocs. The corresponding DG algebra $\mathcal{U} = \bigoplus_{j=0}^\infty \overline{V}^{\otimes_A j}$ is augmented, non-negatively graded and finite dimensional. Let $\mathcal{D}^! = \mathbb{D}T(\mathcal{U}[1])$ be the $\mathbbm{k}$-dual of the bar construction of $\mathcal{U}$ and $I\subset \mathcal{D}^!$ the ideal generated by $\mathcal{D}^!_{\leq -1}$ and $d(\mathcal{D}^!_{-1})$. Since $I$ is a differential ideal, the quotient $\mathcal{U}^! =\mathcal{D}^!/I$ is a DG algebra. By \cite[Lemma 8.1]{KKO14} $\mathcal{U}^!$ is a DG algebra assigned to a directed normal bocs $\mathfrak{B}^! = (B^!,W^!)$ with $B^! = \mathbb{L}[\mathbb{D} \overline{V}]/(\mathbb{L}[\mathbb{D}\overline{V}]\cap I)$ and $W^! = \mathcal{U}^!_1/ (d(B^!))$. The algebra $B^!$ is the quotient of $\mathbb{L}[\mathbb{D} \overline{V}]$ by the ideal $J$ generated by the image of the map $\mathbb{D} \partial_1 + p_{\overline{V}, \overline{V}}\mathbb{D}m_{\overline{V}} \colon \mathbb{D}(\overline{V} \otimes \overline{V}) \to \mathbb{D} \overline{V} \oplus (\mathbb{D} \overline{V} \otimes_\mathbb{L} \mathbb{D} \overline{V})$. The bimodule $W^!$ is generated over $B^!$ by $\mathbb{D}A$ and its group-like elements are $\mathbb{D}e_i$. The projective bimodule $\overline{W}^!$ is generated as a $B^!$-bimodule by $\mathbb{D}(\textrm{rad}A)$.

The algebra structure on $B^!$ is the algebra structure of the tensor algebra over $\mathbb{L}$. $W^!$ is the quotient of a projective bimodule generated by $\mathbb{L}$-bimodule $\mathbb{D}A$. The comultiplication $\mu \colon W^! \to W^! \otimes_{B^!} W^!$ is induced by the multiplication on $A$. Finally, the counit $\varepsilon_{\mathfrak{B}^!}\colon W^! \to B^!$ is the $B^!$-bimodule morphism generated by $\mathbb{D}\varepsilon_{\mathfrak{A}}\colon \mathbb{D}A \to \mathbb{D}V$.

\begin{defn}
	Let $\mathfrak{A}$ be a directed bocs. We call $\mathfrak{B}^!$ as constructed above the  \emphbf{Koszul dual bocs}.
	The bocs $\mathfrak{B} = ((\mathfrak{A}^{\op})^!)^{\op}$ is the \emphbf{Ringel dual bocs}.
\end{defn}


We combine our results so far to obtain the main theorem of this paper. 

\begin{thm}\label{thm_Rigel_dual_bocs}
Let $\mathfrak{A}=(A,V)$ be a directed bocs and $\mathfrak{B}=(B,W)$ its Ringel dual. Then the right algebra of $\mathfrak{B}$ is Morita equivalent to the Ringel dual of the right algebra of  $\mathfrak{A}$. 
\end{thm}

\begin{proof}
By Theorem \ref{maintheoremKKO} the right algebra $R_\mathfrak{A}$ of $\mathfrak{A}$ is quasi-hereditary with $\modu \mathfrak{A}\simeq \mathcal{F}(\Delta_{R_\mathfrak{A}})$. Let $\mathcal{F}(\nabla)$ be the subcategory of $\nabla$-filtered objects of $R_\mathfrak{A}$. According to Proposition \ref{diamondcostandard}, the category $\mathcal{F}(\nabla)$ is equivalent to $\mathcal{F}(\Diamond)$ in $\mathcal{D}^b(\modu \mathfrak{A})$. Applying the duality $\mathbb{D}$, the category $\mathcal{F}(\Diamond)$ is in turn equivalent to $\mathcal{F}(\Box)^{\op}$ in $\mathcal{D}^b(\modu \mathfrak{A}^{\op})$. As observed in Theorem \ref{modN}, the category $\mathcal{F}(\Box)^{\op}$ is equivalent to the category $N(\mathfrak{A}^{\op})^{\op}$ which is in turn equivalent to $\mathcal{R}(\mathfrak{A}^{\op})^{\op}$ by Proposition \ref{dualmodN}. Using the description of bocs representations in \ref{boxrepresentations}, it is easy to see that $\modu \mathfrak{B}^{\op}$ is equivalent to $\mathcal{R}(\mathfrak{A}^{\op})^{\op}$. Indeed, an object of $\modu \mathfrak{B}$ is a $B = \mathbb{L}[N_0]/J$ module, i.e. an $\mathbb{L}$-module $M$ together with a map $s_M \colon \mathbb{D}\overline{V} \to \Hom_\mathbbm{k}(M,M)$ vanishing on $J$. The later condition can be written as $s_M \circ \mathbb{D}\partial_1 + m(s_M\otimes s_M)p_{\overline{V}, \overline{V}} \mathbb{D}m_{\overline{V}} =0$. A morphism $f\colon M \to N$ in $\modu \mathfrak{B}$ is a map $s_f \colon\mathbb{D}A = \mathbb{D}\mathbb{L} \oplus \mathbb{D}\textrm{rad}(A) \to \Hom_{\mathbbm{k}}(M,N)$ which vanishes on the image of $N_0$ in $N_1$. Taking into account the definition of $\mathcal{U}$, the second condition translates into $m((s_f \otimes s_M)p_{\overline{V},A} \mathbb{D}m_r - (s_N \otimes s_f)p_{A, \overline{V}} \mathbb{D}m_l) + s_f \mathbb{D}_0 = 0$.  Thus, applying duality, $\modu \mathfrak{B}\simeq \mathcal{R}(\mathfrak{A}^{\op})\simeq \mathcal{F}(\nabla)$. By Theorem \ref{dlabringel}, the quasi-hereditary algebra with prescribed category of standard modules is unique up to Morita equivalence. Thus, $R_\mathfrak{B}$ is  Morita equivalent to the Ringel dual of $R_\mathfrak{A}$. 
\end{proof}

\begin{ex}\label{example_dual_bocs}
	In the running example \ref{example} we have considered a DG algebra $\mathcal{U}$ associated to the bocs $\mathfrak{A}$.  $\mathcal{U}$ is concentrated in degrees $0, 1, 2$. Its bar construction is a DG coalgebra with $A \oplus (A \otimes_{\mathbb{L}} \overline{V}) \oplus (\overline{V} \otimes_{\mathbb{L}} A)$ in degree $-1$, $\overline{V}$ in degree $0$ and $\overline{V} \otimes_{A} \overline{V}$ in degree $1$. Its dual is the DG algebra $\mathcal{D}^!$ with 
	\begin{align*} 
	&{\mathcal{D}^!}_{-1} = \textrm{span}\{ \widehat{\psi \varphi}\},&\\ &{\mathcal{D}^!}_0 = \mathbb{L} \oplus \textrm{span}\{\widehat{\varphi},\widehat{\psi}, \widehat{\chi}, \widehat{\psi a}, \widehat{b\varphi},\widehat{\psi} \otimes_{\mathbb{L}} \widehat{\varphi}\},&\\
	&{\mathcal{D}^!}_1 = ({\mathcal{D}^!}_0 \otimes_{\mathbb{L}} \mathbb{L}) \oplus \textrm{span}\{ \widehat{a}, \widehat{b},\widehat{c}, \widehat{ba}, \widehat{\psi} \otimes_{\mathbb{L}} \widehat{a}, \widehat{b}\otimes_{\mathbb{L}} \widehat{\varphi}\}.
	\end{align*}
	The non-zero differentials are
	\begin{align*}
	&\partial^{-1}(\widehat{\psi \varphi}) = \widehat{\chi}+ \widehat{\psi} \otimes_{\mathbb{L}} \widehat{\varphi},& &\partial^0(\widehat{\psi a}) = \widehat{c} + \widehat{\psi} \otimes_{\mathbb{L}} \widehat{a},& &\partial^0(\widehat{b \varphi}) = \widehat{c} + \widehat{b} \otimes_{\mathbb{L}} \widehat{\varphi},& &\partial^1(\widehat{ba}) = \widehat{b} \otimes_{\mathbb{L}} \widehat{a}.&
	\end{align*}
	Let $\mathfrak{B}^! =(B^!,W^!)$ be the Koszul dual bocs. $B^! = {\mathcal{D}^!}_0/\partial^{-1}({\mathcal{D}^!}^{-1})$ is the path algebra of the quiver
	\begin{center}
		\includegraphics{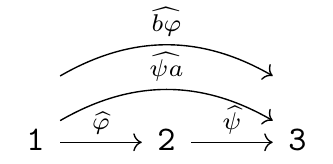}
	\end{center}
	If we put $e = e_\mathtt{1} + e_\mathtt{2} + e_\mathtt{3}$ then $W^!$ is isomorphic to
	$$
	W ^!= {\mathcal{U}^!}_1/\partial^0(B^!) \cong B^! \otimes_{\mathbb{L}} e \oplus \textrm{span}\{\widehat{a}, \widehat{b}, \widehat{c}, \widehat{ba}, \widehat{\psi} \otimes_{\mathbb{L}} \widehat{a},  \widehat{b}\otimes_{\mathbb{L}} \widehat{\varphi}\}.
	$$ 
	The right $B^!$-module structure is twisted by $\partial^0$, i.e. $e \cdot x= x\otimes_{\mathbb{L}}e - \partial^0(x)$, for any $x\in B$. The bimodule $\overline{W}^! = \textrm{span}\{\widehat{a}, \widehat{b}, \widehat{c}, \widehat{ba}, \widehat{\psi} \otimes_{\mathbb{L}} \widehat{a},  \widehat{b}\otimes_{\mathbb{L}} \widehat{\varphi}\}$ is a projective $B^!$-bimodule generated by $\widehat{a}, \widehat{b}, \widehat{c}, \widehat{ba}$. The comultiplication on $B^! \otimes_{\mathbb{L}} e$ is determined by $\mu(e) = e \otimes e$, while the comultiplication on $\overline{W}^!$ is determined by the multiplication in $A$, i.e. we have $\mu(\widehat{ba}) = e \otimes_{\mathbb{L}} \widehat{ba} + \widehat{ba} \otimes_{\mathbb{L}} e + \widehat{b} \otimes_{\mathbb{L}} \widehat{a}$. 
	
	This bocs is in some way not minimal possible, namely as we will see in Example \ref{regularisation_example} it is not regular. Similarly to \cite[Appendix A.2]{KKO14}, it provides an instance of non-uniqueness of directed bocses for quasi-hereditary algebras.
\end{ex}

\section{Smooth rational surfaces and curve-like algebras}\label{applications}

In this section we demonstrate how knowledge of additional properties of a quasi-hereditary algebra can be used to exclude certain possibilities for the $A_\infty$-structure on the $\Ext$-algebra of the standard modules.  

Recall, that an algebra $\Lambda$ is \emphbf{left strongly quasi-hereditary} if it is quasi-hereditary and the projective dimension of every standard $\Lambda$-module is at most one. The corresponding biquiver has then a simple form (Remark \ref{rmk_bocs_of_strongly}). Using the language of bocses, there are different equivalent description of this:

\begin{prop}[{\cite[Proposition 4.42]{K16}}]\label{prop_almost_strong}
The following are equivalent for a quasi-hereditary algebra $\Lambda$:
\begin{enumerate}[(1)]
\item $\Lambda$ is left strongly quasi-hereditary.
\item The exceptional collection of standard modules is almost strong.
\item $\Lambda$ is Morita equivalent to the right algebra of a free normal bocs.
\end{enumerate}
\end{prop}
\begin{proof}
	The implication $\mathit{(1)}\Rightarrow \mathit{(2)}$ is clear. For $\mathit{(2)}\Rightarrow \mathit{(3)}$, let $\Lambda$ be a quasi-hereditary algebra with almost strong exceptional collection of standard modules. The algebra $\Lambda$ is Morita equivalent to the right algebra of a directed normal bocs $(A, V)$, see Theorem \ref{maintheoremKKO}.(i). For any pair $L(\mathtt{i})$, $L(\mathtt{l})$ of simple $A$-modules, Theorem \ref{maintheoremKKO}.(iv) implies vanishing of $\Ext^2_A(L(\mathtt{i}), L(\mathtt{l})) \cong \Ext^2_\Lambda(\Delta(\mathtt{i}), \Delta(\mathtt{l})$. Hence, the algebra $A$ is hereditary, i.e. the bocs $(A,V)$ is free. 
	
	Let now $\mathfrak{A} =(A,V)$ be a free directed normal bocs and $\Lambda$ its right algebra. Since the equivalence $T\colon \modu \mathfrak{A} \xrightarrow{\simeq} \mathcal{F}(\Delta)$ is an additive functor which maps $A$ to $\Lambda$, $T$ maps projective $A$-modules to projective $\Lambda$-modules.
	As the exact structure on $\modu \mathfrak{A}$ comes from the exact structure on $\modu A$, a short projective resolution in $\modu A$ yields a short projective resolution in $\modu \Lambda$. Hence, if $\textrm{projdim}_A(L(\mathtt{i})) \leq1$ then so is $\textrm{projdim}_\Lambda(\Delta(\mathtt{i}))$. In particular, if $A$ is hereditary, $\textrm{projdim}_\Lambda(\Delta(\mathtt{i})) \leq 1$, i.e. $\Lambda$ is left strongly quasi-hereditary which finishes the proof of $\mathit{(3)}\Rightarrow \mathit{(1)}$.
\end{proof}

Assume now that $\Lambda$ is a left strongly quasi-hereditary algebra with a duality $\mathcal{D}$ on $\modu \Lambda$ preserving simple modules. Then the functor $\mathcal{D}$ maps standard modules to costandard and projective to injective, hence $\Lambda$ is also right strongly quasi-hereditary, i.e. all costandard $\Lambda$-modules have injective dimension less than two. By \cite{Par01}, $\Lambda$ has global dimension two. 

Let $R(\Lambda)$ be the Ringel dual of $\Lambda$. Then, by \cite{Rin10}, $R(\Lambda)$ is right strongly quasi-hereditary and it has a duality preserving simple modules, hence it is also left strongly quasi-hereditary.

It follows that the class of left strongly quasi-hereditary algebras with duality preserving simple modules is closed under Ringel duality.  Since the duality maps standard $\Lambda$-modules to costandard, we have
\begin{equation}
\begin{aligned}
&\dim \Hom_\Lambda(\Delta(\mathtt{i}), \Delta(\mathtt{l})) = \dim \Hom_\Lambda(\nabla(\mathtt{l}), \nabla(\mathtt{i})),&\\
&\dim \Ext^1_\Lambda(\Delta(\mathtt{i}), \Delta(\mathtt{l})) = \dim \Ext^1_\Lambda(\nabla(\mathtt{l}), \nabla(\mathtt{i})). &
\end{aligned}
\end{equation}

We say that an algebra $\Lambda$ is \emphbf{curve-like} if it is a left strongly quasi-hereditary algebra with a duality preserving simple modules and  
\[\dim \Hom_\Lambda(\Delta(\mathtt{i}), \Delta(\mathtt{l})) =1 = \dim \Ext^1_\Lambda(\Delta(\mathtt{i}), \Delta(\mathtt{l}))\] 
for all $\mathtt{1}\leq \mathtt{i}< \mathtt{l} \leq \mathtt{n}$ (see the introduction for a motivation where the name comes from).
	We believe that curve-like quasi-hereditary algebras provide an interesting class of finite-dimensional algebras. Some examples of these have already provided useful counterexamples in the work of V. Mazorchuk \cite{Maz10} and the second author \cite{K16}. 
	
	We use the explicit construction of a Ringel dual bocs to give non-obvious conditions on the Ext-algebra of standard modules over a curve-like algebra, Lemmas \ref{lem_comp_of_homs}, \ref{lem_comp_hom_ext}. We prove that for algebras with a small number of simple objects any bocs satisfying these conditions is the bocs of a curve-like algebra.

\begin{defn}
	Let $(A,V)$ be a directed bocs with $A$ basic. An arrow $a$ in the quiver of $A$ (which is identified with an element of $A$) is called \emphbf{superfluous} or \emphbf{non-regular} if $\partial(a)$ is a generator of an indecomposable direct summand of the projective bimodule $\overline{V}$. A bocs is called \emphbf{regular} if it does not contain any superfluous arrows.
\end{defn}

In the case that the bocs corresponds to a directed biquiver, a solid arrow $a\in Q(\mathtt{i},\mathtt{l})$ is called superfluous if $\partial(a)=\lambda v+\sum_j \mu_j p_j$ where $0\neq \lambda\in \mathbbm{k}$, $\mu_j\in \mathbbm{k}$, and the $p_j$ are paths from $\mathtt{i}$ to $\mathtt{l}$ with $v$ not contained in any of them. In this case, the corresponding element $a\in A$ is superfluous. There is an equivalence of module categories of bocses removing $a$ and $v$ and replacing any occurence of $v$ in the differentials of the arrows by $-\frac{1}{\lambda}\sum_j\mu_j p_j$. This process is called \emphbf{regularisation} and was introduced by M. Kleiner and A. Roiter in the case where $A$ is the path algebra of a quiver in \cite{KR75}.

\begin{ex}\label{regularisation_example}
	In the Example \ref{example_dual_bocs} the Koszul dual bocs is not regular as the arrows $\widehat{\psi a}$ and $\widehat{b \varphi}$ are superfluous. The regularisation of $\mathfrak{B}^!$ is a bocs $\mathfrak{B}^!_r = (B^!_r, W^!_r)$ with algebra $B^!_r$ equal to the subalgebra $B^!\setminus \{\widehat{\psi a}\}$ of $B^!$ (one could also choose $B^!_r = B^! \setminus \{\widehat{b \varphi}\}$). The bimodule $W^!_r$ is $W^! \setminus \{c\}$. We also have $\partial^0_{\mathfrak{B}^!_r}(\widehat{b \varphi}) = \widehat{b} \otimes_{\mathbb{L}} \widehat{\varphi} - \widehat{\psi} \otimes_{\mathbb{L}} \widehat{a}$. In other words $B^!_r$ is the path algebra of the quiver   
	\begin{center}
		\includegraphics{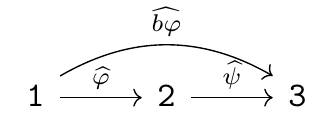}
	\end{center}
	If we put $e = e_\mathtt{1} + e_\mathtt{2} + e_\mathtt{3}$ then $W^!_r\cong  B^!_r \otimes_{\mathbb{L}} e \oplus \textrm{span}\{\widehat{a}, \widehat{b}, \widehat{ba}, \widehat{\psi} \otimes_{\mathbb{L}} \widehat{a},  \widehat{b}\otimes_{\mathbb{L}} \widehat{\varphi}\}$ as let $B^!_r$ module and the right $B^!_r$-module structure is twisted by $\partial^0$. It follows that the bocs $\mathfrak{A}$ of Example \ref{example} is self-Koszul dual up to regularisation. It is well-known that in this case $\Lambda, \Lambda^{op}$ and their Ringel duals are all Morita equivalent.
\end{ex}

We need the following characterisation of regular directed bocses, which can be found in unpublished notes of S. Ovsienko and might be well known in the Kiev school. A proof  will appear in a forthcoming article of the second author with V. Miemietz \cite{KM17}.

\begin{lem}\label{miemietz-lemma}
Let $\mathfrak{A}$ be a directed normal bocs. Then, the following are equivalent:
\begin{enumerate}[(1)]
\item $\mathfrak{A}$ is regular.
\item $\Ext^1_A(L(\mathtt{i}),L(\mathtt{l}))\cong \Ext^1_{\mathfrak{A}}(\Delta(\mathtt{i}),\Delta(\mathtt{l}))$.
\item As a projective bimodule, $\overline{V}$ has $\sum_{\mathtt{i}\neq \mathtt{l}}\dim \Hom_{\mathfrak{A}}(L(\mathtt{i}),L(\mathtt{l}))$ generators.
\end{enumerate} 
\end{lem}

We use the construction of the bocs of $R(\Lambda)$ given in Theorem \ref{thm_Rigel_dual_bocs} to exclude possible $A_\infty$-structures on the Ext-algebra of standard modules over a curve-like algebra.

\begin{lem}\label{lem_comp_of_homs}
	Let $\Lambda$ be a curve-like algebra. Then the composition of homomorphisms between standard $\Lambda$-modules is non-zero.
\end{lem}

\begin{proof}
	Assume that 
	$\psi \in \Hom(\Delta(\mathtt{l}), \Delta(\mathtt{m}))$, $\varphi \in \Hom(\Delta(\mathtt{i}), \Delta(\mathtt{l}))$ such that $\psi\varphi =0$. By the construction in \cite{KKO14}, in the bocs $(A,V)$ corresponding to $\Lambda$, this will give corresponding generators of the directs summands $Ae_{\mathtt{l}}\otimes_{\mathbbm{k}} e_{\mathtt{m}}A$ and $Ae_{\mathtt{i}}\otimes_{\mathbbm{k}} e_{\mathtt{l}}A$ of $\ker\varepsilon$, which by abuse of notation we denote by the same letters. We depict the situation with the following picture:
\[
\begin{tikzcd}
\mathtt{i}\arrow[dashed]{r}{\varphi}&\mathtt{l}\arrow[dashed]{r}{\psi}&\mathtt{m}
\end{tikzcd}
\]
	The bocs of the Ringel dual quasi-hereditary algebra has $\widehat{\varphi}$, $\widehat{\psi}$ in degree zero and $\widehat{\psi\varphi}$ in degree minus one (where we denote the corresponding elements of a dual basis with a hat above their names). We depict the situation with the following picture:
	\[
\begin{tikzcd}
\mathtt{i}\arrow{r}{\hat{\varphi}}\arrow[bend right, dotted]{rr}[swap]{\wh{\psi\varphi}}&\mathtt{l}\arrow{r}{\hat{\psi}}&\mathtt{m}
\end{tikzcd}
\]
	As $\varphi$, $\psi$ are generators of $\ker\varepsilon$ and the bocs $(A,V)$ was assumed to be regular, only the term $p_{\overline{V},\overline{V}}\mathbb{D}m_{\overline{V}}$ contributes to $d|_{N_{-1}}(\wh{\psi\varphi})$. Thus, $\partial(\wh{\psi\varphi}) = \wh{\psi}\otimes \wh{\varphi}$ (note that, if $\psi \circ \varphi = \tau$, we would have $\partial(\wh{\psi\varphi})=\wh{\psi} \otimes \wh{\varphi} + \wh{\tau}$). To prove that this gives a relation in $B$, i.e. a non-trivial $\Ext^2$ between costandard modules, 
	we have to prove that $\hat{\varphi},\hat{\psi}$ are not superfluous in $(A,V)$, cf. the foregoing lemma. To prove that they are not superfluous, note that since $\varphi,\psi$ were generators of $\ker\varepsilon$ and $(A,V)$ is assumed to be regular, $\partial(\hat{\varphi})=\partial(\hat{\psi})=0$. Furthermore, note that $\hat{\varphi}$ and $\hat{\psi}$ are also non-zero in $A$ as the only relations arise from $\mathbb{D}(\overline{V}\otimes_A\overline{V})$ and always involve the term $p_{\overline{V},\overline{V}}\mathbb{D}\overline{V}$. From Theorem \ref{maintheoremKKO} \eqref{maintheoremKKO:iv} we conclude that $\Ext^2(\nabla(\mathtt{i}), \nabla(\mathtt{m}))\neq 0$, a contradiction to the fact that the algebra is strongly quasi-hereditary.
\end{proof}

\begin{lem}\label{lem_comp_hom_ext}
Let $\Lambda$ be a curve-like algebra. Let \small
\[\varphi\in \Hom_\Lambda(\Delta(\mathtt{i}),\Delta(\mathtt{l})), \psi\in \Hom_\Lambda(\Delta(\mathtt{l}),\Delta(\mathtt{m})), a\in \Ext^1_\Lambda(\Delta(\mathtt{i}),\Delta(\mathtt{l})),b\in \Ext^1_\Lambda(\Delta(\mathtt{l}),\Delta(\mathtt{m})).\] \normalsize
Then at least one of the compositions $b\varphi$ and $\psi a$ is non-zero.
\end{lem}

\begin{proof}
Note that $\psi a$ and $b \varphi$ are distinct elements of $\mathbb{L}^1[\mathbb{D}sE]$. Thus, if the  arrows $\wh{\psi a}$ and $\wh{b\varphi}$ in the Ringel dual bocs were not superfluous they would give two distinct arrows between vertices $\mathtt{i}$ and $\mathtt{m}$ in the algebra $B$ of the Ringel dual bocs, which by Lemma \ref{miemietz-lemma} would give a contradiction to the fact that $\Lambda$, whence its Ringel dual, is assumed to be curve-like since this would give a more than $2$-dimensional $\Ext^1$-space between costandard modules for $\Lambda$.

It is sufficient to prove that the arrows $\wh{\psi a}$ and $\wh{b\varphi}$ are not superfluous. By construction and the previous lemma, the Ringel dual bocs is again free (i.e. the algebra $B$ is hereditary). Moreover, note that $\partial(\wh{\psi a})$ has no term which is a generator of $\overline{W}$. Indeed, such a generator would come from the term $\mathbb{D}\partial_0$. Since $\partial_0$ is constructed as the dual of the $m_i$ on $\Ext^*(\Delta,\Delta)$, $\partial_0(c)=\psi a$ would mean that $\psi a=c$ in $\Ext^*(\Delta,\Delta)$. Thus, if $\psi a = 0 = b \varphi$, both $\wh{\psi a}$ and $\wh{b \varphi}$ are not superfluous which contradicts the assumption that the algebra, hence also its Ringel dual, is curve-like by Lemma \ref{miemietz-lemma}.
\end{proof}

%

We continue by illustrating how this yields a classification of the curve-like quasi-hereditary algebras in small examples. 

Clearly there is only one biquiver of a curve-like algebra with two simple modules. The unique curve-like algebra with two simple modules  is the algebra $\xymatrix{\bullet \ar@<1ex>[r]^\alpha & \bullet \ar@<1ex>[l]^\beta}$ with $\alpha\beta =0$.

\subsection{Curve-like algebras with three simples}\label{ssec_3_simples}

In the case of three simples, the situation is restricted enough that we can classify not only the curve-like quasi-hereditary algebras, but all quasi-hereditary algebras with the same dimensions of $\Hom$- and $\Ext$-spaces between standard modules. 
The corresponding biquiver is the biquiver (\ref{eqtn_quiver_with_3_vertices}) considered in the running example, see Example \ref{example}.


Using possibly scaling of the arrows, there are the following $8$ possibilities for the differential of the bocs $(A,V)$:
\[\partial(\chi)=\begin{cases}0&\text{case 1}\\\psi\varphi&\text{case 2}\end{cases}\text{ and }\partial(c)=\begin{cases}\psi a&\text{case A}\\b\varphi&\text{case B}\\\psi a+b\varphi&\text{case C}\\0&\text{case D}\end{cases}\]

By Lemma \ref{lem_comp_of_homs}, the four algebras in case $1$ are not right strongly quasi-hereditary. The algebras in case $D$ will have different dimension of $\Ext^1_\Lambda(\nabla(\mathtt{1}),\nabla(\mathtt{3}))$. Hence, only three of the algebras in question are curve-like, namely cases $2A$, $2B$, and $2C$. The cases $2A$ and $2C$ actually arise from the geometry of surfaces, as explained in \cite{BodBon17}. The case $2B$ is Ringel dual to the case $2A$. The algebra $2B$ appears in a paper by V. Mazorchuk, see \cite[Example 23]{Maz10}, and also \cite[Example 4.59]{K16} for its category of filtered modules. A Morita representative of the corresponding quasi-hereditary algebras can be obtained by taking the right algebra of the corresponding bocs. This will usually not be basic. For convenience of the reader we instead list the corresponding basic algebras. To illustrate what the Ringel dual algebra might be if the assumption on being curve-like is omitted we provide also the Ringel dual algebras.


\begin{description}
	\item[1A] Quiver
	\begin{center} 
		\includegraphics{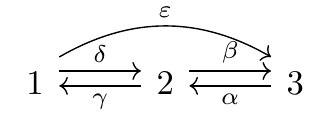}
	\end{center} 
	with relations $\gamma\delta=\beta\delta=\alpha\varepsilon=\alpha\beta=0$, which has Ringel dual given by
	\begin{center}
		\includegraphics{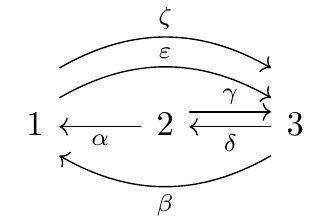}
	\end{center}
	with relations $\delta\zeta=\beta\varepsilon=\beta\zeta = \delta\gamma =\alpha\delta\varepsilon=0$, after removing the superfluous arrow $\wh{\psi a}$ (together with its counterpart $\hat{c}$), the bocs corresponding to the Ringel dual looks as follows:
	\begin{center}
		\includegraphics{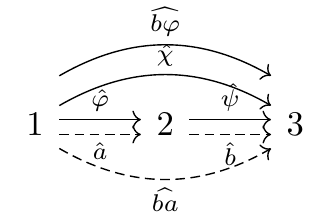}
	\end{center}
	\text{with relation $\hat{\psi}\hat{\varphi}$ and differential $\partial(\wh{b\varphi})=\hat{b}\otimes \hat{\varphi}$ and $\partial(\wh{ba})=\hat{b}\otimes \hat{a}$}
	\item[1B] Quiver
	\begin{center} 
		\includegraphics{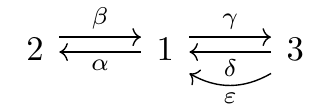}
	\end{center} 
	with relations $\alpha\delta, \gamma\delta, \gamma\varepsilon, \beta\alpha\varepsilon, \alpha\beta$, its Ringel dual is isomorphic to the opposite algebra of $\bf{1A}$. Its bocs looks like in case $\bf{1A}$ with $\wh{b\varphi}$ replaced by $\wh{\psi a}$ and $\partial(\wh{\psi a})=\hat{\psi}\otimes \hat{a}$. 
	\item[1C] Quiver 
	\begin{center}
		\includegraphics{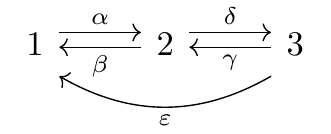}
	\end{center}
	with relations $\alpha\varepsilon=\delta\gamma=\beta\gamma=0, \alpha\beta=\gamma\delta$. Its Ringel dual is 
	\begin{center}
		\includegraphics{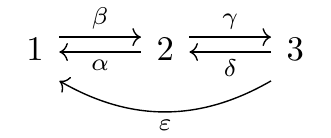}
	\end{center}
	with relations $\alpha\beta=\alpha\delta=\varepsilon\gamma=0, \beta\alpha=\delta\gamma$. The bocs corresponding to the Ringel dual has biquiver 
	\begin{center}
		\includegraphics{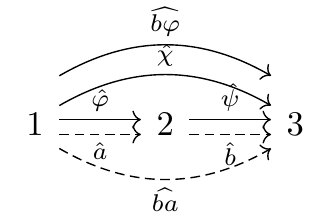}
	\end{center}
	with differential $\partial(\wh{b\varphi})=\hat{b}\otimes \hat{\varphi}-\hat{\psi}\otimes \hat{a}$ and $\partial(\wh{ba})=\hat{b}\otimes \hat{a}$ and relations $\hat{\psi}\hat{\varphi}=0$. 
	\item[1D] Quiver
	\begin{center}
		\includegraphics{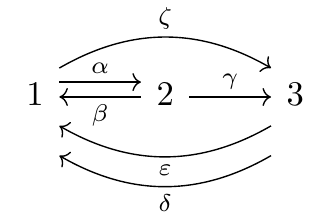}
	\end{center}
	with relations $\alpha\delta=\beta\alpha\varepsilon=\gamma\alpha\varepsilon=\alpha\beta=\zeta\beta=\zeta\delta=\zeta\varepsilon=0$. Its Ringel dual is isomorphic to its own opposite algebra. The $\Ext$-spaces between costandard modules are $\Ext^i_\Lambda(\nabla(1),\nabla(3))=\begin{cases}2&\text{for $i=0$}\\3&\text{for $i=1$},\\1&\text{for $i=2$},\\0&\text{otherwise.}\end{cases}$
	\item[2A] Quiver
	\begin{center}
		\includegraphics{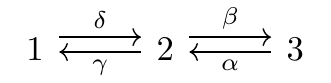}
	\end{center}
	with relations $\beta\alpha=\delta\gamma=0$, which is Ringel dual to $\bf{2B}$,
	\item[2B] Quiver
	\begin{center}
		\includegraphics{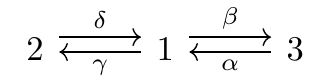}
	\end{center}
	with relations $\gamma\delta=\beta\alpha=\beta\delta\gamma\alpha=0$,
	\item[2C] Quiver
	\begin{center}
		\includegraphics{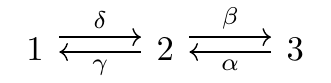}
	\end{center}
	with relations $\beta\alpha=0$ and $\delta\gamma=\alpha\beta$. This is in fact the Auslander algebra of $k[x]/(x^3)$ and is well-known to be Ringel self-dual.
	\item[2D] Quiver
	\begin{center}
		\includegraphics{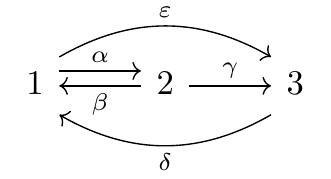}
	\end{center}
	with relations $\varepsilon\delta=\gamma\alpha\delta=\alpha\beta=\varepsilon\beta=0$. Its Ringel dual is isomorphic to its opposite algebra and has a bocs given by the biquiver
	\begin{center}
		\includegraphics{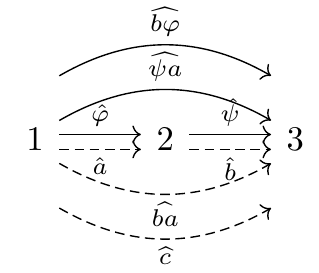}
	\end{center}
	without relations and differential $\partial(\wh{b\varphi})=\hat{b}\otimes \hat{\varphi}, \partial(\wh{\psi a})=\hat{\psi}\otimes \hat{a}$, and $\partial(\wh{ba})=\hat{b}\otimes \hat{a}$.
\end{description}

\subsection{Curve-like algebras with four simples}\label{ssec_four_simple}

Finally, we classify all possible $A_\infty$-structures on curve-like algebras with four simple modules. 

At first we do not give names to the composed maps and extensions, but only write down the irreducible maps which we denote as
\begin{equation}\label{eqtn_quiver}
\xymatrix{\mathtt{1} \ar@<1ex>[r]|a \ar@{.>}@<-1ex>[r]|{\varphi} & \mathtt{2} \ar@<1ex>[r]|b \ar@{.>}@<-1ex>[r]|{\psi} & \mathtt{3} \ar@<1ex>[r]|c \ar@{.>}@<-1ex>[r]|{\rho}& \mathtt{4}}
\end{equation} 
with $a$, $b$, $c$ representing non-trivial elements in $\Ext^1$-groups and $\varphi$, $\psi$, $\rho$ non-trivial homomorphisms between standard modules. (By Lemmas \ref{lem_comp_of_homs} and \ref{lem_comp_hom_ext}, all other homomorphisms and elements of $\Ext^1$ are composition.)

It follows from Lemma \ref{lem_comp_of_homs} that for any curve-like algebra $\rho \circ \psi \circ \varphi \neq 0$. Moreover, Lemma \ref{lem_comp_hom_ext} implies that there are the following possibilities on the composition morphisms with elements of the first Ext-groups:

\begin{center}	\begin{tabular}{c|c|c|c|c|c|c|c|}
		& $\psi a$ & $b\varphi$ & $\rho b$ & $c \psi$ & $\rho\psi a$ & $\rho b \varphi$ & $c\psi \varphi $\\
		\hline
		A & 0 & $\neq 0$ & 0 & $\neq 0$&0  &0 &$\neq 0$\\
		\hline 
		B & 0&$\neq 0$ & $\neq 0$&0 &0 &$\neq 0$ &0\\
		\hline 
		C &0 & $\neq 0$& $\neq 0$& $\neq 0$&0 &$\neq 0$ &$\neq 0$\\
		\hline 
		D &$\neq 0$ &0 &0 &$\neq 0$ &0 &0 &$\neq 0$\\
		\hline 
		E &$\neq 0$ &0 &0 &$\neq 0$ & $\neq 0$&0 &0\\
		\hline 
		F & $\neq 0$&0 &0 &$\neq 0$ &$\neq 0$ &0 &$\neq 0$\\
		\hline 
		G &$\neq 0$ &0 &$\neq 0$ &0 &$\neq 0$ &0 &0\\
		\hline 
		H &$\neq 0$ &0 &$\neq 0$ &$\neq 0$ &$\neq 0$ &0 &0\\
		\hline 
		I &$\neq 0$ &$\neq 0$ &0 &$\neq 0$ &0 &0 &$\neq 0$\\
		\hline 
		J &$\neq 0$ &$\neq 0$ &$\neq 0$ &0 &$\neq 0$ &$\neq 0$ &0\\
		\hline
		K &$\neq 0$ & $\neq 0$& $\neq 0$& $\neq 0$& $\neq 0$&$\neq 0$ &$\neq 0$\\
		\hline 
	\end{tabular} 
\end{center}

	Rescaling $a,b,c$, if necessary, we can always assume that two elements of $\Hom$ or $\Ext^1$ that differ by a non-zero scalar are in fact equal (and then choose this composition as the basis element of the corresponding $\Ext^1$-space).
	
	Some of the above algebras might have different $A_{\infty}$-structures. 
	First, we show that, up to $A_\infty$-quasi-isomorphism, $m_3(c, \psi, \varphi)$, $m_3(\rho, b, \varphi)$ and $m_3(\rho, \psi, a)$ vanish.

	Recall, that an $A_\infty$\emphbf{-morphism} $F \colon A \to B$ of $A_\infty$-algebras is a family of morphism $F_n \colon A^{\otimes n} \to B$ of degree $1-n$ such that
	\begin{equation}\label{eqtn_A_morphism}
	\sum_{r+s+t=n} (-1)^{r+st}F_{r+1+t}(\id^{\otimes r}\otimes m_s\otimes \id^{\otimes t}) = \sum_{i_1+\ldots+i_r = n}(-1)^wm_r(F_{i_1}\otimes \dots \otimes F_{i_r}),
	\end{equation}
	where $w = (r-1)(i_1-1) + (r-2)(i_2-1) + \ldots + (i_{r-1}-1)$. We say that $F$ is a \emphbf{quasi-isomorphism} if $F_1\colon A \to B$ is. 
	
	\begin{rmk}\label{rem_vanishing_m_3}
		Consider an $A_\infty$-quasi-isomorphism from the $A_\infty$-algebra on the quiver (\ref{eqtn_quiver}) to itself
		with $F_1 = \textrm{Id}$, 
		\begin{align*} 
		&F_2(c,\psi\varphi) = m_3(c,\psi, \varphi),&  &F_2(\rho\psi, a)=-m_3(\rho,\psi,a),& &F_2(\rho,b)=\lambda \rho \psi,&
		\end{align*}
		for $\lambda$ such that $m_3(\rho, b,\varphi)= \lambda \,\rho\psi \varphi$. 
		We assume that $F_2$ vanishes on the remaining pairs of elements and that $F_n =0$ for $n>2$. It follows from (\ref{eqtn_A_morphism}) that we get an $A_\infty$-structure on (\ref{eqtn_quiver}) with
		$$
		m_3(c,\psi, \varphi) = m_3(\rho,\psi,a) =m_3(\rho, b,\varphi)=0.
		$$ 
		For the reader's convenience let us check that $m_3(\rho,b,\varphi )$ indeed vanishes. Equality (\ref{eqtn_A_morphism}) implies that
		$$
		\la \rho \psi \varphi = F_1m_3(\rho, b, \varphi) = \overline{m}_3(\rho,b, \varphi) + m_2(F_2(\rho,b)\otimes \varphi) = \overline{m}_3(\rho, b, \varphi) + \la \rho\psi \varphi.
		$$
	\end{rmk}
	
	Consider the quiver
\begin{center}
	\includegraphics{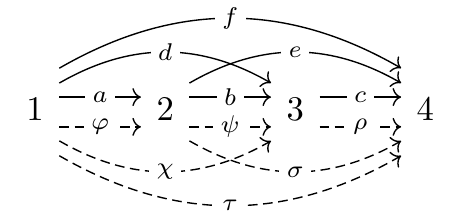}
\end{center}	

where the arrows correspond to basis elements of the $\Ext^1$-spaces (for the solid arrows) or $\Hom$-spaces (for the dashed arrows) as before.

	By Lemma \ref{lem_comp_of_homs} and Remark \ref{rem_vanishing_m_3} the differential on the dashed arrows of the biquiver is given by
	\begin{align*}
	&\partial(\chi) = \psi \f,& &\partial(\sigma) = \rho\psi,& &\partial(\tau) = \sigma \psi + \rho \chi.&
	\end{align*}
	The differentials on the solid arrows of the biquiver depend on the case A--K. 
	
	\begin{tabular}{c|c|c|c|c|c|c|}
		& A & B& C & D& E & F\\
		\hline 
		$\partial(d)$ & $b\f$ & $b\f$ & $b\f$ & $\psi a$& $\psi a$& $\psi a$\\
		\hline
		$\partial(e)$ &$c\psi$ &$\rho b$ &$c\psi+ \rho b$ &$c\psi$ &$c\psi$ &$c\psi$\\
		\hline
		$\partial(f)$ &$c\chi+ e\f$ & $\rho d + e \f$ & $\rho d+ e\f + c\chi$ & $c\chi+ e \f$ & $\rho d + \sigma a$ &$\rho d+ \sigma a + c\chi+ e \f$\\
		\hline 
	\end{tabular} 
	
	\vspace{0.5cm}
	
	\begin{tabular}{c|c|c|c|c|c|}
		& G& H& I& J& K\\
		\hline 
		$\partial(d)$ & $\psi a$ & $\psi a$& $b\f + \psi a$& $b\f + \psi a$ & $b\f + \psi a$\\
		\hline 
		$\partial(e)$ &$\rho b$ & $\rho b + c\psi$ & $c\psi$ &$\rho b$ &$\rho b + c \psi$ \\
		\hline
		$\partial(f)$ &$\rho d+ \sigma a$ &$\rho d+ \sigma a$ & $c\chi + e\f$ & $\rho d+ \sigma a+ e\f$ & $\rho d+ \sigma a+ c\chi+ e \f$\\
		\hline
	\end{tabular} 
	
	with possible further terms in $\partial(f)$ depending on a non-vanishing $A_\infty$-structure.
	 
	The spaces included in the construction of the Ringel dual bocs are
	\begin{align*}
	\mathbb{D}(\overline{V}\otimes_A \overline{V})&=\textrm{span}\{\wh{\psi \f},\, \wh{\rho\psi},\, \wh{\sigma\f},\, \wh{\rho\chi},\, \wh{c\psi\f},\, \wh{\rho b\f},\, \wh{\rho\psi a} \},\\
	\mathcal{D}\overline{V}&= \textrm{span}\{\wh{\f},\, \wh{\psi},\, \wh{\rho},\, \wh{\chi}, \, \wh{b\f},\, \wh{\psi a },\, \wh{\sigma},\, \wh{\rho b},\, \wh{c\psi},\, \wh{\tau},\, \wh{\rho d},\, \wh{c\chi},\, \wh{\sigma a},\, \wh{e\f},\, \wh{\rho b a},\, \wh{c\psi a},\, \wh{cb\f} \},\\
	\mathcal{D}A & =\textrm{span}\{\wh{a},\, \wh{b},\, \wh{c},\, \wh{d},\, \wh{ba},\, \wh{e},\, \wh{cb},\, \wh{f},\, \wh{cba},\, \wh{cd},\, \wh{ea} \}.
	\end{align*}
	
	Let us consider in details cases D and H. 
	
	In case D:
	\begin{align*} 
	&\partial(\wh{\psi\f}) = \wh{\psi} \otimes \wh{\f} + \wh{\chi},&
	&\partial(\wh{\rho \psi}) = \wh{\rho} \otimes \wh{\psi} + \wh{\sigma},&\\
	&\partial(\wh{\sigma \varphi}) = \wh{\sigma} \otimes \wh{\rho} + \wh{\tau},&
	&\partial(\wh{\rho \chi}) = \wh{\sigma}\otimes \wh{\chi} + \wh{\tau},&\\
	&\partial(\wh{c\psi \f}) = \wh{c}\otimes \wh{\psi\f} + \wh{c\psi} \otimes \wh{\f} + \wh{c\chi} + \wh{e\f},&
	&\partial(\wh{\rho\psi a}) = \wh{\rho \psi } \otimes \wh{a} + \wh{\rho} \otimes \wh{\psi a} + \wh{\sigma a}  + \wh{\rho d}, &\\
	&\partial(\wh{\rho b \f}) = \wh{\rho} \otimes \wh{b\f} + \wh{\rho b} \otimes \wh{\f}.& 
	\end{align*} 
	We use the above differentials to regularise $\widehat{\psi \varphi}$ with $\wh{\chi}$, $\wh{\rho \psi}$ with $\wh{\sigma}$, $\wh{\sigma \varphi}$ with $\wh{\tau}$, $\wh{c\psi \varphi}$ with $\wh{c\chi}$ and $\wh{\rho \psi a}$ with $\wh{\sigma a}$ in the dual quiver. The following arrows remain:
	\begin{align*}
	& \deg 0  = \{\wh{\f},\, \wh{\psi},\, \wh{\rho}, \, \wh{b\f},\, \wh{\psi a },\, \wh{\rho b},\, \wh{c\psi},\, \wh{\rho d},\, \wh{e\f},\, \wh{\rho b a},\, \wh{c\psi a},\, \wh{cb\f} \},&\\
	&\deg 1  = \{\wh{a},\, \wh{b},\, \wh{c},\, \wh{d},\, \wh{ba},\, \wh{e},\, \wh{cb},\, \wh{f},\, \wh{cba},\, \wh{cd},\, \wh{ea} \}.&
	\end{align*} 
	Then
	\begin{align*}
	&\partial(\wh{b\f})=\wh{b} \otimes \wh{\f},&  
	&\partial(\wh{\psi a})=\wh{\psi} \otimes \wh{a} + \wh{d},&\\
	&\partial(\wh{\rho b})=\wh{\rho} \otimes \wh{b},& 
	&\partial(\wh{c\psi})=\wh{c} \otimes \wh{\psi} + \wh{e},&\\
	&\partial(\wh{\rho d})=\wh{\rho} \otimes \wh{d},&
	&\partial(\wh{e \f})=\wh{e} \otimes \wh{\f} + \wh{f},&\\
	&\partial(\wh{\rho b a})=\wh{\rho} \otimes \wh{ba} + \wh{\rho b} \otimes \wh{a},&
	&\partial(\wh{cb\f})=\wh{c} \otimes \wh{b\f} + \wh{cb} \otimes \wh{\f},&\\
	&\partial(\wh{c\psi a})=\wh{c\psi} \otimes \wh{a} + \wh{c} \otimes \wh{\psi a}+ \wh{cd} + \wh{ea}.&
	\end{align*}
	The differentials of $\wh{\rho b a}$, $\wh{cb\f}$ and $\wh{c\psi a}$ can also depend on the $A_\infty$-structure. If $m_3(\rho, b, a)$, $m_3(c, b, \f)$ or $m_3(c, \psi, a)$ is non-zero, then it is equal to $\la f$ and in $\partial(\wh{\rho b a})$, $\partial(\wh{cb\f})$ and $\partial(\wh{c\psi a})$ a term $\la\wh{f}$ needs to be added. However, it does not affect the dimension of the regular quiver as $\partial(\wh{f})$ can be regularised with $\wh{e\f}$.
		
%
	We can regularise $\wh{\psi a}$ with $\wh{d}$, $\wh{c\psi}$ with $\wh{e}$, $\wh{e\f}$ with $\wh{f}$, and $\wh{c\psi a}$ with $\wh{cd}$. We are left with
	\begin{align*}
	&\deg 0  = \{\wh{\f},\, \wh{\psi},\, \wh{\rho}, \, \wh{b\f},\, \wh{\rho b},\, \wh{\rho d},\, \wh{\rho b a},\, \wh{cb\f} \},& &
	\deg 1  = \{\wh{a},\, \wh{b},\, \wh{c},\, \wh{ba},\, \wh{cb},\, \wh{cba},\, \wh{ea} \}.&
	\end{align*}
This shows that the dimensions of $A$ and $V$ do not agree with the dimensions for a curve-like algebra given by Lemma \ref{miemietz-lemma}, hence case D cannot be a bocs of a curve-like algebra.
	
	In the similar manner we can exclude cases E and F.
	
	To illustrate what happens in the ``good'' case we consider in detail the case H. The differentials are
	\begin{align*}
	&\partial(\wh{\psi\f}) = \wh{\psi} \otimes \wh{\f} + \wh{\chi},&
	&\partial(\wh{\rho \psi}) = \wh{\rho} \otimes \wh{\psi} + \wh{\sigma},&\\
	&\partial(\wh{\sigma \varphi}) = \wh{\sigma} \otimes \wh{\rho} + \wh{\tau},&
	&\partial(\wh{\rho \chi}) = \wh{\sigma}\otimes \wh{\chi} + \wh{\tau},&\\
	&\partial(\wh{c\psi \f}) = \wh{c}\otimes \wh{\psi\f} + \wh{c\psi} \otimes \wh{\f} + \wh{c\chi} + \wh{e\f},&
	&\partial(\wh{\rho\psi a}) = \wh{\rho \psi } \otimes \wh{a} + \wh{\rho} \otimes \wh{\psi a} + \wh{\sigma a}   + \wh{\rho d},&\\
	&\partial(\wh{\rho b \f}) = \wh{\rho} \otimes \wh{b\f} + \wh{\rho b} \otimes \wh{\f} +\wh{e\f}.& 
	\end{align*}
	As before we regularise the dual quiver to get
	\begin{align*}
	 &\deg 0  = \{\wh{\f},\, \wh{\psi},\, \wh{\rho}, \, \wh{b\f},\, \wh{\psi a },\, \wh{\rho b},\, \wh{c\psi},\, \wh{\rho d},\, \wh{\rho b a},\, \wh{c\psi a},\, \wh{cb\f} \},\\
	&\deg 1  = \{\wh{a},\, \wh{b},\, \wh{c},\, \wh{d},\, \wh{ba},\, \wh{e},\, \wh{cb},\, \wh{f},\, \wh{cba},\, \wh{cd},\, \wh{ea} \}.
	\end{align*}
	Then
	\begin{align*}
	&\partial(\wh{b\f})=\wh{b} \otimes \wh{\f},& 
	&\partial(\wh{\psi a})=\wh{\psi} \otimes \wh{a} + \wh{d},&\\
	&\partial(\wh{\rho b})=\wh{\rho} \otimes \wh{b} + \wh{e},&\\
	&\partial(\wh{c\psi})=\wh{c} \otimes \wh{\psi} + \wh{e},&
	&\partial(\wh{\rho d})=\wh{\rho} \otimes \wh{d} + \wh{f},&\\
	&\partial(\wh{\rho b a})=\wh{\rho} \otimes \wh{ba} + \wh{\rho b} \otimes \wh{a} + \wh{ea},&
	&\partial(\wh{cb\f})=\wh{c} \otimes \wh{b\f} + \wh{cb} \otimes \wh{\f},&\\
	&\partial(\wh{c\psi a})=\wh{c\psi} \otimes \wh{a} + \wh{c} \otimes \wh{\psi a}+ \wh{cd} + \wh{ea}.&
	\end{align*}
	Again, the differentials of $\wh{\rho b a}$, $\wh{cb\f}$ and $\wh{c\psi a}$ can also depend on the $A_\infty$-structure which does not affect the dimension of the regularised differential biquiver as $\partial(\wh{f})$ can be regularised with $\wh{\rho d}$.

	We can regularise $\wh{\psi a}$ with $\wh{d}$, $\wh{\rho b}$ with $\wh{e}$, $\wh{c\psi}$ with $\wh{e}$, $\wh{\rho d}$ with $\wh{f}$, $\wh{\rho ba}$ with $\wh{ea}$, and $\wh{c\psi a}$ with $\wh{cd}$. We are left with 
	\begin{align*}
	&\deg 0  = \{\wh{\f},\, \wh{\psi},\, \wh{\rho}, \, \wh{b\f},\, \wh{\rho b},\, \wh{cb\f} \},& &
	\deg 1  = \{\wh{a},\, \wh{b},\, \wh{c},\, \wh{ba},\, \wh{cb},\, \wh{cba} \},&
	\end{align*} 
	hence dimensions agree with the dimensions of degree zero and degree one part of the dual quiver of a curve-like algebra. 
	
	Similar calculations show that the dimensions agree in cases A--C, G and I--K.

	We have excluded cases D, E, F. It remains to check how many non-isomorphic $A_\infty$-structures algebras A--C and G--K can be endowed with. Below we write down a table which lists possible $A_\infty$-quasi-isomorphism that can be used to make given $m_3$ zero.

	To exclude possible $A_\infty$-structures, we proceed as in Remark \ref{rem_vanishing_m_3}. The non-trivial value of $F_2$ used to take the given $m_3$ to zero is listed in the table.
	
%
	\begin{center}
	\begin{tabular}{c|c|c|c|}
		&  $m_3(\rho,b,a)$ & $m_3(c,\psi,a)$ & $m_3(c,b,\f)$\\
		\hline 
		A & &  $F_2(c\psi, a)$ & $F_2(c, b\f)$\\
		\hline
		B &$F_2(\rho b,a)$& &$F_2(c, b\f)$ \\
		\hline
		C & $F_2(\rho b,a)$& $F_2(\psi, a)^*$&$F_2(c, b\f)$ \\
		\hline
		G & $F_2(\rho b,a)$& $F_2(c\psi, a)$& \\
		\hline
		H & $F_2(\rho b,a)$& $F_2(c\psi, a)$& \\
		\hline
		I & & $F_2(c\psi, a)$& $F_2(c, b\f)$\\
		\hline
		J & $F_2(\rho b,a)$& $F_2(c\psi, a)$& $F_2(c, b)^*$\\
		\hline
		K & $F_2(\rho b,a)$& $F_2(c\psi, a)$&$F_2(c, b)^*$ \\
		\hline
	\end{tabular} 
	\end{center}
	In the cases marked with $^*$ we use the fact that the possible non-zero value of $m_3$ can be decomposed, i.e. we proceed as in Remark \ref{rem_vanishing_m_3} and $m_3(\rho,\psi, a)$.
	
	
	It follows that there are 13 possible $A_\infty$-algebra structures:
	
	\begin{itemize}
		\item[A1:] 
		\begin{align*}
		&\psi a=0,& &b\f\neq0,& &\rho b=0,& &c\psi\neq0,&
		&\rho \psi a=0,& &\rho b \f=0,& &c\psi\f\neq0,& &m_3(\rho, b, a)  =0& 
		\end{align*}
		
		\item[A2:] 
		\begin{align*}
		&\psi a=0,& &b\f\neq0,& &\rho b=0,& &c\psi\neq0,&
		&\rho \psi a=0,& &\rho b \f=0,& &c\psi\f\neq0,& &m_3(\rho, b, a)  =c\psi\f.& 
		\end{align*}
		
		\item[B1:] 
		\begin{align*}
		&\psi a=0,& &b\f\neq0,& &\rho b\neq0,& &c\psi = 0,&
		&\rho \psi a=0,& &\rho b \f\neq0,& &c\psi\f=0,& &m_3(c, \psi, a) = 0.& 
		\end{align*}
		
		\item[B2:] 
		\begin{align*}
		&\psi a=0,& &b\f\neq0,& &\rho b\neq0,& &c\psi = 0,&
		&\rho \psi a=0,& &\rho b \f\neq0,& &c\psi\f=0,& &m_3(c, \psi, a) = \rho b\f.& 
		\end{align*}
		
		\item[C] 
		\begin{align*}
		&\psi a = 0,& &b\f\neq 0,& &\rho b\neq0,& &c\psi\neq 0,&
		&\rho \psi a=0,& &\rho b \f\neq 0,& &c\psi\f\neq 0.& 
		\end{align*}
		
		\item[G1:]
		\begin{align*}
		&\psi a\neq0,& &b\f=0,& &\rho b\neq 0,& &c\psi = 0,&
		&\rho \psi a\neq 0,& &\rho b \f=0,& &c\psi\f=0,& &m_3(c, b, \f) = 0.&
		\end{align*}
		
		\item[G2:] 
		\begin{align*}
		&\psi a\neq0,& &b\f=0,& &\rho b\neq 0,& &c\psi=0,&
		&\rho \psi a\neq 0,& &\rho b \f=0,& &c\psi\f=0,& &m_3(c, b, \f) = \rho \psi a.&
		\end{align*}
		
		\item[H1:] 
		\begin{align*}
		&\psi a\neq 0,& &b\f=0,& &\rho b\neq0,& &c\psi\neq 0,&
		&\rho \psi a\neq 0,& &\rho b \f= 0,& &c\psi\f=0,& &m_3(c, b, \f) =0.& 
		\end{align*}
		
		\item[H2:] 
		\begin{align*}
		&\psi a\neq 0,& &b\f=0,& &\rho b\neq0,& &c\psi\neq 0,&
		&\rho \psi a\neq 0,& &\rho b \f=0,& &c\psi\f=0,& &m_3(c, b, \f) =\rho \psi a.& 
		\end{align*}
		
		\item[I1:] 
		\begin{align*}
		&\psi a\neq0,& &b\f\neq 0,& &\rho b=0,& &c\psi\neq0,&
		&\rho \psi a=0,& &\rho b \f=0,& &c\psi\f\neq0,& &m_3(\rho, b, a) = 0.&
		\end{align*}
		
		\item[I2:] 
		\begin{align*}
		&\psi a\neq0,& &b\f\neq 0,& &\rho b=0,& &c\psi\neq0,&
		&\rho \psi a=0,& &\rho b \f=0,& &c\psi\f\neq0,& &m_3(\rho, b, a) = c\psi \f.&
		\end{align*}
		
		\item[J:] 
		\begin{align*}
		&\psi a\neq0,& &b\f\neq 0,& &\rho b\neq 0,& &c\psi=0,&
		&\rho \psi a\neq0,& &\rho b \f\neq 0,& &c\psi\f=0.& 
		\end{align*}
		
		\item[K:] 
		\begin{align*}
		&\psi a\neq0,& &b\f\neq 0,& &\rho b\neq 0,& &c\psi\neq0,&
		&\rho \psi a\neq0,& &\rho b \f\neq 0,& &c\psi\f\neq0.& 
		\end{align*}
	\end{itemize}
We list the corresponding algebras to show that they can be equipped with a duality preserving simple modules:
\begin{itemize}
	\item[A1:] Quiver
	\begin{center}
		\includegraphics{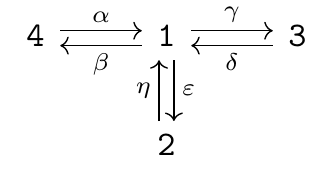}
	\end{center}
	with relations
	\begin{align*}
	&\beta \alpha=0,& &\gamma\alpha=0,& &\varepsilon\eta=0,& &\gamma\delta =0,& &\beta \delta =0,& &\beta \eta \varepsilon\alpha =0,& &\gamma \eta \varepsilon \delta =0,& &\beta \eta \varepsilon \delta \gamma \eta \varepsilon \alpha = 0.&
	\end{align*}
	\item[A2:] Quiver
	\begin{center}
		\includegraphics{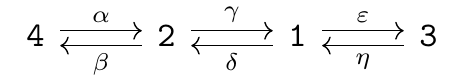}
	\end{center}	
	with relations
	\begin{align*}
	&\beta \alpha =0,& &\varepsilon\eta =0,& &\delta \gamma = \alpha \beta,& &\varepsilon\gamma \delta \eta =0,& &\beta \delta \eta \varepsilon \gamma \alpha =0.&
	\end{align*}
	\item[B1:] Quiver
	\begin{center}
		\includegraphics{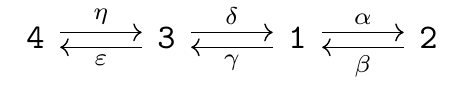}
	\end{center}	
	with relations
	\begin{align*}
	&\alpha \beta =0,&&\varepsilon\eta =0,& &\gamma \delta=0,& &\gamma \beta \alpha \delta =0.&
	\end{align*}
	\item[B2:] Quiver
	\begin{center}
		\includegraphics{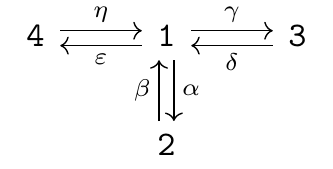}
	\end{center}
	with relations
	\begin{align*}
	&\varepsilon \beta =0,& &\gamma \delta =0,& &\alpha \beta =0,& &\varepsilon\eta =0,& &\alpha \eta =0,& &\gamma \beta \alpha \delta =0,& &\varepsilon\delta \gamma \eta =0.&
	\end{align*}
	\item[C:] Quiver
	\begin{center}
		\includegraphics{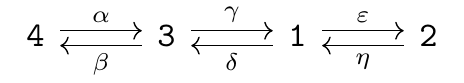}
	\end{center}	
	with relations
	\begin{align*}
	&\beta \alpha=0,& &\delta\gamma =0,& &\varepsilon\eta =0,& &\alpha \beta = \delta \eta \varepsilon\gamma.&
	\end{align*}
	\item[G1:] Quiver
	\begin{center}
		\includegraphics{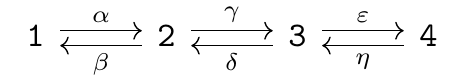}
	\end{center}
	with relations
	\begin{align*}
	&\alpha\beta =0,& &\gamma \delta =0,& &\varepsilon\eta =0.&
	\end{align*}	
	\item[G2:] Quiver
	\begin{center}
		\includegraphics{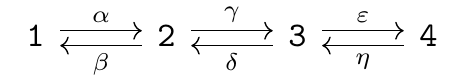}
	\end{center}	
	with relations
	\begin{align*}
	&\alpha\beta =\delta\eta \varepsilon\gamma,& &\gamma \delta =0,& &\varepsilon\eta =0.&
	\end{align*}	
	\item[H1:] Quiver
		\begin{center}
			\includegraphics{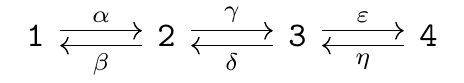}
		\end{center}
	with relations
	\begin{align*}
	&\alpha\beta =0,& &\gamma \delta =\eta\varepsilon,& &\varepsilon\eta =0.&
	\end{align*}	
	\item[H2:] Quiver
	\begin{center}
		\includegraphics{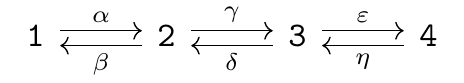}
	\end{center}
	with relations
	\begin{align*}
	&\alpha\beta =\delta \eta\varepsilon\gamma,& &\gamma \delta =\eta\varepsilon,& &\varepsilon\eta =0.&
	\end{align*}
	\item[I1:] Quiver
	\begin{center}
		\includegraphics{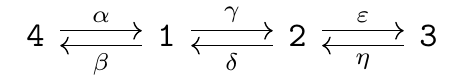}
	\end{center}
	with relations
	\begin{align*}
	&\varepsilon\eta=0,& &\beta\alpha=0,& &\gamma \delta= \eta \varepsilon,& &\beta \delta \gamma \alpha=0,& &\beta \delta \eta \varepsilon\gamma \alpha=0.&
	\end{align*}		
	\item[I2:] Quiver
	\begin{center}
		\includegraphics{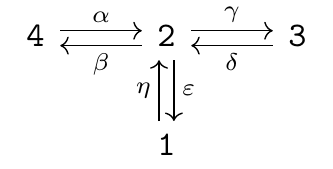}
	\end{center}	
	with relations
	\begin{align*}
	&\beta \alpha =0,& &\gamma \delta =0,& &\beta \delta \gamma \alpha=0,& &\eta \varepsilon= \delta \gamma + \alpha \beta.&
	\end{align*}
	\item[J:] Quiver
	\begin{center}
		\includegraphics{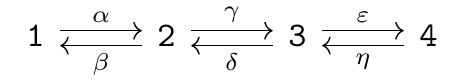}
	\end{center}
	with relations
	\begin{align*}
	&\alpha \beta= \delta \gamma,& &\gamma \delta =0,& &\varepsilon\eta =0.&
	\end{align*}		
	\item[K:] Quiver
	\begin{center}
		\includegraphics{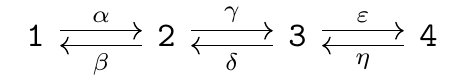}
	\end{center}	
	with relations
	\begin{align*}
	&\alpha \beta= \delta \gamma,& &\gamma \delta =\eta\varepsilon,& &\varepsilon\eta =0.&
	\end{align*}		
\end{itemize} 	
The algebras marked B1, B2 and K are Ringel self-dual. The algebra A1 is Ringel dual to G1, A2 is Ringel dual to G2, the algebra C is Ringel dual to J, the algebra H1 is Ringel dual to I1, and, finally, the algebra H2 is Ringel dual to I2.

In the remainder, we comment on the connection to the geometry of surfaces. For an introduction to the topic, see e.g. \cite[Chapter V]{Har77}.
Let $f\colon X\to Y$ be a birational morphism of smooth surfaces. It can be (non-uniquely) decomposed into a sequence of blow-ups of smooth points, see e.g. \cite[Corollary V.5.4]{Har77}. If for simplicity we assume that $f$ is an isomorphism on a complement to a closed point $y \in Y$ then the exceptional divisor $C$ of $f$, i.e. the curve $C \subset X$ contracted by $f$ to this point $y$, is a tree of rational curves. In other words, the irreducible components $C_\mathtt{i}$ of $C$ are smooth and isomorphic to $\mathbb{P}^1$. At every point at most two components meet and their intersection number is one, i.e.  $C$ is a divisor with normal crossings. Finally, the intersection graph, i.e. the graph whose vertices correspond to components of $C$ and whose edges to the intersection points of those, is a tree. The decomposition of $f$ into a sequence of blow-ups, $f = g_n \circ \ldots \circ g_1$, determines the self-intersection numbers of components. More precisely, $C_{\mathtt{i}}^2 =-1$ if $C_{\mathtt{i}}$ is the exceptional divisor of $g_1$. If, on the other hand, a component $C_{\mathtt{i}}$ is a strict transform of a component $C'_{\mathtt{i}}$ of the exceptional divisor of $g_{n} \circ \ldots \circ g_2$ (i.e. $C_\mathtt{i}$ is the closure of $C_\mathtt{i}\cap U\cong C'_\mathtt{i}\cap U$ in $X$ for the open set $U\subset X$ on which $g_1$ is an isomorphism) then $C_{\mathtt{i}}^2 = {C'_{\mathtt{i}}}^2$ if $C_{\mathtt{i}}\subset U$ and $C_{\mathtt{i}}^2 = {C'_{\mathtt{i}}}^2 -1$ otherwise. In the opposite direction, the intersection form on components of $C$ yields a decomposition of $f$ into a sequence of blow-ups of smooth points. Namely, any component of self-intersection $-1$ can be contracted by the first smooth contraction $g_1$.

Let now $f\colon X \to Y$ be a birational morphism which can be decomposed into 4 blow-ups of smooth points. 
Then the category 
$$
\mathscr{A}_f:=\{E\in \Coh(X)\,|\,Rf_*E =0\}
$$
is equivalent to the category of modules over a quasi-hereditary algebra $\Lambda_f$, see \cite{BodBon15, BodBon17}. If the decomposition of $f$ into blow-ups is unique, i.e. if the associated partial order on simple $\Lambda_f$-modules is a total order, then the algebra $\Lambda_f$ is Morita equivalent to one of the algebras A2, C, I1, I2 or K, see \cite{BodBon17}. 

Further properties are required to homologically characterise the curve-like quasi-hereditary algebras coming from geometry. 
One such property is that $\Ext^2(L(\mathtt{i}),L(\mathtt{l}))=0$ for $\mathtt{i}\neq \mathtt{l}$:

Simple objects in the category $\mathscr{A}_f$, i.e. simple modules over $\Lambda_f$, are $\mathcal{O}_{C_\mathtt{i}}(-1)$, \cite{BodBon15}. If $C_{\mathtt{i}} \cap C_\mathtt{l} = \emptyset$, the support $C_\mathtt{i}$ of  $\mathcal{O}_{C_\mathtt{i}}(-1)$ is disjoint from the support $C_{\mathtt{l}}$  of $\mathcal{O}_{C_\mathtt{l}}(-1)$, hence $\Ext^2( \mathcal{O}_{C_\mathtt{i}}(-1),  \mathcal{O}_{C_\mathtt{l}}(-1)) =0$. If, on the other hand, $C_{\mathtt{i}} \cap C_\mathtt{l} \neq \emptyset$ then $ C_{\mathtt{i}} . C_\mathtt{l} = 1$. In particular $\mathcal{O}_{C_\mathtt{l}}(-1) \cong \mathcal{O}_{C_\mathtt{l}}(-C_\mathtt{i})$ and $\mathcal{O}_{C_\mathtt{i}}(-1) \cong \mathcal{O}_{C_\mathtt{i}}(-C_\mathtt{l})$. In the long exact sequence obtained by applying $\Hom(-, \mathcal{O}_{C_\mathtt{l}}(-C_\mathtt{i}))$ to sequence
$$
0 \to \mathcal{O}_X(-C_\mathtt{l}- C_\mathtt{i}) \to \mathcal{O}_X(-C_\mathtt{l}) \to \mathcal{O}_{C_\mathtt{i}}(-C_\mathtt{l}) \to 0
$$
we have isomorphisms 
\begin{align*} 
&\Ext^j_X(\mathcal{O}_X(-C_\mathtt{l}-C_\mathtt{i}), \mathcal{O}_{C_\mathtt{l}}(-C_\mathtt{i})) \cong H^j(\mathbb{P}^1, \mathcal{O}_{\mathbb{P}^1}(C_\mathtt{l}^2) ),&\\ & \Ext^j_X(\mathcal{O}_X(-C_\mathtt{l}), \mathcal{O}_{C_\mathtt{l}}(-C_\mathtt{i})) \cong H^j(\mathbb{P}^1, \mathcal{O}_{\mathbb{P}^1}(C_\mathtt{l}^2-C_\mathtt{l}.C_\mathtt{i})  ).&
\end{align*} 
Since $H^2(\mathbb{P}^1, \mathcal{O}_{\mathbb{P}^1}(C_\mathtt{l}^2-C_\mathtt{l}.C_\mathtt{i}))=0$ and  the map $ H^1(\mathbb{P}^1, \mathcal{O}_{\mathbb{P}^1}(C_\mathtt{l}^2-C_\mathtt{l}.C_\mathtt{i})  ) \to  H^1(\mathbb{P}^1, \mathcal{O}_{\mathbb{P}^1}(C_\mathtt{l}^2) )$ is surjective, the space $\Ext^2_X( \mathcal{O}_{C_\mathtt{i}}(-1),  \mathcal{O}_{C_\mathtt{l}}(-1))$ is zero. Explicit calculations of the Ringel dual of an arbitrary $\Lambda_f$ in \cite{BodBon17} show that the vanishing of $\Ext^2$ between distinct simple modules also holds for the Ringel duals of ``geometric'' curve-like algebras. 

We note that the algebras A1 and B2 do not satisfy the above additional condition, hence there is no curve attached to them. There are no non-zero elements of $\Ext^2$ between distinct simple modules over the algebra B1, while it is not of geometric origin.

In the five geometric cases one can read off from the quiver of the algebra $\Lambda_f$ the intersection graph of the curve contracted by $f$. Namely, the quiver of $\Lambda_f$ is the double quiver of the intersection graph of the curve $C = \bigcup_{\mathtt{i}=\mathtt{1}}^\mathtt{4} C_{\mathtt{i}}$, i.e. $C_\mathtt{i} \cap C_\mathtt{l} =1$ if and only if there is an arrow $\mathtt{i} \to \mathtt{l}$ in the quiver of $\Lambda_f$. One can read the self-intersection $C_{\mathtt{i}}^2$ from the relations in the algebra. More precisely, the long exact sequence obtained by applying $\Hom_X(-, \mathcal{O}_{C_\mathtt{i}})$ to the short exact sequence
$$
0 \to \mathcal{O}_X(-C_\mathtt{i}) \to \mathcal{O}_X \to \mathcal{O}_{C_\mathtt{i}} \to 0%
$$
gives an isomorphism $\Ext^1_X(\mathcal{O}_{C_\mathtt{i}}(-C_\mathtt{i}), \mathcal{O}_{C_\mathtt{i}}) \xrightarrow{\cong} \Ext^2(\mathcal{O}_{C_\mathtt{i}}, \mathcal{O}_{C_\mathtt{i}})$. Since the latter space is isomorphic to $\Ext^2_X(\mathcal{O}_{C_\mathtt{i}}(-1), \mathcal{O}_{C_\mathtt{i}}(-1))$ and $\Ext^1_X(\mathcal{O}_{C_\mathtt{i}}(-C_\mathtt{i}), \mathcal{O}_{C_\mathtt{i}}) \cong H^1(X, \mathcal{O}_{C_\mathtt{i}}(C_\mathtt{i}))\cong H^1(\mathbb{P}^1, \mathcal{O}_{\mathbb{P}^1} (C_\mathtt{i}^2))$, the number $\dim \Ext^2_X(\mathcal{O}_{C_\mathtt{i}}(-1),\mathcal{O}_{C_\mathtt{i}}(-1))$ of relations at the given vertex $\mathtt{i}$ equals $ h^1(\mathbb{P}^1, \mathcal{O}_{\mathbb{P}^1}(C_\mathtt{i}^2)) = -C_\mathtt{i}^2 - 1$. 

In an analogous manner one can assign to 
the algebra B1 an isomorphism class of a tree of rational curves together with an intersection matrix of the components. It is a curve $C$ with components $C_\mathtt{1}, \ldots, C_\mathtt{4}$ with intersection matrix
$$
\left(\begin{array}{cccc}-1 & 1 & 1 & 0\\
1&-2&0 & 0 \\
1&0&-3&1\\
0 & 0&1&-2 \end{array} \right)
$$
The curve $C$ is isomorphic to the curve in the geometric example labelled by C.

\bibliographystyle{alpha}
\bibliography{publication}

\end{document}